\documentclass[10pt]{article}

\usepackage[numbered]{bookmark}
\usepackage{authblk}
\usepackage{changepage} 
\usepackage[T1]{fontenc}
\usepackage[utf8]{inputenc}
\usepackage{graphics, graphicx}
\usepackage{pgf,tikz,pgfplots}
\usetikzlibrary{arrows.meta}
\usepackage{amsmath, amsfonts, amsthm, amssymb, mathrsfs, mathtools, mathabx}
\usepackage{verbatim}
\usepackage{enumerate}

\usepackage{xcolor}
\usepackage{cite}
\usepackage{wrapfig}
\usepackage{subcaption}
\usepackage[font=small]{caption}

\usepackage[a4paper, margin={3cm}]{geometry}

\newtheorem{theorem}{Theorem}[section]
\newtheorem{definition}[theorem]{Definition}
\newtheorem{lemma}[theorem]{Lemma}
\newtheorem{proposition}[theorem]{Proposition}
\newtheorem{corollary}[theorem]{Corollary}

\newtheorem*{remark}{Remark}
\newtheorem{claim}{Claim}

\DeclareMathOperator{\id}{id}

\newcommand{\Z}{\mathbb{Z}} 
\newcommand{\N}{\mathbb{N}}
\newcommand{\Q}{\mathbb{Q}}
\newcommand{\R}{\mathbb{R}}
\newcommand{\C}{\mathbb{C}} 
\newcommand{\T}{\mathbb{T}}

\newcommand{\bigo}[1]{\mathcal{O}\left(#1\right)}

\newcommand{\re}{\text{Re}\,}
\newcommand{\im}{\text{Im}\,}

\newcommand{\ZZ}{\mathbb{Z}}
\newcommand{\QQ}{\mathbb{Q}}
\newcommand{\RR}{\mathbb{R}}
\newcommand{\CC}{\mathbb{C}}

\renewcommand{\SS}{\mathbb{S}}

\newcommand{\TT}{\mathbb{T}}

\newcommand{\Domain}[3]{D_{#1, #2, #3}}
\newcommand{\DomainP}[3]{D_{#1, #2, #3}^+}

\date{}
\begin{document}
\title{Rigidity of Mather's $\beta$-function on a KAM set for analytic billiards-like maps and unique quasianalytic continuation}
\author[1]{Corentin Fierobe\thanks{Università degli Studi di Roma Tor Vergata, Via della Ricerca Scientifica 1, Roma 00133, Italy.
    \textit{E-mail}: \href{cpef@gmx.de}{\texttt{cpef@gmx.de}}}\hspace{2mm} 
and Vadim Kaloshin\thanks{Institute of Science and Technology Austria
   Am Campus 1, Klosterneuburg 3400, Austria.
    \textit{E-mail}: \href{vadim.kaloshin@ist.ac.at}{\texttt{vadim.kaloshin@ist.ac.at}}}\hspace{2mm} 
and Frank Trujillo\thanks{Centre de Mathématiques Laurent Schwartz (CMLS), CNRS, École polytechnique, Institut Polytechnique de Paris, Palaiseau, France.
    \textit{E-mail}: \href{frank.trujillo-amezquita@cnrs.fr}{\texttt{frank.trujillo-amezquita@cnrs.fr}}}}

\date{}

\maketitle

\begin{abstract}
    In his seminal paper \cite{Kac}, Kac famously asked whether one can "hear the shape of a drum"—that is, whether the isometry class of a bounded domain $\Omega$ in $\mathbb R^n$ is uniquely determined by the spectrum of its Laplace spectrum. The Laplace spectrum is closely related to the length spectrum of the associated billiard. For a convex bounded planar domain $\Omega$, to each billiard periodic orbit one can associate not only its length, but also its rotation number. The set of pairs of length and rotation number of each periodic orbit is called the marked length spectrum. Using the marked length spectrum of $\Omega$, one can associate with it its minimal action function (also known as Mather’s $\beta$-function), denoted by $\beta_\Omega$.

Via unique quasi-analytic continuation, we prove the following rigidity: knowing that Mather’s $\beta$-functions of two billiards in two analytic planar domains $\Omega_1$ and $\Omega_2$ coincide on a set of positive measure of KAM diophantine numbers implies that they coincide on a set of \underline{all} KAM diophantine numbers. 
In particular, knowing that Mather’s $\beta$-functions of two domains coincide
on a positive measure set of KAM diophantine numbers implies that the Marvizi-Melrose invariants of these domains coincide. 
\end{abstract}

\section{Introduction}

In his seminal paper \cite{Kac}, Kac famously asked whether one can "hear the shape of a drum"—that is, whether the isometry class of a bounded domain in Euclidean space is uniquely determined by the spectrum of its Laplacian subject to Dirichlet or Neumann boundary conditions. This question inspired a lot of research on geometric and spectral rigidity. A central development in this direction is the study of marked length spectrum rigidity for geodesic flows on Riemannian manifolds \cite{Croke, CS, dMM, Otal}.

In the present paper, we investigate an analogue of this rigidity problem for certain exact symplectic twist maps on the cylinder arising from dynamical billiards, which we refer to as \textit{billiard-like maps} (see Definition \ref{def:billiard-like}). 
To any exact symplectic twist map $T$, one can associate its \textit{minimal action function} (also known as \textit{Mather's $\beta$-function}), denoted by $\beta_T$ \cite{Mather90, Siburg}. This function assigns to each rotation number $\rho \in \mathbb{R}$ the average action $\beta_T(\rho)$ of a minimal orbit with rotation vector $\rho$. As shown in \cite{Siburg}, $\beta_T$ is a well-defined convex function defined on a subinterval $I \subset \mathbb{R}$, known as the \textit{twist interval} of $T$. In general, $\beta_T$ is not differentiable at rational numbers \cite{Mather90}, but is known to be smooth in a Whitney sense on a positive set of Diophantine rotation numbers \cite{Popov}. As we shall see, for analytic billiard-like maps $\beta_T$ is smooth (see Theorem \ref{thm:mather_beta_function}) in a Whitney sense on a closed set having $0$ as a density point.

In the spirit of Kac's question, we can wonder to what extent the dynamics of $T$ can be reconstructed solely from the knowledge of its minimal action $\beta_T: I \to \mathbb{R}$.

\medskip
Recent developments in Birkhoff billiards and standard maps demonstrate that partial knowledge of the minimal action often determines the map \cite{DKW, FierobeSorrentinoVig, BBS, FKS_ellipse}. Specifically, symmetric, length-spectrum-preserving deformations close to a disk are rigid if they preserve the infinite jet of Mather's $\beta$-function at zero \cite{DKW}; near an ellipse, finite derivatives at a non-zero Diophantine number are also required \cite{FierobeSorrentinoVig} (see also ongoing work by Helfter for standard maps).

The jet of $\beta$ at $0$ is closely tied to the \textit{Marvizi--Melrose invariants} $\{\ell_k\}_{k\ge 0}$ \cite{MM, Siburg}. For smooth, strictly convex domains, the perimeters $L_n$ (supremum) and $l_n$ (infimum) of simple $n$-periodic trajectories satisfy $L_n - l_n = o(n^{-k})$ for all $k \ge 1$, with an asymptotic expansion as $n \to \infty$:
\[
    L_n \sim \ell_0 + \sum_{k=1}^{\infty} \frac{\ell_k}{n^{2k}}
\]
where $\ell_0$ is the perimeter of the domain and $\ell_k$ are constants depending on boundary curvature. Notably, \cite{BuchKalo} showed that non-isometric domains can share the same Marvizi--Melrose invariants, and thus the same infinite jet of $\beta$ at zero.

This motivates a natural question: given an exact-symplectic twist map $T$, is $\beta_T$ uniquely determined by:
\begin{enumerate}
    \item \label{q:positive_meas} its restriction to a set of positive measure?
    \item \label{q:finite} its infinite jet at a finite collection of rotation numbers?
    \item \label{q:one} its infinite jet at a single rotation number?
    \item \label{q:one_zero} its Marvizi-Melrose invariants, when they are defined?
\end{enumerate}
Answers to Questions \ref{q:finite} and \ref{q:one} may depend on whether the rotation numbers are Diophantine \cite{MM, FKS_ellipse, Yunzhe}. As explained, Question \ref{q:one_zero} is equivalent to Question \ref{q:one} in the particular case when the rotation number is zero. For instance, the result in \cite{BuchKalo} suggests a negative answer to Question~\ref{q:one_zero}. Furthermore, \cite{Yunzhe} recently showed that smooth, non-trivial deformations of standard maps can preserve the infinite jet of $\beta$ at a Liouville rotation number, providing evidence against a positive answer to Question~\ref{q:one} in the Liouville case.

\medskip 
In the present paper we give a partial answer to Question \ref{q:positive_meas} for analytic billiard-like maps (see Theorem \ref{thm:quasianalyticity_beta}): 
we show that the restriction of $\beta_T$ to the Diophantine rotation numbers sufficiently close to $0$ is determined by the values of $\beta_T$ on any set of positive measure near $0$. In particular, Marvizi-Melrose invariant are also determined -- see Corollary \ref{cor:mm_answer}.

More precisely, 
given $\gamma, \tau > 0$ we consider the set of Diophantine numbers 
\begin{equation}
    \label{eq:Diophantine_Lazutkin}
D(\gamma, \tau) := \left\{\omega\in\R\,:\, \forall m \in \Z, \forall n\in \Z_{> 0}\quad 
| n\omega-m|\geq \gamma|m|n^{-\tau}\right\},
\end{equation}
introduced in \cite{Lazutkin_book, Lazutkin_KAM} and for which $0$ is a density point, and prove the following: Given two analytic billiard-like maps $T_1, T_2$ there is a threshold $a > 0$, depending only on $\gamma, \tau$ and on the norms of the maps, such that if the $\beta$-functions of the two maps coincide on a subset of $D(\gamma, \tau) \cap (0,a)$ of positive Lebesgue measure then the $\beta$-functions coincide in $D(\gamma, \tau) \cap (0, a)$. 

The proof is related to Mather's beta function's regularity, which is deduced from the structure of so-called KAM curves in analytic billiard-like maps. 
Lazutkin \cite{Lazutkin_KAM} proved that a sufficiently smooth billiard-like map  $T$ possesses invariant curves of any rotation number in a set of the form $D(\gamma, \tau) \cap (0,b)$
where $b>0$ depends on $\gamma$, $\tau$ and the map.
Moreover, these curves depend $\mathscr C^s$-smoothly (in a Whitney sense) on their rotation numbers where $s$ depends on the regularity of the map. 
To show the aforementioned rigidity result on $\beta$, we study stronger regularity properties of the collection of invariant curves of $T$ when the latter is assumed to be \emph{analytic}. As we shall see (Theorem \ref{thm:main_rough_formulation}), although falling short of being analytic, the collection of invariant curves in classical KAM results (initially indexed by sets of the form $D(\gamma, \tau) \cap (0,b)$ for a small $b>0$, see Equation \eqref{eq:Diophantine_Lazutkin}), can be uniquely extended to a larger collection (indexed by an appropriate compact connected subset in the complex plane of the form $K(\gamma, \tau) \cap \{z \in \C \mid |z| \leq b\}$, with $D(\gamma, \tau)\subseteq K(\gamma, \tau)$, see Figure \ref{figure:diamonds}), having the so-called \textit{$\mathscr C^{\infty}$-holomorphic regularity} and, as a consequence, displaying strong \emph{quasianaliticity properties} (e.g., the functions are uniquely determined by their restriction to any non-trivial set with positive one-dimensional Hausdorff measure contained in their domain of definition, see \cite{MarmiSauzin_2011}). Similar results were already published on the regularity of KAM curves \cite{CMS} and Mather's $\beta$ function \cite{CMSA} in the case of families of standard maps depending on a complex parameter.

\begin{figure}[h]
\centering
\begin{tikzpicture}[line cap=round,line join=round,x=4cm,y=4cm]
\clip(-1.4,-0.4) rectangle (1.4,0.4);
\fill[black!10] (-1.4,-0.4) rectangle (1.4,0.4);
\fill[line width=2pt,color=white,fill=white,fill opacity=1] (0.35,0) -- (0.5,0.15) -- (0.65,0) -- (0.5,-0.15) -- cycle;
\fill[line width=2pt,color=white,fill=white,fill opacity=1] (0.7,0) -- (1,0.3) -- (1.3,0) -- (1,-0.3) -- cycle;
\fill[line width=2pt,color=white,fill=white,fill opacity=1] (0.32,0) -- (0.24,0.08) -- (0.16,0) -- (0.24,-0.08) -- cycle;
\fill[line width=2pt,color=white,fill=white,fill opacity=1] (0.12,0) -- (0.08,0.04) -- (0.04,0) -- (0.08,-0.04) -- cycle;
\fill[line width=2pt,color=white,fill=white,fill opacity=1] (0.035,0) -- (0.02,0.015) -- (0.005,0) -- (0.02,-0.015) -- cycle;
\fill[line width=2pt,color=white,fill=white,fill opacity=1] (-0.32,0) -- (-0.24,0.08) -- (-0.16,0) -- (-0.24,-0.08) -- cycle;
\fill[line width=2pt,color=white,fill=white,fill opacity=1] (-0.12,0) -- (-0.08,0.04) -- (-0.04,0) -- (-0.08,-0.04) -- cycle;
\fill[line width=2pt,color=white,fill=white,fill opacity=1] (-0.35,0) -- (-0.5,0.15) -- (-0.65,0) -- (-0.5,-0.15) -- cycle;
\fill[line width=2pt,color=white,fill=white,fill opacity=1] (-0.7,0) -- (-1,0.3) -- (-1.3,0) -- (-1,-0.3) -- cycle;
\fill[line width=2pt,color=white,fill=white,fill opacity=1] (-0.035,0) -- (-0.02,0.015) -- (-0.005,0) -- (-0.02,-0.015) -- cycle;
\draw [line width=0.5pt] (-2,0)-- (2,0);
\draw [line width=0.5pt] (0,-1)-- (0,1);
\draw [line width=0.7pt,dotted] (0.35,0)-- (0.5,0.15);
\draw [line width=0.7pt,dotted] (0.5,0.15)-- (0.65,0);
\draw [line width=0.7pt,dotted] (0.65,0)-- (0.5,-0.15);
\draw [line width=0.7pt,dotted] (0.5,-0.15)-- (0.35,0);
\draw [line width=0.7pt,dotted] (0.7,0)-- (1,0.3);
\draw [line width=0.7pt,dotted] (1,0.3)-- (1.3,0);
\draw [line width=0.7pt,dotted] (1.3,0)-- (1,-0.3);
\draw [line width=0.7pt,dotted] (1,-0.3)-- (0.7,0);
\draw [line width=0.7pt,dotted] (0.32,0)-- (0.24,0.08);
\draw [line width=0.7pt,dotted] (0.24,0.08)-- (0.16,0);
\draw [line width=0.7pt,dotted] (0.16,0)-- (0.24,-0.08);
\draw [line width=0.7pt,dotted] (0.24,-0.08)-- (0.32,0);
\draw [line width=0.7pt,dotted] (0.12,0)-- (0.08,0.04);
\draw [line width=0.7pt,dotted] (0.08,0.04)-- (0.04,0);
\draw [line width=0.7pt,dotted] (0.04,0)-- (0.08,-0.04);
\draw [line width=0.7pt,dotted] (0.08,-0.04)-- (0.12,0);
\draw [line width=0.7pt,dotted] (0.035,0)-- (0.02,0.015);
\draw [line width=0.7pt,dotted] (0.02,0.015)-- (0.005,0);
\draw [line width=0.7pt,dotted] (0.005,0)-- (0.02,-0.015);
\draw [line width=0.7pt,dotted] (0.02,-0.015)-- (0.035,0);
\draw [line width=0.7pt,dotted] (-0.32,0)-- (-0.24,0.08);
\draw [line width=0.7pt,dotted] (-0.24,0.08)-- (-0.16,0);
\draw [line width=0.7pt,dotted] (-0.16,0)-- (-0.24,-0.08);
\draw [line width=0.7pt,dotted] (-0.24,-0.08)-- (-0.32,0);
\draw [line width=0.7pt,dotted] (-0.12,0)-- (-0.08,0.04);
\draw [line width=0.7pt,dotted] (-0.08,0.04)-- (-0.04,0);
\draw [line width=0.7pt,dotted] (-0.04,0)-- (-0.08,-0.04);
\draw [line width=0.7pt,dotted] (-0.08,-0.04)-- (-0.12,0);
\draw [line width=0.7pt,dotted] (-0.35,0)-- (-0.5,0.15);
\draw [line width=0.7pt,dotted] (-0.5,0.15)-- (-0.65,0);
\draw [line width=0.7pt,dotted] (-0.65,0)-- (-0.5,-0.15);
\draw [line width=0.7pt,dotted] (-0.5,-0.15)-- (-0.35,0);
\draw [line width=0.7pt,dotted] (-0.7,0)-- (-1,0.3);
\draw [line width=0.7pt,dotted] (-1,0.3)-- (-1.3,0);
\draw [line width=0.7pt,dotted] (-1.3,0)-- (-1,-0.3);
\draw [line width=0.7pt,dotted] (-1,-0.3)-- (-0.7,0);
\draw [line width=0.7pt,dotted] (-0.035,0)-- (-0.02,0.015);
\draw [line width=0.7pt,dotted] (-0.02,0.015)-- (-0.005,0);
\draw [line width=0.7pt,dotted] (-0.005,0)-- (-0.02,-0.015);
\draw [line width=0.7pt,dotted] (-0.02,-0.015)-- (-0.035,0);
\begin{scriptsize}
\draw[color=black] (-0.05,-0.1) node {$0$};
\draw[color=black] (0.3,0.3) node {$K(\gamma,\tau)$};
\draw[color=black] (1.3,0.3) node {${\CC}$};
\end{scriptsize}
\end{tikzpicture}
 \caption{The set $K(\gamma,\tau)$ consists of the grey region together with the diamond boundaries (dotted lines). Notice that $D(\gamma, \tau) \subset K(\gamma, \tau)$ and that $K(\gamma, \tau) \cap \{z \in \C \mid |z| \leq b \}$ is connected, for any $b > 0$.}
\label{figure:diamonds}
\end{figure}
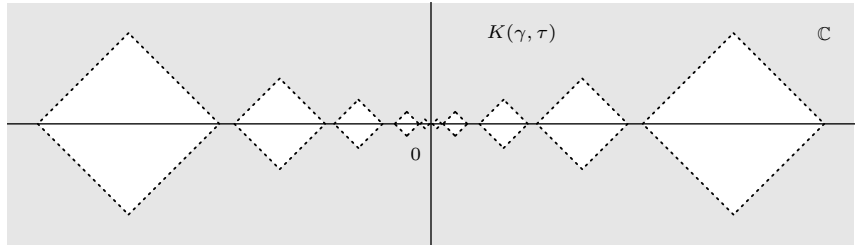

Let us mention that being $\mathscr C^{\infty}$-holomorphic is much more restrictive than being $\mathscr C^\infty$ in the sense of Whitney as this means that the infinitely many derivatives of the map (seen as functions from a compact set to a Banach space of functions) are \emph{complex derivatives in the sense of Whitney}, that is, the function and all its derivatives verify the Cauchy-Riemann equations. Moreover, under suitable conditions on the domain of definition, $\mathscr C^{\infty}$-holomorphic functions share some of the properties of holomorphic functions  (e.g., Cauchy's theorem or Cauchy's integral formula, see \cite[Remark 2.1]{MarmiSauzin_2001}).

\subsection*{Acknowledgmenents}
The authors are grateful to Mathieu Helfter for interesting discussions.
CF and VK acknowledge the support of the ERC Advanced Grant SPERIG (\#885707). CF acknowledges the support of the Italian Ministry of University and Research’s PRIN 2022 grant “Stability in Hamiltonian dynamics and beyond”, as well as the Department of Excellence grant MatMod\@TOV (2023-27) awarded to the Department of Mathematics of University of Rome Tor Vergata.

\section{Preliminaries}

\subsection{Billiards and billiard-like maps}

A strongly convex domain $\Omega\subset\R^2$ with smooth boundary induces four billiard models: two models, namely the \textit{Birkhoff billiards} and the \textit{symplectic billiards}, have their dynamics located inside the domain, and and the two other models, called the \textit{outer billiards} and the outer length billiards, have their dynamics located outside the domain. See Appendix \ref{sec:four_billiards} for more details on these billiard models.

Each of these four models of billiards can be described by a diffeomorphism $T:M\to M$ of an homotopically non-trivial set $M$ of the cylinder $\SS^1\times\R$, \textit{the billiard map}. The latter is $\mathscr C^r$-smooth if $\partial\Omega$ is a $\mathscr C^{r+1}$-smooth embedded manifold of $\R^2$ for a given $r\in\Z_{\geq 0}\cup\{+\infty,\omega\}$, and admits an invariant curve filled with fixed points looping once around the cylinder. 

The billiard map $T$ is known to be an exact-symplectic twist map, see \cite{FKS_lecturenotes} for a precise definition. In particular, there exist an exact area form $\omega=d\lambda$ on $M$ and a smooth map $S:V\subset\R^2\to R$ defined on an open subset $V$ of $\R^2$ such that $(x_0+1,x_1+1)\in V$ and $S(x_0+1,x_1+1)=S(x_0,x_1)$ for any $(x_0,x_1)\in V$ and 
\[
T^\ast\lambda-\lambda = dS.
\]
This property together with the existence of the invariant cuvre filled with fixed points play a central role in the various proving the existence of KAM curves for these billiard models. 

As it is recalled in Appendix \ref{sec:four_billiards}, it was shown that each billiard map arising from one of these four models are what we call a \textit{billiard-like map}, \textit{defined} as follows:

\begin{definition}
\label{def:billiard-like}
A \emph{billiard-like map} is a diffeomorphism onto its image $$T: \T \times (-b, b) \to \T \times \R,$$
for some $b > 0$ and $m \in \N_{\geq 2}$, of the form
\begin{equation}
 \label{equation:expression_billiard_like_maps}
T(x, y) = \big(x+y+\bigo{y^m},y+\bigo{y^{m+1}}\big),
\end{equation}
 for which there exists a 1-form $\lambda$ of the type
\begin{equation}
    \label{equation:expression_one_form}
\lambda = (y^2 + \mu(x, y))dx + \nu(x, y)dy, \qquad \mu(x, y) = \bigo{y^3}, \quad \nu(x, y)= \bigo{y^2},
\end{equation}
such that $T^*\lambda - \lambda $ is well-defined and exact on $\T \times (-b, b)$. 

The integer $m$ is called the \emph{order} of the billiard-like map $T$.
\end{definition}

In particular, it follows from the definition that a billiard like-map as above preserves a 2-form of the type $(y + \bigo{y^2})dy \wedge dx$. Moreover, there exists $0  <\overline{b} \leq b$ such that the restriction of $T$ to $\T \times (-\overline{b}, \overline{b})$ satisfies the intersection property.


Notice that Definition \ref{def:billiard-like} extends naturally to real analytic maps of the form \eqref{equation:expression_billiard_like_maps} defined on domains of the form $\T_r \times U$, for some $r > 0$ and some open set $0 \in U \subseteq \C$, and taking values in $\C/\Z \times \C$. In this case, we assume that the associated 1-form is also real analytic. We will refer to these transformations as \emph{analytic billiard-like maps}.

\subsection{$\mathscr{C}^1$, $\mathscr{C}^r$ and $\mathscr{C}^\infty$-holomorphic maps}
\label{sc:holomorphic}

Let $K\subseteq \CC$ be a compact set without isolated points, $(B, |\cdot|)$ be a Banach space and two maps $f,g:K\to B$. Consider the functions $\delta_f, \Omega_{f,g}: K \times K \to B$ defined, for any $z,z'\in K$, by
$$\delta_f(z,z') = f(z')-f(z),$$
and 
$$\Omega_{f,g}(z,z') = \left\{ \begin{array}{ccl} \frac{f(z')-f(z)}{z'-z}-g(z) & \text{if} & z \neq z', \\ 0 & \text{if} & z = z' . \end{array}\right.$$

\begin{definition}
A map $f:K\to B$ is said to be \emph{$\mathscr C^1$-holomorphic} if it is continuous and there exist a continuous map $g:K\to B$, such that $\Omega_{f,g}$ has a continuous extension to $K\times K$. Denote by $\mathscr C^1_{\text{hol}}(K,B)$ the set of $\mathscr C^1$-holomorphic maps from $K$ to $B$.
\end{definition}

Since $K$ has no isolated points, there is at most one such $g$ and we denote it by $f'$. Notice that if $f$ is analytic on an open neighbourhood of $K$ then this map coincides with the restriction to $K$ of the usual complex derivative of $f$. We then define recursively the notions of $\mathscr C^r$-holomorphic and $\mathscr C^\infty$-holomorphic maps.

\begin{definition}
Given an integer $r\geq 2$, a map $f:K\to B$ is said to be \emph{$\mathscr C^r$-holomorphic} if it is $\mathscr C^1$-holomorphic and its derivative $f'$ is $\mathscr C^{r-1}$-holomorphic. The map $f$ is said to be $\mathscr C^\infty$-holomorphic if it is $\mathscr C^r$-holomorphic for all integers $r>0$.
\end{definition}

We define a Banach norm on the space $\mathscr C^1_{\text{hol}}(K,B)$, denoted by $\|\cdot\|_{\mathscr C^1_{\text{hol}}(K,B)}$, or $\|\cdot\|_{\mathscr C^1_{\text{hol}}(K)}$ when there is no ambiguity on the space $B$, by setting
\[
\|f\|_{\mathscr C^1_{\text{hol}}(K,B)} = n_0(f)+n_1(f)+n_2(f),
\quad f\in \mathscr C^1_{\text{hol}}(K,B)\]
where
\[n_0(f) := |f|_K, \quad n_1(f) := \max\{ |f'|_K, |\delta_f|_{K \times K}\}, \quad n_2(f) := |\Omega_{f,f'}|_{K\times K}\]
and $|h|_X\in\R^+$ denotes the sup-norm of a given continuous function $h:X\to B$ defined on a topological space $X$.

Given $r>0$, if we denote by $f^{(k)}$ the successive derivatives of a function $f\in  \mathscr C^r_{\text{hol}}(K,B)$, we can define its $\mathscr C^r$-holmorphic norm $\|\cdot\|_{\mathscr C^r_{\text{hol}}(K,B)}$, or $\|\cdot\|_{\mathscr C^r_{\text{hol}}(K)}$, by 
\[
\|f\|_{\mathscr C^r_{\text{hol}}(K,B)} = \max_{k=1,\ldots,r} \|f^{(k)}\|_{\mathscr C^1_{\text{hol}}(K,B)}
\]
which endow $\mathscr C^r_{\text{hol}}(K,B)$ with the structure of a Banach space.

Throughout the text we will often consider the $\mathscr C^1$-holomorphic norm of holomorphic functions defined on open sets. Notice that the definitions of the norms above extend immediately to this context.

\subsection{Quasianaliticity properties}
\label{sc:quasianalytic}
\textbf{Quasianalyticity} is a uniqueness property satisfied by certain classes of functions, though several non-equivalent definitions exist in the literature. In this subsection, we adopt the notion of quasianalyticity presented in~\cite[Definition $1$]{MarmiSauzin_2011} and extend their Theorem~A.
To formulate this definition, let $\mathscr{H}^1$ denote the one-dimensional Hausdorff measure on $\mathbb{C} \simeq \mathbb{R}^2$.

\begin{definition}
    Let $B$ be a complex Banach space, a subset $K\subset\CC$, and $E$ be a linear space of maps $K\to B$. The space $E$ is said to be \textit{$\mathscr H^1$-quasianalytic} if for any subset $K'\subset K$ of positive one-dimensional Hausdorff measure, the only map of $E$ vanishing identically on $K'$ is the $0$ function.
\end{definition}

As an example, if $K$ is an open set, then any subspace of the set of analytic maps on $U$ is quasianalytic.

\begin{proposition}
\label{proposition:simply_connected_quasianalytic}
    If $U\subset \CC$ is a bounded simply connected open set such that $\partial U$ is a Jordan curve. Write $K:=\overline U$. Then the space $\mathscr C^1_{\text{hol}}(K,B)$ is $\mathscr H^1$-quasianalytic.
\end{proposition}

\begin{proof}[Proof of Proposition \ref{proposition:simply_connected_quasianalytic}]
The argument closely mirrors the proof of Theorem~A in \cite{MarmiSauzin_2011}. We show that our setting satisfies the same hypotheses, after which the remaining steps follow directly from \cite{MarmiSauzin_2011}. 

Let $f \in \mathscr{C}^1_{\text{hol}}(K, B)$ and let $K' \subset K$ such that $f|_{K'} = 0$. Without loss of generality, by considering compositions of $f$ with continuous linear functionals $\ell \in B^*$, we may assume that $B = \mathbb{C}$. Carathéodory's theorem on conformal mappings yields a biholomorphic map
\[
\varphi \colon B(0,1) \to U,
\]
which extends continuously to a map $\varphi \colon \overline{B}(0,1) \to K$. Consequently, the composition $f \circ \varphi$ is holomorphic on $B(0,1)$, continuous on $\overline{B}(0,1)$, and vanishes on $\gamma = \varphi^{-1}(K')$. The result then follows from the proof in \cite{MarmiSauzin_2011}.

\end{proof}

\section{Main results}

\begin{figure}
    \centering
\usetikzlibrary{arrows}
\begin{tikzpicture}[line cap=round,line join=round,>=triangle 45,x=1.5cm,y=1.5cm]
\clip(-2.5,-2) rectangle (2.5,2);
\fill[black!10] (-2.5,-2) rectangle (2.5,2);
\fill[line width=0pt,fill=white] (0,0) -- (-2.5,-2.5) -- (-2.5,2.5) -- cycle;
\fill[line width=0pt,fill=white] (0,0) -- (2.5,-2.5) -- (2.5,2.5) -- cycle;
\draw [line width=0.7pt] (-2.5,0)-- (2.5,0);
\draw [line width=0.7pt,dotted] (-3,-3)-- (3,3);
\draw [line width=0.7pt,dotted] (-3,3)-- (3,-3);
\begin{scriptsize}
\draw[color=black] (2.3,0.1) node {$\R$};
\draw[color=black] (0,0.2) node {$\omega$};
\draw[color=black] (1.2,1.9) node {{\tiny $\im z = \re z-\omega$}};
\draw[color=black] (1.2,-1.9) node {{\tiny $\im z = \omega-\re z$}};
\end{scriptsize}
\end{tikzpicture}
    \caption{The fractal set $\text{Diamonds}(\gamma,\tau)$ is obtained by intersecting iteratively the white regions between the separatrices $\im z = \pm(\re z - \omega)$ for each $\omega\in D(\gamma,\tau)$.}
    \label{fig:diamonds_precise}
\end{figure}
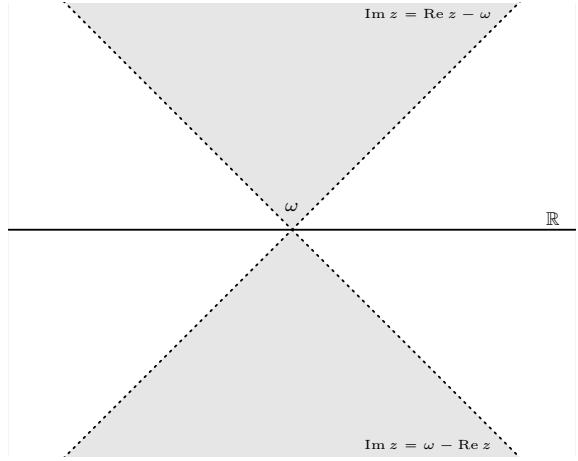

Given $(\gamma,\tau)\in (0,1)\times[1,+\infty)$, we start by considering the set $D(\gamma, \tau)$ already introduced at \eqref{eq:Diophantine_Lazutkin} by
\[
D(\gamma, \tau) = \left\{\omega\in\R\,:\, \forall m \in \Z, \forall n\in \Z_{> 0}\quad 
| n\omega-m|\geq \gamma|m|n^{-\tau}\right\}.
\]
We build on this set an open set $\text{Diamonds}(\gamma,\tau)\subset \C$ called \emph{diamonds}, defined by 
\[
\text{Diamonds}(\gamma,\tau) = 
\{ z\in\C
\,:\,
\forall\omega\in D(\gamma,\tau)\quad |\im z|<|\re z-\omega|
\}.
\]
Figure \ref{fig:diamonds_precise} is an attempt to represent the method to construct $\text{Diamonds}(\gamma,\tau)$. The complement of this set will be denoted by
\begin{equation}
\label{eq:complement_diamonds}
K(\gamma,\tau) = \C\setminus \text{Diamonds}(\gamma,\tau)
\end{equation}
and referred to as the \textit{complement of diamonds}. It is a closed set of $\C$ without isolated points containing the set $D(\gamma,\tau)$.

\vspace{0.2cm}
To ease the reading, we will sometime adopt a different notation for a $\mathscr C^{\infty}$-holomorphic map $\varphi:\mathcal K\to \mathscr C^{\omega}(\T_r,\C^2)$ defined on a closed set $\mathcal K\subset \C$, given $r>0$. The map $\varphi$ can in fact be viewed as a map $\varphi: \T_r\times \mathcal K\to\C^2$ by writing
\[
\varphi(\theta,\omega) = \varphi(\omega)(\theta),
\qquad (\theta,\omega)\in \T_r\times \mathcal K.
\]
Our main result is the following.

\begin{theorem}
\label{thm:main_rough_formulation}
Let $T: \T_r \times B(0, b) \to \C/\Z \times \C$ be an analytic billiard-like map of order $2$. Given $(\gamma,\tau)\in(0,1)\times[1,+\infty)$ there exist $\varrho=\varrho(|T|, \gamma,\tau)>0$ and a $\mathscr{C}^\infty$-holomorphic map 
\[
{\varphi}:
\mathcal K(|T|, \gamma, \tau) \rightarrow \mathscr C^\omega(\T_{r/4}, \C^2), \qquad 
\mathcal K(|T|, \gamma, \tau):= K(\gamma, \tau)\cap \overline{B(0, \varrho)},\]
satisfying 
\[
 \| {\varphi}\|_{\mathscr C^1_{\textup{hol}}(\mathcal K(T, \gamma, \tau))} < \tfrac{1}{2}
\]
such that the transformation $\Phi=\id + {\varphi}: \T_{r/4} \times \mathcal K(|T|, \gamma, \tau) \to \T_{r/2} \times B(0, b)$ given by
 \begin{equation}
 \label{eq:KAM_curve_conjugacy}
  \Phi(\theta, \omega) := (\theta, \omega) + {\varphi}(\theta, \omega)
 \end{equation}
is well-defined, diffeomorphic onto its image, and satisfies
\begin{equation}
\label{eq:KAM_curve_equation}
    T\Phi(\theta, \omega) = \Phi(\theta + \omega, \omega),
\quad \forall (\theta, \omega)\in\T_{r/4}\times \mathcal K(|T|, \gamma, \tau).
\end{equation}

\end{theorem}
\begin{figure}[ht!]
\centering
\begin{tikzpicture}[
    point/.style={circle, fill=black, inner sep=1.8pt},
    arrow/.style={->, >=Stealth, semithick, shorten >=2pt, shorten <=2pt, blue!80!black}
]

    \begin{scope}[local bounding box=leftbox]
        \draw[thick] (0,0) rectangle (3,3);
        
        \draw[thick] (0,1.5) .. controls (1,2.5) and (2,0.5) .. (3,1.5)
        node[pos=0.5, below=14pt, font=\small\itshape] {Curve $\theta\mapsto \Phi(\theta,\omega)$}
            node[pos=0.3, point] (A1) {}
            node[pos=0.7, point] (B1) {};
        
        \draw[arrow] (A1) to[bend left=40] node[above, font=\small] {$T$} (B1);
    \end{scope}

    \begin{scope}[xshift=4.5cm, local bounding box=rightbox]
        \draw[thick] (0,0) rectangle (3,3);
        
        \draw[thick] (0,1.5) -- (3,1.5)
        node[pos=0.5, below=14pt, font=\small\itshape] {Curve $\omega = \text{cst}$}
            node[pos=0.3, point, label=below:{$\theta$}] (A2) {}
            node[pos=0.7, point, label=below:{$\theta+\omega$}] (B2) {};
        
        \draw[arrow] (A2) to[bend left=40] node[above, font=\small] {$\Phi^{-1}T\Phi$} (B2);
    \end{scope}
\end{tikzpicture}

\caption{Illustration in the space of pairs $(\theta,\omega)\in\SS^1\times D(\gamma,\tau)$ of the role of $\Phi$ in Theorem \ref{thm:main_rough_formulation}: given $\omega\in D(\gamma,\tau)$, the set $\{\Phi(\theta,\omega)\,|\,\theta\in\SS^1\}$ defines an invariant curve for $T$ whose dynamics is conjugated via $\Phi$ to a rotation of angle $\omega$.}
\label{fig:mainthm}
\end{figure}
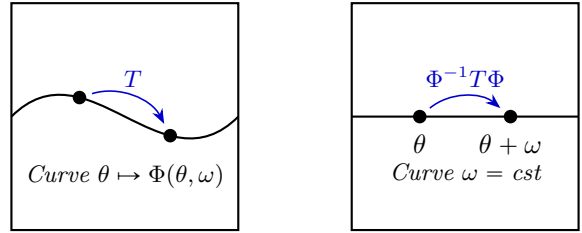

\begin{remark}
    It follows from the proof of the result above, in particular from Equation \eqref{eq:formula_rho}, that the function $\varrho_T = \rho(|T|, \cdot)$ can be taken as a continuous function.
 \end{remark}

We refer to \cite{Lazutkin_KAM} for the seminal work on Birkhoff billiards and \cite{Douady} for improving the regularity significantly in Lazutkin's result as well as developping KAM-type results for the outer billiard model. KAM results for the symplectic billiards are stated in \cite{AlbersTabachnikov} and for the outer length billiards in \cite{BBF}. Let us also cite \cite{Popov} for a deep proof using a so-called interpolating Hamiltonian and \cite{MartinRamirezTamarit} for their use of KAM techniques in exponentially small estimates of billiard maps. Let us finally mention \cite{CMS} who introduced the same type of KAM structure and regularity in KAM invariant curves of analytic perturbation on the integrable standard map.

\vspace{0.2cm}
\noindent \textit{Unicity of KAM curves.}  A map $\varphi$ satisfying the same properties as in Theorem \ref{thm:main_rough_formulation} will be called a \textit{KAM curve} of $T$. Using the quasianalytic properties in the previous section, we can prove the following \emph{KAM rigidity} property for billiard-like maps. 

\begin{corollary}
\label{cor:main_uniqueness_rough}
 Let $T, \tilde T : \T_{r}\times B(0, b) \to \C/\Z \times \C$ be two real analytic billiard-like maps. Using the notations of Theorem \ref{thm:main_rough_formulation}, let ${\varphi}$ and $\tilde{\varphi}$ be KAM curves of $T$, $\tilde T$ defined on sets $K$, $\tilde K \subseteq \C$, respectively.
 
 If the set $\{ \omega \in K \cap \tilde K \mid {\varphi}(\omega) = \tilde {\varphi}(\omega) \}$ has an accumulation point then $T = \tilde T$.
\end{corollary}

\begin{proof}
By \eqref{eq:KAM_curve_equation}, it follows that $T$ and $\tilde T$ coincide on the set $\bigcup_{\omega \in K \cap \tilde K} \varphi(\omega)(\T) \subseteq \T \times \R.$ Since $K \cap \tilde K$ have an accumulation point and $T, \tilde T$ are analytic maps, it is not difficult to show that the two billiard-like maps must coincide. Indeed, if $\omega^* \in K \cap \tilde K$ is an accumulation point of $K \cap \tilde K$, $T$ and $\tilde T$  have the same Taylor expansion at $(\theta^*, \omega^*)$ for any $(\theta^*, \omega^*) \in \varphi (\omega^*)(\T)$ (see, e.g., \cite[Lemma 2.5]{trujillo_uniqueness_2021}). 

\end{proof}

\vspace{0.2cm}
\noindent \textit{Mather's beta function.} Theorem \ref{thm:main_rough_formulation} can be applied to show that the so-called \emph{Mather's beta function} associated to an analytic billiard-like map admits $\mathscr C^\infty$-holomorphic extensions. This function is first defined as follows. Given a billiard like map $T$, 
we can consider a generating map $S:V\subset \R^2$ defined on an open subset $V$ of $\R^2$ such that $(x_0+1,x_1+1)\in V$ and $S(x_0+1,x_1+1)=S(x_0,x_1)$ for any $(x_0,x_1)\in V$ and satisfying $T^\ast\lambda-\lambda=dS$, where $\lambda$ is a one-form as in Definition \ref{def:billiard-like}. We can define its beta function as 
\[
\beta_T(\omega) = \lim_{n\to+\infty}\frac{1}{2n}\sum_{k=-n}^{n-1}S(x_k,x_{k+1})
\]
for any $\omega\in\R$ such that $T$ admits a minimal configuration $(x_k)_{k\in\Z}$ of rotation number $\omega$. A minimal configuration is a sequence $(x_k)_{k\in\Z}$ minimizing the partial sums of the formal sum
\[
\sum_{k\in\Z}S(x_k,x_{k+1})
\]
among finite families $(x'_k)_{u\leq k\leq v}$ with fixed endpoints, namely $x_u'=x_u$ and $x_v'=x_v$. 
Since $T$ has the twist property for small $y$'s, $\beta$ is well-defined and convex on an open interval containing $0$.
We refer to \cite{Gole} and \cite{Siburg} for more details.

\begin{theorem}
\label{thm:mather_beta_function}
 Let $T$ be an analytic billiard-like map. Given $(\gamma,\tau)\in(0,1)\times[1,+\infty)$, denote by $\mathcal K(T, \gamma,\tau)$ the set introduced in Theorem \ref{thm:main_rough_formulation}. There is a $\mathscr C^{\infty}$-holomorphic map
 $\beta^{\ast}_T:\mathcal K(|T|, \gamma,\tau)\to\C$
 such that 
 \[
 \left.{\beta^{\ast}_T}\right|_{D(\gamma,\tau)} = \left.{\beta_T}\right|_{ D(\gamma,\tau)}.
 \]
\end{theorem}

This result is similar as the one which can found in \cite{CMSA} where it is obtained for families of standard maps with almost zero potential. Since the proof of Theorem \ref{thm:mather_beta_function} is the same, we refer the reader to this paper for the details.

\begin{theorem}
    \label{thm:quasianalyticity_beta}
    Let $T_1,T_2$ analytic billiard-like map and $(\gamma,\tau)\in(0,1)\times[1,+\infty)$. Then there exists $\varrho = \varrho(\max\{|T_1|, |T_2|\}, \gamma, \tau) > 0$ such that if $\beta_{T_1}$ and $\beta_{T_2}$ coincide on a positive measure set of $D(\gamma, \tau) \cap (0, \varrho)$ then $\beta_{T_1}$ and $\beta_{T_2}$ coincide on $D(\gamma, \tau) \cap (0, \varrho)$.
\end{theorem}

\begin{proof}
It is a consequence of Theorem \ref{thm:mather_beta_function} together with quasianalyticity properties of $\mathscr C^1$-holomorphic maps, see Proposition \ref{proposition:simply_connected_quasianalytic}.  
\end{proof}

\begin{corollary}
    \label{cor:mm_answer}
    If $\Omega_1$ and $\Omega_2$ are two strictly bounded planar convex domains with analytic boundary such that the corresponding billiard maps satisfy the conditions of Theorem \ref{thm:quasianalyticity_beta}. Then the Marvizi-melrose invariants of $\Omega_1$ and $\Omega_2$ coincide.
\end{corollary}

\section{Proof of the main result}
Theorem \ref{thm:main_rough_formulation} will be a consequence of Theorem \ref{thm:KAM}, which we state below and amounts to a perturbative version of Theorem \ref{thm:main_rough_formulation} for billiard-like maps.

Let us introduce the linear map $A:\C^2\to\C^2$ given by 
\[
A(x, y) = (x+y,y),
\qquad (x,y)\in\C^2.
\]
It projects into a diffeomorphism of $\C/\Z\times \C$. Now given $r,b>0$ and a map $R: \T_r \times B(0, b) \to \C/\Z\times\C$, we denote by $T_R:\T_{r}\times B(0, b) \to \C/\Z\times\C$ the map given by 
\begin{equation}
\label{eq:perturbation_form}
T_R := A+R.
\end{equation}

In the following, we adopt two points of view on analytic maps of the form $\Phi:U\to C$, where $U$ is an open subset of $\T_r\times B(0,b)$ and $C$ is either $\C^k$ for some $k\in\N$, or $\C/\Z\times \C$. The first is the classical point of view; in particular we can associate to $\Phi$ its sup norm $|\Phi|_U\in[0,+\infty]$; the second, given an integer $r>0$, is the $\mathscr C^r$-holomorphic (respectively $\mathscr C^\infty$-holomorphic) point of view, which we now describe. Such a map can be regarded as a $\mathscr C^r$-holomorphic (respectively $\mathscr C^\infty$-holomorphic) map defined on $K=B(0,b)$ and taking values in the Banach space $B=\mathscr C^\omega(\T_r,C)$, by rewriting $\Phi$ as
\[
\Phi:\left\{
\begin{array}{ccc}
    B(0,b) & \to & B=\mathscr C^\omega(\T_r,C) \\
    y & \mapsto &\Phi(\cdot,y). 
\end{array}
\right.
\]

\begin{theorem}
\label{thm:KAM}
There exists a continuous function $\varepsilon: (0,1) \times [1, +\infty) \to (0, 1)$ such that for any $(\gamma,\tau)\in(0,1)\times[1,+\infty)$ the following holds.

Let $r,b\in(0,1)$ and $R: \T_r \times B(0, b) \to \C^2$ be a real analytic map of the form
 \begin{equation}
 \label{eq:form_R}
 R(x, y) = (y^2 \overline R_1(x, y), y^3 \overline R_2(x, y)),
 \end{equation}
 for which $T_R$ (as in \eqref{eq:perturbation_form}) is an analytic billard-like map of order $2$ (in the sense of Definition \ref{def:billiard-like}) satisfying 
 \begin{equation}
 \label{eq:smallness_condition_R}
 \max_{j=1,2}
 \| \overline R_j \|_{\mathscr C^1_{\textup{hol}}(B(0,b))}
 \leq \varepsilon_0(\gamma,\tau,r,b) :=\varepsilon(\gamma,\tau)r^{8\tau+10}
 \min\left\{\left(\frac{r}{\ln r}\right)^{\tau},b\right\}^{10},
 \end{equation}
  with associated 1-form $\lambda$ defined on $\T_{r + 2b} \times B(0, 2b)$ of the form
\begin{equation}
\label{eq:form_form}
\lambda(x, y) = (y^2 + y^2\overline \mu(x, y))dx + y \overline \nu(x, y)dy, 
\end{equation}
satisfying
\begin{equation}
    \label{eq:smallness_condition_form}
    | \overline \mu|_{\T_{r + 2b} \times B(0, 2b)},\, | \overline \nu|_{\T_{r + 2b} \times B(0, 2b)} < \frac{1}{8},
\end{equation}
and 
\begin{equation}
    \label{eq:exact_symplectic}
    T^*\lambda - \lambda = dS,
\end{equation}
for some real-analytic map $S: \T_r \times B(0, b) \to \C$.

Then, there exists a $\mathscr{C}^{\infty}$-holomorphic map
\[
{\varphi}:
\left\{
\begin{array}{ccc}
\mathcal K(\gamma, \tau)
& \to &
\mathscr C^\omega(\T_{r/4}, \C^2)\\
y & \mapsto & (y\overline \varphi_1(\cdot,y),y^2\overline \varphi_2(\cdot,y))
\end{array}\right.
\qquad 
\mathcal K(\gamma, \tau):= K(\gamma, \tau)\cap \overline{B(0,b/2)},
\]
satisfying 
    \begin{equation*}
 \label{eq:bounds_curve_constants}
\max\{
\| \overline \varphi_1\|_{\mathscr C^1_{\textup{hol}}(\mathcal K(\gamma, \tau))},
\|\overline \varphi_2\|_{\mathscr C^1_{\textup{hol}}(\mathcal K(\gamma, \tau))}\}
\leq \tfrac{1}{2},
 \end{equation*}
such that the transformation $\Phi=\id +{\varphi}: \T_{r/2} \times \mathcal K(\gamma, \tau) \to \T_{r/2} \times B(0, b)$ given by
 \[
  \Phi(x,y) := (x+y\overline \varphi_1(x,y),y+y^2\overline \varphi_2(x,y))
 \]
is well-defined, diffeomorphic onto its image, and satisfies
\[
T\Phi(x,y) = \Phi(x+y,y),
\quad \forall (x,y)\in\T_{r/2}\times \mathcal K(\gamma, \tau).
\]
\end{theorem}

Before delving into the proof of Theorem \ref{thm:KAM}, let us show how it implies Theorem \ref{thm:main_rough_formulation}.

\begin{proof}[Proof of Theorem \ref{thm:main_rough_formulation}]
Let $T: \T_r \times B(0, b) \to \C/\Z \times \C$ be an analytic billiard-like map. Notice that any such map can be expressed as $T = T_R$, where $R  : \T_r \times B(0, b) \to \C^2$ is a real analytic function of the form $R(x, y) = (y^2 \overline R_1(x, y), y^3 \overline R_2(x, y))$ and $T_R$ is defined as in \eqref{eq:perturbation_form}. Moreover, since $T$ is a billiard-like map of order $2$ there exists a $1$-form $\lambda$ of the form 
\[ \lambda(x, y) = (y^2 + \mu(x, y))dx + \nu(x, y)dy, \qquad \mu =y^3 \mu^*, \quad \nu = y^2 \nu^*,\]
satisfying \eqref{eq:exact_symplectic}. However, since $(\overline R_1, \overline R_2)$ and $(\overline \nu, \overline \mu) = (y\nu^*, y \mu^*)$ do not necessarily satisfy \eqref{eq:smallness_condition_R} and \eqref{eq:smallness_condition_form}, respectively, we cannot directly apply Theorem \ref{thm:KAM}.

The strategy is to show that, up to an appropriate change of coordinates, we can assume that $\overline R_1$ and $\overline R_2$ are of the form
\[
\overline R_j(x,y) = y^m R_j^\ast(x,y), \qquad j=1,2,
\]
 for a sufficiently large integer $m\geq 0$ (in our case $m=12$), and then to replace $b$ by a sufficiently small value $b_{\ast}>0$, so that Conditions \eqref{eq:smallness_condition_R} and \eqref{eq:smallness_condition_form} are satisfied, when restricted to $B(0, b^\ast)$.

This is enabled by the following result which is a direct corollary of \cite[Theorem 3]{MartinRamirezTamarit}); see also \cite{Lazutkin_book} for the original idea of changing the order of the map. 
\begin{theorem}
    \label{thm:RMS}
    Let $T:\T_r\times B(0,b)\to\C/\Z \times \C$ be an analytic billiard-like map of order $m\geq 2$. For any $\overline m \geq m$ and any $\overline r \in (0,r)$, there exist $\overline b \in(0,b)$ and a real-analytic map 
    \[
    \Phi: \T_{\overline r}\times B(0, \overline b)\to \C/\Z\times \C
    \]
    diffeomorphic onto its image, of the form $\Phi = I+\varphi$
    with 
    \[
    \varphi(x,y) = (y\varphi_1(x,y),y^2\varphi_2(x,y))
    \]
    and such that
    $\Phi^{-1}T\Phi : \T_{\overline r}\times B(0, \overline b)\to \C/\Z\times\C$
    is a well-defined real analytic billiard-like map of order $\overline m$. 
\end{theorem}

\begin{proof}
    The proof is an immediate consequence of \cite[Theorem 3]{MartinRamirezTamarit}). We just need to ensure that $T':=\Phi^{-1}T\Phi$ preserves a one-form of the type \eqref{equation:expression_one_form}. 
    By conjugation, $T'$ preserves the one-form $\lambda':=\Phi^\ast\lambda$ 
    and from the nature of $\Phi$ given by $\Phi=I+\varphi$ whith $\varphi$ given in the statement of Theorem \ref{thm:RMS}, 
    $\lambda'$ is of the type given by \eqref{equation:expression_one_form}.
\end{proof}

By Theorem \ref{thm:RMS}, we may assume that $\overline R_1$, $\overline R_2$, $\lambda$ are defined on $\TT_{\overline{r}}\times B(0, \overline b)$, for $\overline{r}=r/2$ and some $\overline{b}\in(0,b)$ depending only on $T$, that $\overline R_1$, $\overline R_2$ have the form
\[
\overline R_j(x,y) = y^{12} R_j^{\ast}(x,y), \qquad \text{ for any } j=1,2,  \text{ and any } (x,y)\in\TT_{\overline{r}}\times B(0,\overline{b}),
\]
for some $R_1^*, R_2^* : \TT_{\overline{r}}\times B(0,\overline{b}) \to \C$ real analytic, and that $\lambda$ is of the same form. 
By shrinking first $\overline r$ and $\overline b$ we can assume that $\lambda$ is defined on $\TT_{\overline r}\times B(0,\overline b)$.
Moreover by construction $T_R$ takes the form
\[ T_R(x, y) = (x + y + y^2\overline R_1(x, y), y + y^3\overline R_2(x, y)) =(x + y + y^{14}R_1^*(x, y), y + y^{15}R_2^*(x, y)). \]
By Proposition \ref{proposition:norm_power_times_map}, for any $b_{\ast}\in(0,\overline{b})$, we have
\begin{equation}
\label{eq:rescaling_bound}
\|\overline R_j\|_{\mathscr{C}^1_{\textup{hol}}(B(0, b_{\ast}))} \leq 24b_{\ast}^{11} \|R_j^{\ast}\|_{\mathscr{C}^1_{\textup{hol}}(B(0, \overline{b}))}, \qquad j = 1, 2.
\end{equation}
Also, notice that
\begin{equation}
\label{eq:rescaling_bound_form}
| \overline \mu|_{\T_{\overline r}\times B(0,2b_*)},  \leq 2b_*   | \mu^*|_{\T_{\overline r}\times B(0,2\overline b)}, 
\qquad 
| \overline \nu|_{\T_{\overline r}\times B(0,2b_*)} \leq 2b_*| \nu^*|_{\T_{\overline r}\times B(0,2\overline b)}.
\end{equation}

Fix $(\gamma,\tau)\in(0,1)\times[1,+\infty)$ and define 
\[ 
b_0(\gamma, \tau, r, R):= 
\frac{\varepsilon(\gamma,\tau)\overline{r}^{8\tau + 10}}
{24\left(\max_{j=1,2}\| R_j^\ast \|_{\mathscr C^1_{\text{hol}}B(0, \overline{b}))}+1\right)},
\]
where $\varepsilon(\gamma,\tau)$ is given by Theorem \ref{thm:KAM}. Defining
\[b_{\ast} = 
\tfrac{1}{16}
\min\left\{
b_0,\overline{b}, 
\left(\frac{\overline r}{\ln \overline r}\right)^\tau,
( | \mu^*|_{\T_{\overline r}\times B(0,2\overline b)} + | \nu^*|_{\T_{\overline r}\times B(0,2\overline b)})^{-1} 
\right\},\]
by Equations \eqref{eq:rescaling_bound}, \eqref{eq:rescaling_bound_form} and the definition of $b_*$, it follows that
\[
\max_{j=1,2} \|\overline R_j\|_{\mathscr{C}^1_{\textup{hol}}(B(0, b_{\ast}))}\leq 
24b_\ast^{10}b_0\cdot \max_{j=1,2}\| R_j^\ast \|_{\mathscr C^1_{\text{hol}}B(0, \overline{b}))}
\leq 
\varepsilon(\gamma,\tau) \overline{r}^{8\tau + 10}b_*^{10}\leq\varepsilon_0(\gamma,\tau,\overline r,b_\ast)
\]
for $j=1,2$, and
\[
| \overline \mu|_{B(0,2b_*)}, | \overline \nu|_{B(0,2b_*)} \leq \frac{1}{8}.
\]
Therefore, we can apply Theorem \ref{thm:KAM} to $T_R$ restricted to $\T_{\overline{r}} \times B(0, b_*)$, and the result follows with
\begin{equation}
    \label{eq:formula_rho}
    \rho(T, \gamma, \tau) = \frac{1}{2}b_\ast.
\end{equation} 
\end{proof}

The remaining part of this work is devoted to the proof of Theorem \ref{thm:KAM}.


\subsection{Sketch of the Proof of Theorem \ref{thm:KAM}}  
\label{subsec:sketch}

The proof of Theorem \ref{thm:KAM} relies on an iterative KAM scheme. Before formally stating the procedure to be iterated (see Proposition \ref{prop:estimates_one_step}), let us first motivate it by providing a loose descprition of it.
\vspace{0.2cm}

\noindent \textbf{In the following and for the rest of the proof, we fix $(\gamma,\tau)\in (0,1)\times [1,+\infty)$ and we will write $K$ instead of $K(\gamma, \tau)$ to simplify.}
\vspace{0.2cm}

\noindent We introduce the sets
\[K_0:= K, \qquad K_h := \{z \in \CC \mid \textup{d}(z, K) < h \}, \quad \text{ for any } h > 0,\]
\[K_h^b:= K_h \cap B(0, b), \quad \text{ for any } b > 0 \text{ and any } h \geq 0.\]
We refer the reader to Fig. \ref{fig:neighborhood_diamonds} for a representation of $K_h$.
Notice that given $b, h>0$, the sets $K_h^b$ are open subsets of $\C$ which decrease to $K\cap B(0,b_0)$ when $h\to 0$ and $b\to b_0$.

\begin{figure}[h!]
    \centering
\usetikzlibrary{arrows}
\begin{tikzpicture}[line cap=round,line join=round,x=1.5cm,y=1.5cm]
\clip(-2.5,-2) rectangle (2.5,2);
\fill[black!10] (-2.5,-2) rectangle (2.5,2);
\fill[line width=0pt,fill=white] (-0.8,0) -- (0,0.8) -- (0.8,0) -- (0,-0.8) -- cycle;
\fill[line width=0pt,fill=white] (-1.2,0) -- (-2.5,-1.3) -- (-2.5,1.3) -- cycle;
\fill[line width=0pt,fill=white] (1.2,0) -- (2.5,-1.3) -- (2.5,1.3) -- cycle;
\draw [line width=0.5pt] (-2.5,0)-- (2.5,0);
\draw [line width=0.5pt,dotted] (-1,0)-- (0,-1);
\draw [line width=0.5pt,dotted] (-1,0)-- (0,1);
\draw [line width=0.5pt,dotted] (0,1)-- (1,0);
\draw [line width=0.5pt,dotted] (1,0)-- (0,-1);
\draw [line width=0.5pt,dotted] (-2.5,1.5)-- (-1,0);
\draw [line width=0.5pt,dotted] (-2.5,-1.5)-- (-1,0);
\draw [line width=0.5pt,dotted] (2.5,1.5)-- (1,0);
\draw [line width=0.5pt,dotted] (2.5,-1.5)-- (1,0);
\draw [black, -latex] (-0.5,0.5) -- (-0.4,0.4);
\begin{scriptsize}
\draw[color=black] (2.3,0.1) node {$\R$};
\draw[color=black] (-1,0.2) node {$\omega_1$};
\draw[color=black] (1,0.2) node {$\omega_2$};
\draw[color=black] (-0.35,0.55) node {$h$};
\end{scriptsize}
\end{tikzpicture}
    \caption{Given a small $h>0$, $K_h$ is the $h$-neighborhood of $K(\gamma,\tau)$. It can be thought of as the complement (shadowed) of a shrunk version (in white) of the set $\text{Diamonds}(\gamma,\tau)$. The boundary of the original version of the diamonds is represented with dotted lines. It contains $\omega_1$ and $\omega_2$, as well as the whole $D(\gamma,\tau)$ and an open neighborhood of it.}
    \label{fig:neighborhood_diamonds}
\end{figure}
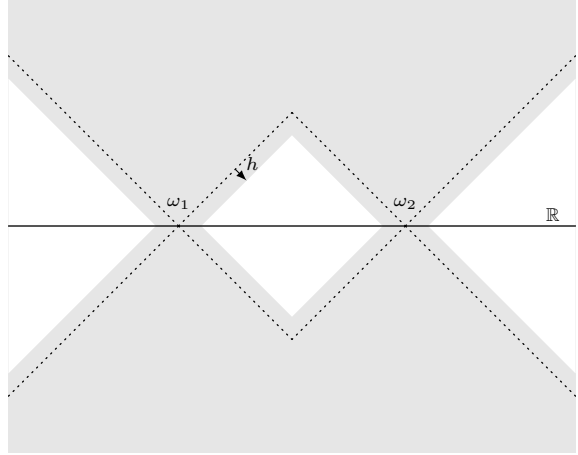

\subsubsection*{KAM step} An arbitrary step of the KAM scheme can be described as follows. Suppose that we are given a billiard-like map 
\[
T := T_R = A + R:\T_{r}\times K^{b}_{h}\to \C/\Z \times \C
\]
as described in Definition \ref{def:billiard-like}, with $R$ (which we'll treat as a small perturbative term) of the form \eqref{eq:form_R}, namely, 
\[  R(x, y) = (R_1(x, y), R_2(x, y)) = (y^2 \overline R_1(x, y), y^3 \overline R_2(x, y)), \]
and $K = K(\gamma ,\tau)$, for some real numbers $b, r, h>0$.

We would like to construct a diffeomorphism (close to the identity on $\T_r \times \C/\Z$) of the form $\Phi=\id +\varphi$, for some $\C^2$-valued real analytic map 
\begin{equation}
\label{eq:varphi_form}
\varphi(x, y) = (\varphi_1(x, y), \varphi_2(x, y)) = (y\overline\varphi_1(x, y), y^2\overline\varphi_2(x, y)),
\end{equation}
such that the map 
\[
T_+ := \Phi^{-1}T\Phi:\T_{r'}\times K_{b'}^{h'}\to \C/\Z \times \C
\]
is well-defined for some $r'\in(0,r)$, $b'\in(0,b)$ and $h'\in(0,h),$ and so that, if we express $T_+ = T_{R_+}$, then the $\mathscr{C}^1$-holomorphic norm of $R_+$ (on the new-domain) is much smaller than the $\mathscr{C}^1$-holomorphic  norm of $R$ (on the initial domain).

Expressing $\Phi^{-1} = \id + \psi$ and assuming that the norm of $\varphi$ is comparable or smaller to (the root of) the norm $R$, it is easy to see that the norm of $\psi$ verifies a similar property (see Proposition \ref{proposition:inverse_estimates}). Moreover, we can express $R_+ = \Phi^{-1}T_R\Phi - A$ explicitly, inspired by the ideas of Martin, Ram\'irez-Ros and Sarol in \cite{MartinRamirezTamarit}, as
\begin{align*}
A + R_+ = T_+ & = (\id + \psi ) (A + R)  (\id + \varphi) \\
& = (A + R)(\id + \varphi) + \psi T\Phi \\
& =A + A \varphi + R\Phi + \psi T \Phi \\
& = A + (\varphi A - \varphi T\Phi) + (R\Phi - R) + (\varphi + \psi)T \Phi + (A\varphi - \varphi A+R).
\end{align*}

General estimates show that (with our current assumptions on the norm of $R, \varphi, \psi$) the norm of the terms $(\varphi A - \varphi T\Phi)$, $(R\Phi - R)$ and $(\varphi + \psi)T \Phi$ will be much smaller than that of $R$. So, if $\varphi$ satisfies an equation of the form
\[ \varphi A - A\varphi = R,\]
this is enough to guarantee that the norm of $R_+$ is much smaller compared to that of $R$. The equation above translates to
\begin{equation*}
 \left\{
 \begin{array}{rcl}
 \varphi_1(x+y,y) - \varphi_1(x,y) &=&  R_1(x, y) + \varphi_2(x,y),\\
 \varphi_2(x+y,y) - \varphi_2(x,y) &=&  R_2(x, y),
 \end{array}
 \right.
\end{equation*}
or, equivalently,
\begin{equation*}
 \left\{
 \begin{array}{rcl}
 \overline\varphi_1(x+y,y)-\overline\varphi_1(x,y) &=& y(\overline R_1(x, y) +\overline\varphi_2(x,y)),\\
 \overline\varphi_2(x+y,y)-\overline\varphi_2(x,y) &=& y\overline R_2(x, y).
 \end{array}
 \right.
\end{equation*}
If $$[\overline R_2](y):= \int_{\T} \overline R_2(x, y) = 0,  \qquad \text{ for all } y \in K^h_b,$$ an explicit (unique) formal solution of the equation (using the Fourier series of the maps $\overline R_1$ and $\overline R_2$) exists, but involves an infinite number of so-called ``small divisors''. A classical way to overcome these difficulties which does not assume that $[\overline R_2]=0$, is to remove the term $[\overline R_2]$, to truncate the maps $\overline R_1$ and $\overline R_2$ (so that the difference between the maps and the truncations is sufficiently small), and to consider the truncated version of the equations above, namely, 
\begin{equation}
\label{equation:multi_difference_equation_no_rescaling}
 \left\{
 \begin{array}{rcl}
 \varphi_1(x+y,y) - \varphi_1(x,y) &=&  \mathscr T_N R_1(x, y) + \varphi_2(x,y),\\
 \varphi_2(x+y,y) - \varphi_2(x,y) &=&  \mathscr T_N R_2(x, y) - [R_2](y),
 \end{array}
 \right.
\end{equation}
or, equivalently,
\begin{equation}
\label{equation:multi_difference_equation}
 \left\{
 \begin{array}{rcl}
 \overline\varphi_1(x+y,y)-\overline\varphi_1(x,y) &=& y(\mathscr T_N \overline R_1(x, y) +\overline\varphi_2(x,y)),\\
 \overline\varphi_2(x+y,y)-\overline\varphi_2(x,y) &=& y (\mathscr T_N \overline R_2(x, y) - [\overline R_2](y)),
 \end{array}
 \right.
\end{equation}
for some appropriate value of $N \in \N$ (which will depend on the norm of $R$),
where $[R_2]$ denotes the average of $R_2$ over $\T$ (notice that $[R_2](y) = y^2[R_2](y)$ for all $y \in K^h_b$), and $\mathscr T_N$ denotes the truncation of the Fourier series of the map to modes of rank at most $N$.

Equation \eqref{equation:multi_difference_equation_no_rescaling} (and \eqref{equation:multi_difference_equation}) is equivalent to
\begin{equation}
    \label{eq:cohomological_truncated}
    \varphi A - A\varphi  
    = \mathscr T_N R - (0, [R_2]).
\end{equation}

By considering an appropriate domain of definition, we shall see that it is possible to find an explicit solution for \eqref{equation:multi_difference_equation} whose norm depends explicitly on the norm of $\overline R$ and on the (parameters of the) chosen domain. 
Assuming that $\varphi$ satisfies \eqref{eq:cohomological_truncated}, and using our previous remarks, we can express $R_+$ as
\begin{equation}
\label{eq:expansion_R_plus}
\begin{aligned}
    R_+ & =   (R - \mathscr T_NR) + (\varphi A - \varphi T\Phi) + (R\Phi - R) + (\varphi + \psi)T \Phi + (0, [R_2]) \\
    & = R_a^+ + R_b^+ + R_c^+ + R_d^+ + R_e^+.
    \end{aligned}
\end{equation}
Simple Taylor expansions combined with Cauchy estimates allow to prove that the first four terms $R_a^+$, $R_b^+$, $R_c^+$ and $R_d^+$ are much smaller in norm than $R$. To show that this holds also for the last term $R_e^+$, a classical assumption is to assume that $T_R$ and hence $T_{R_+}$ satisfy the intersection property: it implies \textit{a posteriori} that if the first four terms of \eqref{eq:expansion_R_plus} are much smaller than $R$, then $R_e^+$ is also much smaller.

This technique however works for real values of the parameters $(x,y)$ and has no equivalent in complex domains. The original step of this proof is to find another tool which extends to complex domains.  The idea behind it is to show that if there is a one-form of the type \eqref{equation:expression_one_form} such that $T_R^{\ast}\lambda-\lambda$ is exact then the norm of $[\overline R_2]$ must be much smaller than the one of $R$, even for complex values of $(x,y)$. 

The procedure above describes a full step of the iterative procedure. Note that the new map $T_+$ will also be a billiard-like map. In particular, if we set $\lambda_+=\Phi^\ast\lambda$, then $T_+^\ast\lambda_+-\lambda_+$ is exact. Hence as we shall see, by properly tuning the parameters involved, this procedure can be iterated and will lead to a proof of Theorem \ref{thm:KAM}.

\subsubsection*{Iterative procedure}

The iterative procedure consists of applying the previously sketched KAM repeatidly in the following way. 

Given $R$ as in Theorem \ref{thm:KAM} and sufficiently small in norm, we build on a sequence of analytic maps $R_n:\T_{r_n}\times K_{h_n}^{b_n}\to\C^2$ where $r_n,b_n,h_n>0$ are decreasing sequences of parameters converging respectively to $r_\infty\geq r/2$, $b_\infty\geq b/2$ and $0$. This allows us to consider the limit set, which turns out to be compact, given by 
\[
K_{\infty} = \bigcap_{n\geq 0} K_{h_n}^{b_n}.
\]

The sequence of maps satisfies the KAM step for any $n\geq 0$: most importantly, the maps $T_{R_n}$ and $T_{R_{n+1}}$ are conjugated via a diffeomorphism $\Phi_n = I+\varphi_n$,
\[
T_{R_{n+1}} = \Phi_n^{-1} T_{R_n} \Phi_n,
\]
where $\varphi_n:\T_{r_n'}\times K_{h_n'}^{b_n'}\to\C^2$ is an analytic map having  $r_n'\in(r_{n+1},r_n)$, $b_n'\in(b_{n+1},b_n)$ and $h_n'\in(h_{n+1},h_n)$. Moreover, the inverse of $\Phi_n$ can be written as
\[
\Phi_n^{-1} = I+\psi_n
\]
for a given analytic map $\psi_n:\T_{r_n'}\times K_{h_n'}^{b_n'}\to\C^2$.

We show by induction that estimates of the form
\[
\|R_n\|_{\mathscr{C}^1_{\textup{hol}}(K_\infty)} \leq \varepsilon_n,
\quad
\|\varphi_n\|_{\mathscr{C}^1_{\textup{hol}}(K_\infty)}\leq \varepsilon_n^{3/4},
\quad
\|\psi_n\|_{\mathscr{C}^1_{\textup{hol}}(K_\infty)}\leq \varepsilon_n^{3/4},
\]
hold, where $(\varepsilon_n)_{n\geq 0}$ is a sequence of positive numbers satisfying
\[
\varepsilon_{n+1}=\varepsilon_n^{3/2}, \qquad n\geq 0.
\]

By choosing $R_0=R$ sufficiently small, we can impose $\varepsilon_0\in(0,1)$ and show the convergence in the space of $\mathscr C^{\infty}$-holomorphic maps\footnote{By convergence in the space of $\mathscr C^{\infty}$-holomorphic maps, we mean that the sequences converge in the space of $\mathscr C^{r}$-holomorphic maps for any $r\geq 0$.} on $K$ of $R_n$ to $0$
and of the composition 
\[
\Phi_0\circ\cdots\circ\Phi_n
\]
to a $\mathscr C^{\infty}$-holomorphic diffeomorphism $\Phi:\T_{r/4}\times K_{\infty}\to\C^2$ satisfying
\[
T_R\Phi = \Phi A,
\]
which will finish the proof of Theorem \ref{thm:KAM}.

\vspace{0.3cm}

Before giving a formal statement of the ``KAM step" outlined in Subsection \eqref{subsec:sketch}, let us introduce a few notations which will simplify the exposition.

\subsection{Notations}

Let positive numbers $r,b,h>0$ and $k\in \N$. The set $K_h^b$ is open, hence we can consider analytic maps of the form $\varphi:\T_{r}\times K^{b}_{h}\to \C^k$. If, in addition, $\varphi$ satisfies
\begin{equation}
 \label{equation:real_analytic}
 \varphi(\R/\Z\times (K_h^b\cap\R))\subset\R^k
\end{equation}
we will say that it is \textit{real analytic}. In particular, it is also $\mathscr C^1$-holomorphic and as such, the norm\footnote{The the introduction to Theorem \ref{thm:KAM} for an explanation of the $\mathscr C^1$-holomorphic norm applied to maps of this form.}
\begin{equation}
\label{eq:norm_version_1}
\|\varphi\|_{\mathscr C_{\text{hol}}^1(K_h^b,\mathscr C^{\omega}(\T_r,\C^k))},
\quad\text{or sometimes written}\quad 
\|\varphi\|_{\mathscr C_{\text{hol}}^1(K_h^b)},
\end{equation}
is well-defined.
We define
\[ \Domain{r}{b}{h}:= \T_r \times K^b_h, \qquad \text{for any } b, r, h > 0,\]
and denote the space of real analytic maps on $\Domain{r}{b}{h}$ 
having finite $\mathscr C^1$-holomorphic norm by 
\[
\mathscr C^1_{r,b,h}(\RR^k) = \{ \varphi\in\mathscr C^{\omega}(\Domain{r}{b}{h},\C^k) \,:\, 
\varphi\left(\RR/\ZZ\times (K_h^b\cap\R)\right)\subset\R^k,\quad \|\varphi\|_{\mathscr C_{\text{hol}}^1(K_h^b)}<+\infty\}.
\]
To simplify the notation, we will sometimes denote this normed space simply by $(\mathscr C^1_{r,b,h}, \|\cdot\|_{r,b,h}),$ for any $k \in \N$,

Throughout the proof, we will often consider $\C^2$-valued $\mathscr{C}^1$-holomorphic functions of the form 
\begin{equation}
\label{eq:m_form}
F(x, y) = (F_1(x, y), F_2(x, y)) = (\mathcal{O}(y^m), \mathcal{O}(y^{m + 1})), \qquad x \in \C/\Z, \quad y \in \C, \quad m \geq 1.  
\end{equation}
As described in the previous section, it will be useful to consider the $\mathscr C^1$-holomorphic norm of the `rescaled' function $\big(\tfrac{F_1}{y^m}, \tfrac{F_2}{y^{m + 1}} \big)$ (see, for example, Equation \eqref{equation:multi_difference_equation}). 
For this reason, we introduce the following notation.



We define 
\[
\|F\|_{r,b,h,m} = \left\{  \begin{array}{ll} \left\|\left(\tfrac{F_1}{y^m}, \tfrac{F_2}{y^{m + 1}} \right)\right\|_{r,b,h}, & \text{if $F$ is of the form }  \eqref{eq:m_form}\\
+\infty. & \text{otherwise}. \end{array}\right.
\]
Let us point out that
\[ \| F\|_{r, b, h} \leq C \| F\|_{r, b, h, m},\]
where $C$ is a constant depending only on $b, h, m$.

We denote the space of maps in $\mathscr C^1_{r,b,h}$ and of the form \eqref{eq:m_form} by 
\[
\mathscr C^1_{r,b,h,m} = \{ \varphi \in \mathscr C^1_{r,b,h}\,:\, \|\varphi\|_{r,b,h,m}<+\infty\}.
\]
Notice that $\| \cdot \|_{r,b,h,m}$ is a well-defined norm on $\mathscr C^1_{r,b,h,m}$.

The solutions of Equation \eqref{equation:multi_difference_equation} will be defined not on domains of the form $\Domain{r}{b}{h}$ but on slightly larger domains of the form
\[ \Domain{r}{b}{h}^+= \Domain{r}{b}{h} \cup A(\Domain{r}{b}{h}), \qquad r, b >0, \qquad  h \geq 0,\]
where $A:\CC^2\to\CC^2$ stands for the linear map $A(x,y)=(x+y,y).$ We denote the space of $\mathscr C^1$-holomorphic maps defined on $\Domain{r}{b}{h}^+$ by 
\[{\mathscr C^{1+}_{r,b,h}} = \{\varphi\in \mathscr C^1_{r,b,h}\,:\, \varphi\circ A\in\mathscr C^1_{r,b,h}\}. \] 
and endow it with the norm
\[
\|\varphi\|_{r,b,h}^+ = \max \{\|\varphi\|_{r,b,h},\|\varphi\circ A\|_{r,b,h}\}. \] 
Furthermore, as we did for the spaces $\mathscr C^{1}_{r,b,h}$, for any $m \geq 1$ we define
\[ \|\varphi\|_{r,b,h,m}^+ =  \max \{\|\varphi\|_{r,b,h,m},\|\varphi\circ A\|_{r,b,h,m}\}. \] 
and denote the space of maps in $\mathscr C^{1+}_{r,b,h}$ and of the form \eqref{eq:m_form} by
\[  {\mathscr C^{1+}_{r,b,h,m}} = \{\varphi\in \mathscr C^{1+}_{r,b,h}\,:\,\|\varphi\|_{r,b,h,m}^+ < +\infty\}. \]

\subsection{Proof of Theorem \ref{thm:KAM}}  

We now can give a formal statement for the iterative step of the KAM scheme.

\begin{proposition}
\label{prop:estimates_one_step}
 Let $r,b,h\in(0,1]$, $\sigma \in (0, \tfrac{r}{4}),$ $\overline{\sigma} \in (0, \tfrac{1}{4}\min\{b, h\})$ and $N \in \N_{> 0}$ satisfying
 \begin{equation}
 \label{eq:relation_N_h}
 h \leq \tfrac{1}{2\sqrt 2}\gamma N^{-\tau}.
 \end{equation}
 There exist constants $0 < c_{\ast}< 1 < C_\ast ,$  depending only on $\gamma$ and $\tau$, such that for any map $R\in\mathscr C^1_{r,b,h, 2}(\R^2)$ satisfying
 \begin{equation}
 \label{equation:assumption_R}
 \|R\|_{r,b,h,2}\leq c_\ast
\sigma^{2\tau+3}\overline\sigma^3
 \end{equation}
 for which $T_R$ 
 (as in \eqref{eq:perturbation_form}) is an analytic billard-like map on $\Domain{r}{b}{h}$ (in the sense of Definition \ref{def:billiard-like}) with associated 1-form $\lambda$ defined on $\DomainP{r + \overline \sigma}{b + \overline \sigma}{h + \overline \sigma}$ of the form
\[\lambda(x, y) = (y^2 + y^2\overline \mu(x, y))dx + y \overline \nu(x, y)dy, 
\qquad |\overline \mu|_{\DomainP{r + \overline \sigma}{b + \overline \sigma}{h + \overline \sigma}} + |\overline \nu|_{\DomainP{r + \overline \sigma}{b + \overline \sigma}{h + \overline \sigma}} < \frac{1}{4},\]
there exists a diffeomorphism onto its image $\Phi: \DomainP{r - \sigma}{b - \overline \sigma}{h - \overline \sigma} \to \C/\Z \times \C$ such that the following holds:

\begin{enumerate}
\item $R_+ := \Phi^{-1} T_R\Phi - A$ restricted to $\Domain{r-3\sigma}{b-3\overline\sigma}{h-3\overline\sigma}$ 
is well-defined and satisfies
 \[
 \|R_+\|_{r-3\sigma,b-3\overline\sigma,h-3\overline\sigma,2}
 \leq \frac{C_\ast e^{-2\pi N\sigma}}{\pi\sigma\overline \sigma^3} \|R\|_{r,b,h,2} + \frac{C_\ast}{\sigma^{4\tau+5}\overline\sigma^5}\|R\|_{r,b,h,2}^2,
 \]
 \item  $\Phi=I+\varphi$, with $\varphi \in \mathscr C^{1+}_{r-\sigma,b-\overline\sigma,h-\overline\sigma, 1}(\R^2)$ satisfying
 \[
 \|\varphi\|_{r-\sigma,b-\overline\sigma,h-\overline\sigma,1}^+
 \leq 
 \frac{C_2}{\sigma^{2(\tau+1)}\overline\sigma^2}\|R\|_{r,b,h,2},
 \]
 \item $\Phi^{-1}\mid_{\DomainP{r-2\sigma}{b-2\overline\sigma}{h-2\overline\sigma}} = I+\psi$, with  $\psi \in \mathscr C^{1+}_{r-2\sigma,b-2\overline\sigma,h-2\overline\sigma, 1}(\R^2)$ satisfying
 \[
 \|\psi\|^+_{r-2\sigma,b-2\overline\sigma,h-2\overline\sigma,1} 
 \leq 
 \frac{KC_2}{\sigma^{2(\tau+1)}\overline\sigma^2}\|R\|_{r,b,h,2},
 \]
\end{enumerate}
 where $C_2$ and $K$ are the constants defined at Propositions \ref{proposition:solution_truncated_difference_equation} and \ref{proposition:inverse_estimates} respectively.
\end{proposition}

For the sake of clarity of the exposition, we postpone the proof of the Proposition above to the next section and show first how we can iterate it to prove  Theorem \ref{thm:KAM}.

\begin{proof}[Proof of Theorem \ref{thm:KAM}]
We elaborate an iterative scheme relying on Proposition \ref{prop:estimates_one_step}. We divide the proof into three steps:
\begin{enumerate}
    \item {\it\underline{Sequences of parameters.}} Given two fixed parameters $\sigma,\overline\sigma\in(0,1)$, we define exponentially decreasing sequences
    $(\sigma_n)_{n\geq 0}$ and $(\overline\sigma_n)_{n\geq 0}$. Given $r,b\in(0,1)$, we can use them to build decreasing sequences
    $(r_n)_{n\geq 0}$, $(b_n)_{n\geq 0}$, $(h_n)_{n\geq 0}$ converging respectively to $r_\infty\geq r/2$, $b_\infty\geq b/2$ and $0$. 
    We construct additional sequences $(\varepsilon_n)_{n\geq 0}$ and $(N_n)_{n\geq 0}$  where $\varepsilon_n$ converges superexponentially to $0$ and $N_n$ is an integer satisfying $h_n\leq \gamma N^{-\tau}/2\sqrt 2$ for any $n$. Moreover $\varepsilon_0$ is of the form $\varepsilon_0 = \varepsilon(\gamma,\tau)r^{8\tau+10}b^{10}$.
    \item {\it\underline{Sequences of maps.}}  We construct sequences of real analytic $\C^2$-valued maps $R_n\in\mathscr C^1_{r_n,b_n,h_n,2}$, $\varphi_n\in\mathscr C^{1+}_{r_n-\sigma_n,b_n-\overline\sigma_n,h_n-\overline\sigma_n, 1},$ $\psi_n\in\mathscr C^{1+}_{r_n-2\sigma_n,b_n-2\overline\sigma_n,h_n-2\overline\sigma_n, 1}$, and of real analytic $\C$-valued maps $\overline \mu_n, \overline \nu_n \in  \mathscr C^{1+}_{r_n + \overline \sigma_n,b_n + \overline \sigma_n,h_n + \overline \sigma_n},$ satisfying the estimates
   \[  \|R_n\|_{r_n, b_n, h_n, 2} \leq \varepsilon_n, \qquad |\overline \mu_n|_{\DomainP{r_n + \overline \sigma_n}{b_n + \overline \sigma_n}{h_n + \overline \sigma_n}} + |\overline \nu_n|_{\DomainP{r_n + \overline \sigma_n}{b_n + \overline \sigma_n}{h_n + \overline \sigma_n}} < \frac{1}{4} \]
   \[\|\varphi_n\|^+_{r_n - \sigma_n, b_n - \overline \sigma_n, h_n - \overline \sigma_n, 1} \leq \varepsilon_n^{3/4},\qquad  \|\psi_n\|^+_{r_n - 2\sigma_n, b_n - 2\overline \sigma_n, h_n - 2\overline \sigma_n, 1}\leq \varepsilon_n^{3/4},\]
   and such that, for any $ n \geq 0$,
 \[ \lambda_{n} = (y^2 + y^2 \overline \mu_{n})dx + y \overline \nu_{n}dy,\]
 satisfies
    \[ 
    T_{R_{n}}^* \lambda_{n} - \lambda_{n} = d S_{n},
    \]
        for some real-analytic function $S_{n}: \Domain{r_n}{b_n}{h_n} \to \C$, and
    \[
    T_{R_{n+1}} = \Phi_n^{-1}\circ T_{R_n}\Phi_n : \Domain{r_{n + 1}}{b_{n + 1}}{h_{n + 1}}  \to \C/\Z \times \C
    \]
    is well-defined, where $\Phi_n = I+\varphi_n$ and $\Phi_n^{-1}\mid_{\DomainP{r_n-2\sigma_n}{b_n-2\overline\sigma_n}{h_n-2\overline\sigma_n}} = I+\psi_n$.    
    
    \item {\it\underline{Conjugating diffeomorphism.}} We show that, for any $k \geq 1$, $\Phi_0\circ\cdots\circ\Phi_n$ converges, as $n$ goes to $+\infty$, to a $\mathscr C^{k}$-holomorphic diffeomorphism $\Phi$ satisfying the conclusions of Theorem \ref{thm:KAM}.
\end{enumerate}

\vspace{0.2cm}
\noindent {\it\underline{Sequences of parameters.}}
We start with $\sigma,\overline\sigma\in(0,1)$ that we are going to set later. Let 
\[
\beta = 4^\tau,
\]
and introduce the sequences $(\sigma_n)_{n\geq 0}$ and $(\overline\sigma_n)_{n\geq 0}$ by setting
\[
\sigma_n = \frac{\sigma}{2^n},
\quad
\overline\sigma_n = \frac{\overline\sigma}{\beta^n},
\qquad\qquad n\in\Z_{\geq 0}.
\]
We first consider the following sequences $(r_n)_{n\geq 0}$, $(b_n)_{n\geq 0}$, $(h_n)_{n\geq 0}$ defined for $n\geq 0$ by
\begin{align*}
r_{n+1} & =  r_n-3\sigma_n, \qquad r_0=r;\\
b_{n+1} & =  b_n-3\overline\sigma_n, \qquad b_0=b;\\
h_{n+1} & =  h_n-3\overline\sigma_n, \qquad h_0=\tfrac{3\beta\overline\sigma}{\beta-1}.
\end{align*}
We easily check the following result:
\begin{lemma}
    \label{lemma:prop_seq_strips}
    The sequence $(h_n)_{n\geq 0}$ is of the form
    \[
    h_n = \frac{h_0}{\beta^n},\qquad n\geq 0
    \]
    hence decreases to $0$.
    Moreover if $\sigma\leq r/6$ and $\overline\sigma\leq \frac{\beta-1}{6\beta}b$, the decreasing sequences $(r_n)_{n\geq 0}$ and $(b_n)_{n\geq 0}$ converge respectively to $r_\infty\geq r/2$ and to $b_\infty\geq b/2$.
\end{lemma}

We now introduce a constant $\tilde c>0$ satisfying simultaneously the following assumptions:
\begin{equation}
    \label{eq:tilde_c}
    \tilde c\leq c_\ast,
    \qquad
    2^{8\tau+10}\beta^{10}\tilde c^{1/2}\leq 1,
    \qquad
    C_\ast(K+1)\tilde c < \tfrac{1}{2},
    \qquad
    \frac{K\tilde c^{3/4}\beta}{\beta-1}\leq\frac{1}{2},
\end{equation}
where $c_\ast$, $C_\ast$, $C_2$ and $K$ are the positive constants appearing in Proposition \ref{prop:estimates_one_step}.
We can define a decreasing sequence $(\varepsilon_n)_{n\geq0}$ of positive numbers by setting
\[
\varepsilon_0 = \tilde c\sigma^{8\tau+10}\overline\sigma^{10}
\quad\text{and}\quad 
\varepsilon_n = \varepsilon_0^{(3/2)^n},
\quad n\geq 0.
\]
Consider then the sequence of positive integers 
$(N_n)_{n\geq 0}$ defined by 

\[
N_n = \left\lceil
\frac{1}{2\pi\sigma_n}
\left|
\ln\left(
\frac{\pi}{2}\sigma_n\varepsilon_n
\right)
\right|
\right\rceil.
\]
We prove simple properties for them:
\begin{lemma}
    \label{lemma:prop_seq_eps_N}
    There exists $\overline\sigma_0\in(0,1)$ of the form
    \[
    \overline\sigma_0 = \sigma_\ast(\gamma,\tau)
    \left(\frac{\sigma}{\ln\sigma}\right)^\tau
    ,\qquad \sigma_\ast(\gamma,\tau)\in(0,1)
    \
    \]
    such that for $\overline\sigma\in(0,\overline\sigma_0]$ the sequences defined above satisfy the following, for any $n\geq 0$:
    \begin{enumerate}
        \item $\varepsilon_n\leq \tilde c \,\sigma_n^{8\tau+10}\overline\sigma_n^{10}$;
        \item $h_n \leq \tfrac{1}{2\sqrt 2}\gamma N_n^{-\tau}$;
        \item $\frac{C_\ast e^{-2\pi N_n\sigma_n}}{\pi\sigma_n\overline\sigma_n^3}\varepsilon_n + \frac{C_\ast}{\sigma_n^{4\tau+5}\overline\sigma_n^5}\varepsilon_n^2\leq \varepsilon_{n+1}$.
        \item \label{cond:fast_decrease} Moreover, for any $a, k,\ell > 0$, 
        \[
        \lim_{n \to \infty} \frac{\varepsilon_n^a}{\sigma_n^{\ell}\overline\sigma_n^{k}} \to 0.
        \]
    \end{enumerate}
\end{lemma}

\begin{proof}[Proof of Lemma \ref{lemma:prop_seq_eps_N}]
We prove the first assertion by induction, noticing that the case when $n=0$ is immediate by construction of $\varepsilon_0$. Assume that the result holds for a given $n\geq 0$. From the definitions of the sequences $(\varepsilon_n)_{\geq 0}$, $(\sigma_n)_{\geq 0}$ and $(\overline\sigma_n)_{\geq 0}$ follows
\[
\frac{\varepsilon_{n+1}}{\tilde c\sigma_{n+1}^{8\tau+10}\overline\sigma_{n+1}^{10}}
= 
2^{8\tau+10}\beta^{10}\frac{\varepsilon_{n}^{3/2}}{\tilde c\sigma_{n}^{8\tau+10}\overline\sigma_{n}^{10}}
\leq 
2^{8\tau+10}\beta^{10} \frac{\tilde c^{3/2}\sigma_{n}^{12\tau+15}\overline\sigma_{n}^{15}}{\tilde c\sigma_{n}^{8\tau+10}\overline\sigma_{n}^{10}}
\leq 
2^{8\tau+10}\beta^{10}\tilde c^{1/2}
\sigma_{n}^{4\tau+5}\overline\sigma_{n}^{5}\leq 1
\]
where the last inequality follows from the fact that $\sigma_n,\overline\sigma_n\in(0,1)$ and Assumption \eqref{eq:tilde_c} on $\tilde c$. 

To prove the second assertion, we use the definition of $(N_n)_{n\geq 0}$ to observe that for $n\geq 0$,
\begin{align*}
    h_n^{1/\tau}N_n
    & \leq 
    \frac{h_0^{1/\tau}}{\beta^{n/\tau}}
    \left(
    \frac{1}{2\pi\sigma_n}
    \left|
    \ln\left(
    \frac{\pi}{2}\sigma_n \varepsilon_n
    \right)
    \right|
    +1
    \right)\\
    & =
    \frac{h_0^{1/\tau}}{\beta^{n/\tau}}
    \left(
    \frac{2^n}{2\pi\sigma}
    \left|
    \ln\left(
    \frac{\pi\sigma}{2^{n+1}}(\varepsilon_0)^{(3/2)^n}
    \right)
    \right|
    +1
    \right)\\
    & \leq
    \frac{h_0^{1/\tau}}{2\pi\sigma}
    \left(
    \frac{2}{\beta^{1/\tau}}
    \right)^n
    \left|
    \ln\left(
    \frac{\pi\sigma}{2^{n+1}}
    \right)
    \right|
    +
    \frac{h_0^{1/\tau}}{2\pi\sigma}
    \left(
    \frac{3}{\beta^{1/\tau}}
    \right)^n
    \left|
    \ln\left(
    \varepsilon_0
    \right)
    \right|
    +
    \frac{h_0^{1/\tau}}{\beta^{n/\tau}}.
    \\
\end{align*}
By the choice of $\beta$, the first, second and third term converge to $0$ in $n$. 
Moreover the sup of each term can bounded by 
\[
C\overline\sigma^{1/\tau}\frac{\ln(\sigma)}{\sigma}
\]
where $C>0$ is a constant depending only on $\gamma$ and $\tau$.
Hence we can define $\overline\sigma_0\in(0,1)$ as in the statement of Lemma \ref{lemma:prop_seq_eps_N} such that for any $\overline\sigma\in(0,\overline\sigma_0)$ we have
\[
\sup_{n\geq 0} 
    \frac{h_0^{1/\tau}}{2\pi\sigma}
    \left(
    \frac{2}{\beta^{1/\tau}}
    \right)^n
    \left|
    \ln\left(
    \frac{\pi\sigma}{2^{n+1}}
    \right)
    \right|
    +
    \frac{h_0^{1/\tau}}{2\pi\sigma}
    \left(
    \frac{3}{\beta^{1/\tau}}
    \right)^n
    \left|
    \ln\left(
    \varepsilon_0
    \right)
    \right|
    +
    \frac{h_0^{1/\tau}}{\beta^{n/\tau}}
    \leq
    \left(\frac{\gamma}{2\sqrt 2}\right)^{1/\tau}
\]
and the second assertion is proven.

To prove the third assertion, we give an upper bound on the two terms in the sum of the left-hand side of the inequality. First by definition of $N_n$ and \eqref{eq:tilde_c},
\[
    \frac{C_\ast e^{-2\pi N_n\sigma_n}}{\pi\sigma_n \overline \sigma_n^3}\varepsilon_n
    \leq 
    \frac{C_\ast}{\pi\sigma_n\overline \sigma_n^3}
    e^{
    \ln\left(
    \frac{\pi}{2}\sigma_n\varepsilon_n
    \right)
    }
    \varepsilon_n
    =
        \frac{C_\ast}{2 \overline \sigma_n^3}  \varepsilon^{2}_n
        =
    \frac{C_\ast}{2 \overline \sigma_n^3}  \left(\tilde c\sigma_n^{8\tau+10}\overline\sigma_n^{10}\right)^{1/2}\varepsilon^{3/2}_n
    \leq
    \frac{1}{2}\varepsilon_n^{3/2}.
\]
For the second term, we use the first assertion proven above to obtain
\[
    \frac{C_\ast}{\sigma_n^{4\tau+5}\overline\sigma_n^5}\varepsilon_n^2
    \leq
    \frac{C_\ast}{\sigma_n^{4\tau+5}\overline\sigma_n^5}
    \left(\tilde c\sigma_n^{8\tau+10}\overline\sigma_n^{10}\right)^{1/2}
    \varepsilon_n^{3/2}
    = 
    C_\ast\tilde c^{1/2}\varepsilon_n^{3/2}
    \leq 
    \frac{1}{2}\varepsilon_n^{3/2}
\]
where the last inequality comes from the third estimates on $\tilde c$ given in \eqref{eq:tilde_c}. Whence the third assertion is proven. 

Finally the fourth assertion follows using similar arguments.
\end{proof}

\vspace{0.2cm}
\noindent {\it\underline{Sequences of maps.}}
Now we assume that $\sigma =  r/6$ and $\overline\sigma= \min\{\overline\sigma_0,(\beta-1)b/\beta\}
$ so that the Lemmas \ref{lemma:prop_seq_strips} and \ref{lemma:prop_seq_eps_N} are satisfied. We define the desired sequences of maps by induction. Note that with this choice of $\sigma, \overline\sigma$, we can write 
\[
\varepsilon_0 = \varepsilon(\gamma,\tau)r^{8\tau + 10}
\min\left\{\sigma_\ast(\gamma,\tau)\left(\frac{r}{6\ln (r/6)}\right)^{\tau},b\right\}^{10}>0
\]
where $\varepsilon(\gamma,\tau)$ depends only on $\gamma$ and $\tau$. By eventually shrinking $\varepsilon$, one can assume that $\varepsilon_0$ has the form given in \eqref{eq:smallness_condition_R}.

Let $R$, $\lambda$ and $S$ as in Theorem \ref{thm:KAM}. 
We consider the restrictions of $R$ and $S$ to $\Domain{r_0}{h_0}{b_0}$ and of $\lambda$ to $\Domain{r_0 + \overline \sigma_0}{h_0 + \overline \sigma_0}{b_0 + \overline \sigma_0}$,  which we denote by $R_0$, $S_0$ and $\lambda_0 = (y^2 + y^2 \overline \mu_0) dx + y \overline \nu_0 dy$, respectively. We assume that
\[
\|R_0\|_{r_0,b_0,h_0,2}\leq \varepsilon_0, \qquad  |\overline \mu_0|_{\DomainP{r_0 + \overline \sigma_0}{b_0 + \overline \sigma_0}{h_0 + \overline \sigma_0}} + |\overline \nu_0|_{\DomainP{r_0 + \overline \sigma_0}{b_0 + \overline \sigma_0}{h_0 + \overline \sigma_0}} \leq \frac{1}{8},
\]
which follows directly from the bounds of the form \eqref{eq:smallness_condition_R} and \eqref{eq:smallness_condition_form} as in the statement of Theorem \ref{thm:KAM}. Notice that, with these assumptions, we are in position to apply Proposition \ref{prop:estimates_one_step}.

 We will construct the sequences $R_n$, $\lambda_n$ and $S_n$ inductively using Proposition \ref{prop:estimates_one_step}, where $\Phi_{n}$ will be given by the proposition and we define 
 \[R_{n + 1} = \Phi_n^{-1} T_{R_n} \Phi_n - A, \qquad \lambda_{n + 1} = \Phi_n^{*} \lambda_n, \qquad S_{n + 1} = S_n \circ \Phi_n,\]
 restricted to appropriate domains. 
 
 Let $n\geq 0$ and assume that $R_n \in \mathscr C^1_{r_n,b_n,h_n, 2}(\R^2)$, $\lambda_n = (y^2 + y^2 \overline \mu_n) dx + y \overline \nu_n dy$, with $\overline \mu_n, \overline \nu_n \in  \mathscr C^{1+}_{r_n + \overline \sigma_n,b_n + \overline \sigma_n,h_n + \overline \sigma_n}(\R)$, and $S_n \in \mathscr C^1_{r_n,b_n,h_n}(\R)$ have been defined (by iterating Proposition \ref{prop:estimates_one_step}) and satisfy
\begin{equation}
    \label{eq:inductiion_hyp}
\|R_n\|_{r_n,b_n,h_n,2}\leq \varepsilon_n,  \qquad T^* \lambda_n - \lambda_n = dS_n, \qquad \Lambda_n  < \frac{1}{4},
\end{equation}
where, for the sake of clarity, we denote
\begin{equation}
\label{eq:Lambda}
\Lambda_n:=  |\overline \mu_n|_{\DomainP{r_n + \overline \sigma_n}{b_n + \overline \sigma_n}{h_n + \overline \sigma_n}} + |\overline \nu_n|_{\DomainP{r_n + \overline \sigma_n}{b_n + \overline \sigma_n}{h_n + \overline \sigma_n}}. 
\end{equation}

Let us show that, with these assumptions, we can apply Proposition \ref{prop:estimates_one_step} to define $R_{n + 1}$, $\lambda_{n + 1}$ and $S_{n + 1}$  satisfying \eqref{eq:inductiion_hyp} (replacing $n$ by $n + 1$), and thus that we can iterate this process indefinitely.

Indeed by Lemma \ref{lemma:prop_seq_eps_N} and the first assumption on $\tilde c$ given in \eqref{eq:tilde_c}, 
\[
\varepsilon_n\leq \tilde c\sigma_n^{8\tau+10}\overline\sigma_n^{10}\leq 
c_\ast\sigma_n^{2\tau+3}\overline\sigma_n^{3},
\]
and the second assertion of Lemma \ref{lemma:prop_seq_eps_N} states that
\[
h_n\leq\frac{1}{2\sqrt 2}\gamma N_n^{-\tau}.
\]

Hence, Proposition \ref{prop:estimates_one_step} implies the existence of two $\C^2$-valued maps $\varphi_n\in\mathscr C^{1+}_{r_n-\sigma_n,b_n-\overline\sigma_n,h_n-\overline\sigma_n}$ and $\psi_n\in\mathscr C^{1+}_{r_n-2\sigma_n,b_n-2\overline\sigma_n,h_n-2\overline\sigma_n}$ such that
$\Phi_n := I+\varphi_n$ is a diffeomorphism onto its image, $\Phi^{-1}_n\mid_{\DomainP{r_n-2\sigma_n}{b_n-2\overline\sigma_n}{h_n-2\overline\sigma_n}} = I + \psi_n$, and the map 
\[
\Phi^{-1}_nT_{R_n}\Phi_n :\Domain{r_{n + 1}}{b_{n + 1}}{h_{n + 1}}  \to\C/\Z\times \C
\]
is a well-defined real analytic map which can be written as
$\Phi^{-1}_nT_{R_n}\Phi_n = A+R_{n+1}$ for some $R_{n+1}\in\mathscr C^1_{r_{n+1},b_{n+1},h_{n+1}}$.
Moreover, we have the following estimates:
\begin{align*}
\|R_{n+1}\|_{r_{n+1},b_{n+1},h_{n+1},2} 
&\leq
\frac{C_\ast e^{-2\pi N_n\sigma_n}}{\pi\sigma_n\overline \sigma_n^3}\|R_n\|_{r_n,b_n,h_n,2}+\frac{C_\ast}{\sigma_n^{4\tau+5}\overline\sigma_n^5}\|R_n\|_{r_n,b_n,h_n,2}^2\\
&\leq
\frac{C_\ast e^{-2\pi N_n\sigma_n}}{\pi\sigma_n\overline \sigma_n^3}\varepsilon_n+\frac{C_\ast}{\sigma_n^{4\tau+5}\overline\sigma_n^5}\varepsilon_n^2\\
&\leq \varepsilon_{n+1},
\end{align*}
where from the second to third line we used the third assertion of Lemma \ref{lemma:prop_seq_eps_N}. 

By Proposition \ref{prop:estimates_one_step} and the first assertion of Lemma \ref{lemma:prop_seq_eps_N}, we also have the following estimates for the norms of $\varphi_n$ and $\psi_n$:
\begin{align*}
\|\varphi_{n}\|^+_{r_{n}-\sigma_n,b_{n}-\overline\sigma_n,h_{n}-\overline\sigma_n,1} 
&\leq
\frac{C_\ast}{\sigma_n^{2(\tau+1)}\overline\sigma_n^2}\varepsilon_n\\
&\leq
\frac{C_\ast}{\sigma_n^{2(\tau+1)}\overline\sigma_n^2}(\tilde c\sigma_n^{8\tau+10}\overline\sigma_n^{10})^{1/4}
\varepsilon_n^{3/4}\\
&\leq
C_\ast\tilde c^{3/4}
\varepsilon_n^{3/4}\\
&\leq \varepsilon_{n}^{3/4},
\end{align*}
where the last inequality follows from the second assumption on $\tilde c$ given in \eqref{eq:tilde_c}. 
In a similar way we obtain
\[
\|\psi_{n}\|^+_{r_{n}-2\sigma_n,b_{n}-2\overline\sigma_n,h_{n}-2\overline\sigma_n,1} \leq \varepsilon_n^{3/4}.
\]

Noticing that $\varepsilon_n^{3/4} < \sigma_n \overline \sigma_n$, it follows that the maps $\Phi_n$, and $T_{R_{n + 1}}$ are well-defined transformations with the following domains/codomains
\[ 
\Phi_n: \DomainP{r_n - \sigma_n}{b_n - \overline \sigma_n}{h_n - \overline \sigma_n} \to  \DomainP{r_n}{b_n}{h_n},\]
\[T_{R_{n + 1}}: \Domain{r_{n + 1}}{b_{n + 1}}{h_{n + 1}} \to \Domain{r_{n + 1} + \overline \sigma_{n + 1}}{b_{n + 1} + \overline \sigma_{n + 1}}{h_{n + 1} + \overline \sigma_{n + 1}}.
\]
In particular, 
\[\lambda_{n + 1}= \Phi_n^* \lambda_n,\qquad S_{n + 1}= S_n \circ \Phi_n,\] are well-defined on $\DomainP{r_{n + 1} + \overline \sigma_{n + 1}}{b_{n + 1} + \overline \sigma_{n + 1}}{h_{n + 1} + \overline \sigma_{n + 1}}$ and $\Domain{r_{n + 1}}{b_{n + 1}}{h_{n + 1}}$, respectively, and satify
\[ T_{R_{n + 1}}^* \lambda_{n + 1} - \lambda_{n + 1} = dS_{n + 1}, \]
on $\Domain{r_{n + 1}}{b_{n + 1}}{h_{n + 1}}$. Denoting 
\[\varphi_n =  (\varphi_{n , 1}, \varphi_{n, 2}) = (y \overline \varphi_{n , 1}, y^2 \overline \varphi_{n, 2}),\]
a direct calculation yields
\begin{align*}
    \lambda_{n + 1} & = y^2\left( (1 + y \overline \varphi_{n, 2})(1 + \overline \mu_n \circ \Phi_n)(1 + \partial_x \varphi_{n, 1}) + y(1 + y \overline \varphi_{n, 2})\overline (\nu_n \circ \Phi_n) \partial_x \overline \varphi_{n, 2}\right)dx \\
    & \quad +y\left(y(1 + y \overline \varphi_{n, 2})(1 + \overline \mu_n \circ \Phi_n)\partial_y \varphi_{n, 1} + (1 + y \overline \varphi_{n, 2})\overline (\nu_n \circ \Phi_n) (1 + \partial_y \varphi_{n, 2}\right)dy,
\end{align*}
Thus, we can express $\lambda_{n + 1}$ as 
\[ \lambda_{n + 1} = (y^2 + y^2 \overline \mu_{n + 1}) dx + y \overline \nu_{n + 1} dy\]
with $\overline \nu_{n + 1}, \overline \mu_{n + 1} \in  \mathscr C^{1+}_{r_{n + 1} + \overline \sigma_{n + 1},b_{n + 1} + \overline \sigma_{n + 1},h_{n + 1} + \overline \sigma_{n + 1}}(\R)$, and there exists a constant $C > 0$, independent of all parameters, such that
\[ \Lambda_{n + 1} \leq \Lambda_n\left(1 + \frac{C}{\overline \sigma_n}\|\varphi_{n}\|^+_{r_{n}-\sigma_n,b_{n}-\overline\sigma_n,h_{n}-\overline\sigma_n,1}\right),  \]
where $\Lambda_n$, $\Lambda_{n + 1}$ are given by \eqref{eq:Lambda}. 

Hence, using the estimates above for $\|\varphi_{n}\|_{r_{n}-\sigma_n,b_{n}-\overline\sigma_n,h_{n}-\overline\sigma_n,1}$, we have
\[ \Lambda_{n + 1} \leq \Lambda_0 \exp\left(C \sum_{j = 0}^n \varepsilon_0^{\tfrac{1}{2}\big(\tfrac{3}{2}\big)^j}\right).\]
Notice that, up to choosing the constant $\tilde c$ in the definition of $\varepsilon$ smaller, we may assume that the sum in the equation above is smaller than $\frac{1}{2C}$. Thus, since $\Lambda_0 < \tfrac{1}{8}$ by assumption, we have
\[ \Lambda_{n + 1} < \Lambda_0  \exp\big(\tfrac{1}{2}\big)< \frac{1}{4}.\]

Therefore, we have shown that $R_{n + 1}$, $\lambda_{n + 1}$ and $S_{n + 1}$ (defined using Proposition \ref{prop:estimates_one_step}) satisfy \eqref{eq:inductiion_hyp} (replacing $n$ by $n + 1$) and thus we can define the sequences described above for all $n \geq 0$.

\vspace{0.2cm}
\noindent {\it\underline{Conjugating diffeomorphism.}}
Consider the composition
\[
\Phi^{(n)} = \Phi_0\circ\Phi_1\circ\cdots\circ\Phi_n,
\qquad n\in\Z_{\geq 0}
\]
and set 
\[
K_{\infty} = \bigcap_{n\geq 0} K_{h_n}^{b_n} = K_0^{b_\infty}.
\]
The following lemma shows that $\Phi^{(n)}:\Domain{r_n-\sigma_n}{h_n-\overline\sigma_n}{b_n-\overline\sigma_n}\to\C/\Z\times\C$ is well-defined and satisfies crucial properties to finish the proof:
\begin{lemma}
    \label{lemma:composition_diffeos}
    For any $n\geq 0$, $\Phi^{(n)}:\Domain{r_n-\sigma_n}{h_n-\overline\sigma_n}{b_n-\overline\sigma_n} \to\C/\Z\times \C$ is a well-defined real-analytic diffeomorphism onto its image which  can be written as
    \[
    \Phi^{(n)} = I + \varphi^{(n)},
    \qquad \varphi^{(n)} \in\mathscr C^1_{r_n-\sigma_n, b_n-\overline\sigma_n, h_n-\overline\sigma_n}.
    \]
    It satisfies
    \begin{enumerate}
        \item $\sup_{n\geq 0} \|\varphi^{(n)}\|_{r_n-\sigma_n, b_n-\overline\sigma_n, h_n-\overline\sigma_n,1}\leq \frac{1}{2}$;
        \item For any integers $n\geq 1$ and $k\geq 0$, there is a constant $C_k>0$ depending only on $k$ such that
        \[
        \|D^k\varphi^{(n)}-D^k\varphi^{(n-1)}\|_{\mathscr C^1_{\text{{hol}}}(K_{\infty})}
        \leq
        C_k\frac{\varepsilon_n^{3/4}}{\sigma_n^{r+1}\overline\sigma_n}.
        \]
    \end{enumerate}
\end{lemma}

\begin{proof}[Proof of Lemma \ref{lemma:composition_diffeos}]
    Given $n\geq 0$, we can formaly write $\varphi^{(n)}$ as
    \[
    \varphi^{(n)} = \sum_{j=0}^n u_j,
    \qquad
    u_j := \varphi_j\circ(I+\varphi_{j+1})\circ\cdots\circ (I+\varphi_n).
    \]
    Let us prove by induction on $j$ dreasing from $n$ to $0$ that
    \begin{equation}
    \label{eq:induction_uj}
    \sum_{i=j+1}^n \|u_j\|_{r_n-\sigma_n, b_n-\overline\sigma_n, h_n-\overline\sigma_n,1}\leq \frac{1}{2}\overline\sigma_j,
    \end{equation}
    and the first assertion will follow by taking $j=0$. The result is true at $j=n$ since by construction $u_n=\varphi_n$, and by Lemma \ref{lemma:prop_seq_eps_N}
    \[
    \|\varphi_n\|_{r_n-\sigma_n, b_n-\overline\sigma_n, h_n-\overline\sigma_n,1}
    \leq \varepsilon_n^{3/4} 
    \leq \tilde c
    \,\sigma_n^{8\tau+10}\overline\sigma_n^{10}
    \leq \frac{1}{2}\overline\sigma_n,
    \]
    where the last inequality follows from the choice of $\tilde c$ and the fact that $\sigma,\overline\sigma\in(0,1)$.
    Assume now that \eqref{eq:induction_uj} is satisfied for a given integer $j$ between $0$ and $n-1$. 
    We can apply Proposition \ref{proposition:composition_estimates} to each $u_i$ with $i\geq j$ and $\kappa=1/2$, noticing that
    \[
    u_i = \varphi_i\circ\left(I+\sum_{i'=i+1}^nu_{i'}\right).
    \]
    We get
    \[
    \|u_i\|_{r_n-\sigma_n, b_n-\overline\sigma_n, h_n-\overline\sigma_n,1}
    \leq K\|\varphi_i\|_{r_n-\sigma_n, b_n-\overline\sigma_n, h_n-\overline\sigma_n,1}
    \leq 
    K\varepsilon_i^{3/4}
    \leq K\tilde c^{3/4}\overline\sigma_i \leq \frac{K\tilde c^{3/4}}{\beta^i}.
    \]
    It implies that
    \[
    \sum_{i=j}^n \|u_i\|_{r_n-\sigma_n, b_n-\overline\sigma_n, h_n-\overline\sigma_n,1}
    \leq 
    K\tilde c^{3/4}\sum_{i=j}^n\frac{1}{\beta^i}
    \leq 
    \frac{K\tilde c^{3/4}\beta}{\beta-1}\overline\sigma_j\leq \frac{1}{2} \overline\sigma_j,
    \]
    where the last inequality follows by the assumptions \eqref{eq:tilde_c} on $\tilde c$. The proof of the first assertion is immediate using this result.

    To prove the second assertion, we apply Lemma \ref{lem:cauchy} so that it suffices to show the statement with $r=0$. Indeed, for $n\geq 1$ we can write
    \[
    \varphi^{(n)}-\varphi^{(n-1)}
    =
    \varphi_{n}
    +
    \varphi^{(n-1)}\circ(I+\varphi_{n})-\varphi^{(n-1)}.
    \]
    Applying Proposition \ref{proposition:difference_estimates}, we obtain
    \[
    \|\varphi^{(n)}-\varphi^{(n-1)}\|
    \leq
    \|\varphi_n\|
    +
    K\|\varphi^{(n-1)}\|\|\varphi_n\|
    \left(\frac{1}{\sigma_{n-1}}+\frac{1}{\overline\sigma_{n-1}}\right)
    \leq 
    \varepsilon_n^{3/4}+K\frac{\varepsilon_n^{3/4}}{\sigma_{n-1}\overline\sigma_{n-1}},
    \]
    where 
    $\|\cdot\|$ stands for $\|\cdot\|_{r_n-\sigma_n, b_n-\overline\sigma_n, h_n-\overline\sigma_n,1}$ to simplify. As a result
    \[
    \|\varphi^{(n)}-\varphi^{(n-1)}\|_{\mathscr C^1_{\text{hol}}(K_\infty)} \leq 2(K+1)\frac{\varepsilon_n^{3/4}}{\sigma_{n-1}\overline\sigma_{n-1}},
    \]
    where we have used Proposition \ref{proposition:norm_power_times_map} and the fact that $\sigma_{n-1},\overline\sigma_{n-1}\in(0,1)$.
\end{proof}

It follows from Lemma \ref{lemma:composition_diffeos} that the sequence of maps $\Phi^{(n)}$ converge to a $\mathscr C^{\infty}$-holomorphic map
\[
\Phi:\T_{r_{\infty}}\times K_\infty\to\C/\Z\times \C,
\]
where the notion of convergence is in the $\mathscr C^k$-holomorphic norm on $K_{\infty}$ for any $k\geq0$. Moreover $\Phi$ is a diffeomorphism onto its image satisfying
\[
\|\Phi-I\|_{\mathscr C^1_{\text{hol}}(K_\infty)}\leq \frac{1}{2}
\]
and
\[
T_R\circ\Phi(x,y)=\Phi\circ A(x,y),
\qquad
(x,y)\in\T_{r_\infty}\times K_\infty.
\]
This concludes the proof of Theorem \ref{thm:KAM}.
\end{proof}

\section{KAM Step -- Proof of Proposition \ref{prop:estimates_one_step}}

This section is devoted to the proof of Proposition \ref{prop:estimates_one_step}, which we call the KAM step. The proof relies on solving the system of equations \eqref{equation:multi_difference_equation}, where we find solutions $(\overline \varphi_1, \overline \varphi_2)$ given $(\overline R_1, \overline R_2)$. This is stated precisely in the following proposition, whose proof is postponed to Section \ref{section:proof_cohomological_equation}:

\begin{proposition}
\label{proposition:solution_truncated_difference_equation}
Let $r,b,h,\sigma,\overline\sigma \in(0,1]$ with $r-\sigma,b-\overline\sigma,h-\overline\sigma>0$, $N \in \N_{> 0}$ and $\overline R_1, \overline R_2\in\mathscr C^1_{r,b,h}(\R)$. 
If 
\[
h \leq \tfrac{1}{2\sqrt 2}\gamma N^{-\tau},
\]
there exist $\overline \varphi_1, \overline \varphi_2\in\mathscr C^{1+}_{r-\sigma,b-\overline\sigma,h-\overline\sigma}(\R)$ satisfying \eqref{equation:multi_difference_equation}. Moreover,
 \begin{equation}
 \label{equation:norm_control_difference_equation}
  \max_{i=1,2}
 \|\overline \varphi_i\|^+_{r - \sigma, b - \overline \sigma, h - \overline \sigma} 
 \leq 
 \frac{C_2}{\sigma^{2(\tau+1)}\overline\sigma^2}
 \max_{i=1,2} 
 \|\overline R_i\|_{r, b, h},
 \end{equation}
where $C_2 = C_2(\gamma, \tau) = \tfrac{\overline{C_2}(\tau)}{\gamma^2}$ and $\overline{C_2}(\tau) > 1$ is a constant depending only on $\tau$.
\end{proposition}

Let us show how the Proposition above can be used to prove Proposition \ref{prop:estimates_one_step}.

\begin{proof}[Proof of Proposition \ref{prop:estimates_one_step}]
Let $r,b,h,\sigma,\overline\sigma\in(0,1]$ and $N\in\Z_{>0}$ be as in the statement of Proposition \ref{prop:estimates_one_step}. Let $R\in\mathscr C^1_{r,b,h}$ and consider the associated map $T_R$ as defined at \eqref{eq:perturbation_form} which we assume to be billiard-like. We describe a procedure to construct a diffeomorphism $\Phi=I+\varphi$ that will allow us to define a map $T^+ = \Phi^{-1}T_R\Phi:\Domain{r'}{b'}{h'}\to \C/\Z \times \C$, for some $r'\in(0,r)$, $b'\in(0,b)$ and $h'\in(0,h)$, which (in norm) is much closer to $A$ than the initial map $T_R$. We will 
\begin{enumerate}
    \item construct $\Phi$, its inverse and estimate their norm;
    \item explain in which sense the composition $T_+:=\Phi^{-1}T_R\Phi$ is well-defined and analytic;
    \item give estimates of the $\mathscr C^1$-holomorphic norm of $R_+:=T_+-A$.
\end{enumerate}
\vspace{0.2cm}

\noindent{\it \underline{Construction of $\Phi$ and its inverse.}} Let $\kappa \in (0, 1)$ and  define
\[
c_\ast = \frac{\kappa}{2(K+1)(4C_2+6K)}>0,
\]
\[
C_{\ast} = KC_2(C_2+3K)+4KC_2+2K^3{C_2}^2+ C_3
\]
where $C_2$ is given by Proposition \ref{proposition:solution_truncated_difference_equation}, $C_3$ by Lemma \ref{lem:average_scheme} and $K$ by Proposition \ref{proposition:inverse_estimates}. 

Assume that  $R\in\mathscr C^1_{r,b,h}$ satisfies \eqref{equation:assumption_R}, with the value of $c_*$ given above, that is, $$\|R \|_{r, b, h, 2} \leq c_\ast \sigma^{2\tau+3}\overline\sigma^3.$$

Recall that by our assumptions we have
\[ h \leq \tfrac{1}{2\sqrt 2}\gamma N^{-\tau}, \qquad \sigma\in(0,r/4), \qquad \overline\sigma\in(0,\min(b,h)/4),\]
and
\[ R = (y^2 \overline R_1, y^3 \overline R_2), \qquad \| \overline R\|_{r, b, h} = \|R \|_{r, b, h, 2},\]
where $\overline R_1, \overline R_2$ are $\mathscr C^1$-holomorphic functions.

Proposition \ref{proposition:solution_truncated_difference_equation} applied to $(\overline R_1, \overline R_2)$ provides $\overline \varphi_1, \overline \varphi_2\in\mathscr C^{1+}_{r-\sigma,b-\overline\sigma,h-\overline\sigma}(\R)$ satisfying the difference equation \eqref{equation:multi_difference_equation} together with norm estimates given by \eqref{equation:norm_control_difference_equation}.
We can then define
\[
\varphi = (y\overline \varphi_1,y^2 \overline \varphi_2) \in \mathscr C^{1+}_{r-\sigma,b-\overline\sigma,h-\overline\sigma},
\]
which satisfies \eqref{eq:cohomological_truncated} by construction, and consider the map
$\Phi = I+\varphi: \Domain{r - \sigma}{h-\overline\sigma}{b-\overline\sigma} \to \C/\Z \times \C$.
The norm of $\varphi$ can be controlled using \eqref{equation:norm_control_difference_equation} together with the assumption \eqref{equation:assumption_R} on $R$ by
\begin{equation}
\label{eq:estimates_phi}
\|\varphi\|_{r-\sigma,b-\overline\sigma,h-\overline\sigma,1}^+ \leq \frac{C_2}{\sigma^{2(\tau+1)}\overline\sigma^2}\|R\|_{r,b,h,2}
\leq c_{\ast}C_2\sigma\overline\sigma.
\end{equation}
Notice that since $\sigma,\overline\sigma\in(0,1)$ and by the construction of $c_\ast$, it holds that $\|\varphi\|_{r-\sigma,b-\overline\sigma,h-\overline\sigma,1}^+<1$ and this hold in particular for the $\mathscr C^1$ norm of $\varphi.$ It implies that $\Phi$ is a diffeomorphism onto its image.
We can now apply Proposition \ref{proposition:inverse_estimates}: indeed Assumption \eqref{equation:assumption_inverse} of Proposition \ref{proposition:inverse_estimates} is satisfied since Equation \eqref{eq:estimates_phi} implies
\[
\|\varphi\|_{r-\sigma,b-\overline\sigma,h-\overline\sigma,1}^+ \leq 2c_{\ast}C_2\frac{\sigma\overline\sigma}{\sigma+\overline\sigma}
\]
together with the observation that $2c_\ast C_2\leq \kappa/K$. Hence by Proposition \ref{proposition:inverse_estimates} there exists
\[
\psi\in\mathscr C^{1+}_{r-2\sigma,b-2\overline\sigma,h-2\overline\sigma,1}
\]
such that $\Phi^{-1} = I+\psi$
and whose norm is given by
\begin{equation}
\label{eq:estimates_psi}
\|\psi\|_{r-2\sigma,b-2\overline\sigma,h-2\overline\sigma,1}^+ \leq 
K\|\varphi\|_{r-\sigma,b-\overline\sigma,h-\overline\sigma,1}^+
\leq c_\ast KC_2
\sigma\overline\sigma,
\end{equation}
where the last inequality follows from \eqref{eq:estimates_phi}.

\vspace{0.2cm}

\noindent{\it \underline{The composition $\Phi^{-1}T_R\Phi$.}}
Let us show that the composition $\Phi^{-1}T_R\Phi$ is well-defined and analytic on the set $\T_{r'}\times K^{h'}_{b'}$, where $r'=r - 3\sigma$, $b'=b-3\overline\sigma$ and $h'=h-3\overline\sigma$.

We first show that $T_R\Phi = A \Phi + R \Phi$ is well-defined and analytic, by proving that $R\circ \Phi$ belongs to $\mathscr C^{1+}_{r-2\sigma,h-2\overline\sigma,b-2\overline\sigma,2}$. Indeed, from the estimates \eqref{eq:estimates_phi} on $\|\varphi\|_{r-\sigma,b-\overline\sigma,h-\overline\sigma,1}$ and Proposition \ref{proposition:norm_power_times_map}, we obtain that
\begin{equation}
\label{eq:estimates_phi_2}
\|\varphi\|_{r-\sigma,b-\overline\sigma,h-\overline\sigma}
\leq 
4\|\varphi\|_{r-\sigma,b-\overline\sigma,h-\overline\sigma,1}
\leq
4c_\ast C_2
\sigma\overline\sigma.
\end{equation}
Observe by construction of $c_\ast$ that $4c_\ast C_2\leq\kappa$ and $\sigma\overline\sigma\leq \min(\sigma,\overline\sigma)$. Hence Proposition \ref{proposition:composition_estimates} can be applied to obtain that $R\circ \Phi\in\mathscr C^{1+}_{r-2\sigma,h-2\overline\sigma,b-2\overline\sigma,2}$
together with the estimates
\begin{equation}
    \label{eq:estimates_R_composed_Phi}
\|R\circ\Phi\|_{r-2\sigma,b-2\overline\sigma,h-2\overline\sigma,2}\leq K\|R\|_{r,b,h,2}.
\end{equation}
Finally, to show that $\Phi^{-1}T_R\Phi = T_R  \Phi + \psi T_R \Phi$
is a well-defined analytic map, we prove that $\psi T_R\Phi$ belongs to the space 
$\mathscr C^{1+}_{r-3\sigma,h-3\overline\sigma,b-3\overline\sigma,2}$. To see this, we write it as
\[
\psi T_R\Phi = \psi A(I+\varphi+A^{-1}R\Phi)
\]
Applying Proposition \ref{proposition:norm_power_times_map} with $b\in(0,1)$ and using the expression $A^{-1}(x,y)=(x-y,y)$, we first get
\[
\|\varphi+A^{-1}R\Phi\|_{r - 2\sigma, b - 2\overline \sigma, g - 2\overline \sigma}
\leq 
4\|\varphi\|_{r-\sigma,b-\overline\sigma,h-\overline\sigma,1}
+
6\|R\Phi\|_{r-2\sigma,b-2\overline\sigma,h-2\overline\sigma,2}.
\]
This combined with the estimates \eqref{eq:estimates_phi} on $\varphi$ and the estimates \eqref{eq:estimates_R_composed_Phi} on $R\circ\Phi$ leads to
\begin{equation}
\label{eq:phi_plus_R}
\|\varphi+A^{-1}R\Phi\|_{r - 2\sigma, b - \overline \sigma, h - \overline \sigma}
\leq 
\left(\frac{4C_2}{\sigma^{2(\tau+2)}\overline\sigma^2}
+6K\right)\|R\|_{r,b,h,2}
\leq 
\frac{4C_2+6K}{\sigma^{2(\tau+2)}\overline\sigma^2}
\|R\|_{r,b,h,2}
\end{equation}
where the last inequality comes form the fact that $\sigma,\overline\sigma\in(0,1)$.
Now using the estimates \eqref{equation:assumption_R} on the norm of $R$ and the observation that $(4C_2+6K)c_\ast \leq \kappa$ gives
\begin{equation}
\label{eq:phi_plus_R_bis}
\|\varphi+A^{-1}R\Phi\|_{r - 2\sigma, b - \overline \sigma, h - \overline \sigma}
\leq 
c_\ast(4C_2+6K)
\sigma\overline\sigma
\leq \kappa \min(\sigma,\overline\sigma).
\end{equation}
We can henceforth apply Proposition \ref{proposition:composition_estimates} to ensure that $\psi T\Phi$ is an element of $\mathscr C^{1+}_{r-3\sigma,h-3\overline\sigma,b-3\overline\sigma,1}$. 
As a result the map $T_+ = \Phi^{-1}T_R\Phi = A + A\varphi + R\Phi + \psi T_R\Phi$ can be written as $T_+=A+R_+$ with $R_+\in\mathscr C^1_{r-3\sigma,b-3\overline\sigma,h-3\overline\sigma,1}$.


\noindent{\it \underline{Norm estimates of $R_+$.}} 
Following the ideas of Martin, Ram\'irez-Ros and Sarol in \cite{MartinRamirezTamarit}, we consider the expansion \eqref{eq:expansion_R_plus}
\[
R_+ = R_a^++R_b^++R_c^++R_d^+ +  R_e^+,
\]
where
\[
R_a^+ = R - \mathscr T_NR,
\quad
R_b^+ = \varphi A - \varphi T\Phi,
\quad
R_c^+ = R\Phi - R,
\quad
R_d^+ = (\varphi + \psi)T \Phi,
\quad
R_e^+ = (0, [R_2]).
\]
For each $x\in \{a,b,c,d,e\}$, let us prove that $R^+_x$ lies in $\mathscr C^1_{r-3\sigma,b-3\overline\sigma,h-3\overline\sigma,2}$ and estimate its norm. The estimates are stated in the next Proposition (Proposition \ref{prop:estimates_each_term}) and its proof is described in the next subsections. Combining the different estimates will conclude the proof of Proposition \ref{prop:estimates_one_step}.

\begin{proposition}
\label{prop:estimates_each_term}
The following estimates hold:
    \begin{itemize}
        \item[a.] $R_a^+\in\mathscr C^1_{r-\sigma,b,h,2}$ and
        \[\|R_a^+\|_{r-\sigma,b,h,2} \leq \frac{e^{-2\pi N\sigma}}{\pi\sigma} \|R\|_{r,b,h,2};\]
        
        \item[b.] $R_b^+\in\mathscr C^1_{r-3\sigma,b-3\overline\sigma,h-3\overline\sigma,2}$ and
        \[
         \|R_b^+\|_{r-3\sigma,b-3\overline\sigma,h-3\overline\sigma,2}
         \leq KC_2(C_2+3K)
         \frac{\|R\|_{r,b,h,m}^2}{\sigma^{4\tau+5}\overline\sigma^5};
        \]
        
        \item[c.] $R_c^+\in\mathscr C^1_{r-\sigma,b-\overline\sigma,h-\overline\sigma,3}$ and
        \[
        \|R_c^+\|_{r-\sigma,b-\overline\sigma,h-\overline\sigma,2} 
        \leq
        4KC_2
        \frac{\|R\|_{r,b,h,2}^2}{\sigma^{2\tau+3}\overline\sigma^3};
        \]
        
        \item[d.] $R_d^+\in\mathscr C^1_{r-3\sigma,b-3\overline\sigma,h-3\overline\sigma,2}$ and
        \[
        \|R_d^+\|_{r-3\sigma,b-3\overline\sigma,h-3\overline\sigma,2} 
        \leq 
        2K^3{C_2}^2
        \frac{\|R\|_{r,b,h,2}^2}{\sigma^{4\tau+5}\overline\sigma^5};
        \]
        \item[e.] $R_e^+\in\mathscr C^1_{r - 3\sigma, b - 3 \overline \sigma, h - 3 \overline \sigma, 2}$ and
        \[
        \|R_e^+\|_{r - 3\sigma, b - 3 \overline \sigma, h - 3 \overline \sigma, 2} 
        \leq 
        C_3\frac{e^{-2\pi N \sigma}}{\pi \sigma \overline \sigma^3}\| R \|_{r,b,h,2} + C_3\frac{\| R \|_{r,b,h,2}^2}{\sigma^{\tau + 2} \overline \sigma^5},
        \]
        where $C_3$ is given by Lemma \ref{lem:average_scheme}.
    \end{itemize}
\end{proposition}
This completes the proof of Proposition \ref{prop:estimates_one_step}.
\end{proof}

\subsection{Estimates of $\|R_a^+\|_{r-\sigma,b,h,2}$}
The estimates for $R_a^+$ will follow immediately from Lemma \ref{lemma:truncation_estimates}.

\subsection{Estimates of $\|R_b^+\|_{r-3\sigma,b-3\overline\sigma,h-3\overline\sigma,2}$}
We start by rewriting $R_b^+ = \varphi A-\varphi T\Phi$ as
\[
R_b^+ = -\left(f(I+\varphi+A^{-1}R\Phi)-f\right)
\]
where $f=\varphi\circ A$ is an element of $\mathscr C^1_{r-\sigma,b-\overline\sigma,h-\overline\sigma,1}$ by construction of $\varphi$.
The map $\varphi+A^{-1}R\Phi$ is an element of $\mathscr C^1_{r-2\sigma,b-2\overline\sigma,h-2\overline\sigma,1}$ and
satisfies the assumptions of 
Proposition \ref{proposition:difference_estimates},
as shown in \eqref{eq:phi_plus_R_bis}. The application of the proposition implies that
$R_b^+\in \mathscr C^1_{r-3\sigma,b-3\overline\sigma,h-3\overline\sigma,2}$ and 
\[
\|R_b^+\|_{r-3\sigma,b-3\overline\sigma,h-3\overline\sigma,2}
\leq 
K\|f\|_{r-\sigma,b-\overline\sigma,h-\overline\sigma,1}
\|\varphi+A^{-1}R\Phi\|_{r-2\sigma,b-2\overline\sigma,h-2\overline\sigma,1}
 \left(\frac{1}{\sigma}+\frac{1}{\overline\sigma}\right). 
 \]
This combined with the estimates \eqref{eq:estimates_phi} on $\varphi$ and the estimates \eqref{eq:estimates_R_composed_Phi} on $R\circ\Phi$ leads to
 \[
 \|R_b^+\|_{r-3\sigma,b-3\overline\sigma,h-3\overline\sigma,2}
 \leq
 K\frac{C_2}{\sigma^{2(\tau+1)}\overline\sigma^2}\|R\|_{r,b,h,2} 
 \left(
 \frac{C_2}{\sigma^{2(\tau+1)}\overline\sigma^2}\|R\|_{r,b,h,2}
 +
 3K\|R\|_{r,b,h,2}\right)\left(\frac{1}{\sigma}+\frac{1}{\overline\sigma}\right)
 \]
and since $\sigma,\overline\sigma\in(0,1)$ we obtain
\[
 \|R_b^+\|_{r-3\sigma,b-3\overline\sigma,h-3\overline\sigma,2}
 \leq 
 2KC_2(C_2+3K)
 \frac{\|R\|_{r,b,h,2}^2}{\sigma^{4\tau+5}\overline\sigma^5}.
\]

\subsection{Estimates of $\|R_c^+\|_{r-\sigma,b-\overline\sigma,h-\overline\sigma,3}$}
We recall that 
\[
R_c^+ = R\Phi - R = R(I+\varphi)-R.
\]
By construction $\varphi$ is an element of $\mathscr C^1_{r-\sigma,b-\overline\sigma,h-\overline\sigma,1}$ and, by the estimates given in \eqref{eq:estimates_phi_2}, we can apply Proposition \ref{proposition:difference_estimates} to obtain that 
$R_c^+\in \mathscr C^1_{r-\sigma,b-\overline\sigma,h-\overline\sigma,2}$ with the estimates
\[
\|R_c^+\|_{r-\sigma,b-\overline\sigma,h-\overline\sigma,3} \leq
K\|R\|_{r,b,h,2}
\|\varphi\|_{r-\sigma,b-\overline\sigma,h-\overline\sigma,1} \left(\frac{1}{\sigma}
+\frac{1}{\overline\sigma}\right).
\]
This combined with the estimates \eqref{eq:estimates_phi} on $\varphi$ gives
\[ 
\|R_c^+\|_{r-\sigma,b-\overline\sigma,h-\overline\sigma,3}
\leq 
KC_2
\frac{\|R\|_{r,b,h,2}^2}{\sigma^{2(\tau+1)}\overline\sigma^2}
\left(\frac{1}{\sigma}+\frac{1}{\overline\sigma}\right)
\leq 
2KC_2
\frac{\|R\|_{r,b,h,2}^2}{\sigma^{2\tau+3}\overline\sigma^3}
\]
where in the last inequality we have used $\sigma,\overline\sigma\in(0,1)$.
The expected upper bound on $\|R_c^+\|_{r-2\sigma,b-\overline\sigma,h-\overline\sigma,2}$ can be easily deduced using Proposition \ref{proposition:norm_power_times_map}.

\subsection{Estimates of $\|R_d^+\|_{r-3\sigma,b-3\overline\sigma,h-3\overline\sigma,2}$} Let us rewrite $R_d^+=(\varphi + \psi)T \Phi$ as
\[
R_d^+ = (\varphi+\psi)\circ A(I+\varphi+A^{-1}R\Phi).
\]
By Proposition \ref{proposition:inverse_estimates}, $\phi+\psi\in\mathscr C^{1+}_{r-2\sigma,b-2\overline\sigma,h-2\overline\sigma,2}$ with
 \[
 \|\varphi+\psi\|_{r-2\sigma,b-2\overline\sigma,h-2\overline\sigma,2}^+ 
 \leq 
 K^2({\|\varphi\|_{r-\sigma,b-\overline\sigma,h-\overline\sigma,1}^+})^2
 \left(\frac{1}{\sigma}+\frac{1}{\overline\sigma}\right).
 \]
The map $\varphi+A^{-1}R\Phi$ is an element of $\mathscr C^1_{r-2\sigma,b-2\overline\sigma,h-2\overline\sigma,1}$ and
satisfies the assumptions of 
Proposition \ref{proposition:composition_estimates},
as shown in \eqref{eq:phi_plus_R_bis}.
Hence Proposition \ref{proposition:composition_estimates} implies that $R_d^+$ is an element of $\mathscr C^1_{r-3\sigma,b-3\overline\sigma,h-3\overline\sigma,2}$ and
\[
 \|R_d^+\|_{r-3\sigma,b-3\overline\sigma,h-3\overline\sigma,2} \leq K^3({\|\varphi\|_{r-\sigma,b-\overline\sigma,h-\overline\sigma,1}^+})^2\left(\frac{1}{\sigma}+\frac{1}{\overline\sigma}\right).
\]
Using the estimates on $\varphi$ given in Equation \eqref{eq:estimates_phi} we obtain
\[
\|R_d^+\|_{r-3\sigma,b-3\overline\sigma,h-3\overline\sigma,2} 
\leq K^3{C_2}^2
\frac{\|R\|_{r,b,h,2}^2}{\sigma^{4(\tau+1)}\overline\sigma^4}
\left(
 \frac{1}{\sigma}
 +
 \frac{1}{\overline\sigma}\right)
 \leq 
 2K^3{C_2}^2
\frac{\|R\|_{r,b,h,2}^2}{\sigma^{4\tau+5}\overline\sigma^5}
\]
where in the last inequality we have used $\sigma,\overline\sigma\in(0,1)$.

\subsection{Estimates of $\|R_e^+\|_{r-3\sigma,b-3\overline\sigma,h-3\overline\sigma,2}$}

This is a direct consequence of Lemma \ref{lem:average_scheme}.

\section{Cohomological equation -- Proof of Proposition \ref{proposition:solution_truncated_difference_equation}}
\label{section:proof_cohomological_equation}

The proof of Proposition \ref{proposition:solution_truncated_difference_equation} boils down to solve an Equation of the form
\begin{equation}
\label{equation:difference_equation}
w(x+y,y)-w(x,y) = y\cdot \mathscr T_N v(x,y),
\end{equation}
where $v$ is a given map and $w$ is the unknown, both being $1$-periodic in $x$ and $y$ can vary in a neighborhood of zero. Given an integer $N>0$, $\mathscr T_N v$ refers to the map obtained from $v$ by removing all Fourrier modes in $x$ of rank at least $N+1$, that is 
\[
\mathscr T_N v(x,y) = \sum_{|p|\leq N}\widehat v_p(y)e^{2i\pi px},
\qquad 
\widehat v_p(y) = \int_0^1v(x,y)e^{-2i\pi px}dx.
\]
Equation \eqref{equation:difference_equation} is at the core of the KAM step, and was studied in various forms and different contexts, see for example \cite{Lazutkin_KAM, MartinRamirezTamarit, CMS}. For our purpose, we will solve it in the space of $\mathscr C^1$-holomorphic maps, and we prove the following result:

\begin{proposition}
\label{prop:cohomological_equation}
Let $r,b,h,\sigma,\overline\sigma\in(0,1]$ and $v\in\mathscr C^{1}_{r,b,h}(\RR)$ such that $r-\sigma,b-\overline\sigma,h-\overline\sigma>0$ and $\widehat v_0=0$. Fix an integer $N > 0$ and assume that 
\[
h \leq \tfrac{1}{2\sqrt 2}\gamma N^{-\tau}.
\]
Then there is a unique $w\in\mathscr C^{1+}_{r-\sigma,b-\overline\sigma,h-\overline\sigma}(\RR)$ such that $\widehat w_0=0$ and satisfying Equation \eqref{equation:difference_equation}. Moerover,
\[
\|w\|_{r-\sigma,b-\overline\sigma,h-\overline\sigma}^+
\leq 
\frac{C_1}{\sigma^{\tau+1}\overline\sigma}
\|v\|_{r,b,h},
\]
where $C_1 = C_1(\gamma, \tau) = \tfrac{\overline{C_1}(\tau)}{\gamma}$ 
and 
$\overline{C_1}(\tau) > 1$ 
is a constant depending only on $\tau$.
\end{proposition}

\begin{proof}
One can easily see that the formal solution $w$ of Equation \eqref{equation:difference_equation} is given by
\[
w(x,y) = \sum_{1\leq |n|\leq N}\lambda_n(y)\widehat v_n(y)e_n(x)
\]
where $\lambda_n$ are defined in Equation \eqref{equation:elementary_maps} and $e_n$ is the function defined by
\[
e_n(x) = e^{2i\pi nx},
\qquad x\in\C.
\]
By Proposition \ref{proposition:bound_elementary} and our assumptions on $N$ and $h$, this is a well-defined analytic function on $\T_{r}\times K_{b}^h$. Given $\sigma,\overline\sigma$ as in the statement of Proposition \ref{prop:cohomological_equation},
\[
\|w\|_{r-\sigma,b-\overline\sigma,h-\overline\sigma}
\leq 
\sum_{1\leq |n|\leq N}
\|\lambda_n\|_{\mathscr C^1_{\text{hol}}(K_{h-\overline\sigma}^{b-\overline\sigma}, \CC)}
\|\hat v_n\|_{\mathscr C^1_{\text{hol}}(K_{h}^{b}, \CC)}e^{2\pi|n|r-\sigma}.
\]
Now, by \cite[Corollary 10]{CMS},
\[
\|\widehat v_n\|_{\mathscr C^1_{\text{hol}}(K^b_h,\CC)} 
\leq e^{-2\pi |n|r}\|v\|_{r,b,h}.
\]
Hence, together with Proposition \ref{proposition:bound_elementary}, we obtain 
\[
\|w\|_{r-\sigma,b-\overline\sigma,h-\overline\sigma}
\leq 
\sum_{1\leq |n|\leq N}
C_0\frac{b}{\gamma\overline\sigma}|n|^{\tau}
\cdot
\|v\|_{r,b,h}e^{-2\pi |n|\sigma}
\leq
\frac{2C_0}{\gamma\overline\sigma}\|v\|_{r,b,h}\sum_{n=1}^N n^{\tau} e^{-2\pi n\sigma}
\]
where $C_0$ is given in Proposition \ref{proposition:bound_elementary}.
It follows from Lemma \ref{lemma:classical_sum}, with $\alpha = 2\pi\sigma$ and $\beta = \tau$ 
that
\[
\sum_{n=1}^N n^{\tau} e^{-2\pi|n|\sigma}
\leq 
\frac{C_{2\pi\sigma,\tau}}{(2\pi\sigma)^{\tau+1}} 
\leq 
\frac{C_{2\pi,\tau}}{(2\pi\sigma)^{\tau+1}}.
\]
Hence if we set
\[
\overline{C_1} = \frac{2C_0}{(2\pi)^{\tau+1}}C_{2\pi,\tau}
\qquad
\text{and}
\qquad
C_1 = \frac{\overline{C_1}}{\gamma}
\]
that 
\[
\|w\|_{r-\sigma,b-\overline\sigma,h-\overline\sigma}
\leq 
\frac{C_1}{\sigma^{\tau+1}{\overline \sigma}}\|v\|_{r,b,h}.
\]
Now note that 
\[
w\circ A(x,y) = \sum_{1\leq |n|\leq N}e_n(y)\lambda_n(y)\widehat v_n(y)e_n(x) = -\sum_{1\leq |n|\leq N}\lambda_{-n}(y)\widehat v_n(y)e_n(x)
\]
where we used the equality $e_n(y)\lambda_n(y) = -\lambda_{-n}(y)$ as in \cite{CMS}. Hence we also conclude that
\[
\|w\circ A\|_{r-\sigma,b-\overline\sigma,h-\overline\sigma}
\leq 
\frac{C_1}{\sigma^{\tau+1}{\overline \sigma}}\|v\|_{r,b,h}.
\]
and therefore $w\in\mathscr C^{1+}_{r,b,h}(\RR)$ with the desired estimates.
\end{proof}

Now we can give the proof of Proposition \ref{proposition:solution_truncated_difference_equation}:

\begin{proof}[Proof of Proposition \ref{proposition:solution_truncated_difference_equation}]
Let $r,b,h,\sigma,\overline\sigma\in(0,1]$ as in the statement of Proposition \ref{proposition:solution_truncated_difference_equation}. Let $\sigma' = \sigma/2$ and $\overline\sigma' = \overline\sigma/2$. By Proposition \ref{prop:cohomological_equation}, there exists $w\in\mathscr C^{1+}_{r-\sigma',b-\overline\sigma',h-\overline\sigma'}(\R)$, with $\widehat w_0=0$, satisfying
 \begin{equation}\label{eq:w_eq}
 w(x+y,y)-w(x,y) = y\mathscr T_N(R_2 -[R_2])(x, y).
 \end{equation}
 By construction, it satisfies $\mathscr T_N w = w$, together with the bound 
 \[
 \|w\|_{r-\sigma',b-\overline\sigma',h-\overline\sigma'}^+\leq \frac{C_1}{{\sigma'}^{\tau+1}{\overline\sigma'}}\|R_2-[R_2]\|_{r,b,h},
 \]
 where $C_1$ is given by Proposition \ref{prop:cohomological_equation}. 
 And since 
 $\|[R_2]\|_{\mathscr C_{\text{hol}}^1(K_h^b,\CC)} \leq \|R_2\|_{r,b,h}$ we can state
  \[
 \|w\|_{r-\sigma',b-\overline\sigma',h-\overline\sigma'}^+\leq \frac{2C_1}{{\sigma'}^{\tau+1}{\overline\sigma'}}\|R_2\|_{r,b,h}.
 \]
 
We consider now $\varphi_2=w - [R_1]$ and we notice that $\varphi_2$ also satisfies \eqref{eq:w_eq}.
Moreover, since
 $\|[R_1]\|_{\mathscr C_{\text{hol}}^1(K_h^b,\CC)} \leq \|R_1\|_{r,b,h}$, we have
 \begin{equation}
 \label{equation:intermediate_phi2}
 \begin{aligned}
 \|\varphi_2\|_{r-\sigma',b-\overline\sigma',h-\overline\sigma'}^+ 
    &\leq \frac{2C_1}{{\sigma'}^{\tau+1}{\overline\sigma'}}
    \|R_2\|_{r,b,h}
    +\|R_1\|_{r,b,h} \\  
    & \leq \frac{2^{\tau+3}C_1}{\sigma^{\tau+1}\overline\sigma}\|R\|_{r,b,h},
 \end{aligned}
 \end{equation}
 where we used the fact that $C_1\geq 1$ and $\sigma,\overline\sigma\leq1$. 
 
 Since $[\mathscr T_NR_1+\varphi_2] = 0$ and $\mathscr T_N R_1+\varphi_2 = \mathscr T_N(R_1+\varphi_2)$, by Proposition \ref{prop:cohomological_equation}, there exists $\varphi_1\in\mathscr C^{1+}_{r-2\sigma',b-2\overline\sigma',h-2\overline\sigma'}(\R)$ satisfying
 \[
 \varphi_1(x+y,y)-\varphi_1(x,y) = y(\mathscr T_N R_1(x, y)+\varphi_2(x, y)),
 \]
 and
\[
\begin{aligned}
 \|\varphi_1\|_{r-2\sigma',b-2\overline\sigma',h-2\overline\sigma'}^+ 
 &\leq \frac{C_1}{{\sigma'}^{\tau+1}{\overline\sigma'}}
 \|R_1+\varphi_2\|_{r-2\tilde\sigma,b,h} \\
 & \leq \frac{2^{\tau+2}C_1}{{\sigma}^{\tau+1}{\overline\sigma}}
 \left(\|R_1\|_{r,b,h}+\|\varphi_2\|_{r-\tilde\sigma,b,h}\right)\\
 & \leq \frac{2^{\tau+2}C_1}{{\sigma}^{\tau+1}{\overline\sigma}}
 \|R\|_{r,b,h}
 +\frac{2^{2\tau+5}C_1^2}{{\sigma}^{2(\tau+1)}{\overline\sigma}^2}
 \|R\|_{r,b,h},
 \end{aligned}
 \]
 where the last inequality follows from Equation \eqref{equation:intermediate_phi2}. Now using the facts that
 $C_1\geq 1$, $\sigma,\overline\sigma \leq 1$  
 and setting
 \[
 C_2 = 2^{2\tau+6}C_1^2, 
 \]
 we achieve the inequality
 \[
 \max_{j=1,2}\|\varphi_j\|_{r-2\sigma',b-2\overline\sigma',h-2\overline\sigma'}^+ \leq \frac{C_2}{{\sigma}^{2(\tau+1)}{\overline\sigma}^2}
 \|R\|_{r,b,h}
 \]
 which is the desired estimate since $2\sigma'=\sigma$ and $2\overline\sigma'=\overline\sigma$.
\end{proof}

\section{Small divisors and diamonds}

In this section, and in this section only, we allow $\gamma$ to vary. We will recall in each result which pair $(\gamma,\tau)$ we consider. 

For any integer $n\neq 0$, let the map $\lambda_n$ defined for all $z\in \CC\setminus\frac{1}{n}\ZZ$ by
\begin{equation}
 \label{equation:elementary_maps}
 \lambda_n(z) = \frac{z}{e^{2i\pi nz}-1}.
\end{equation}
This section is devoted to the proof of the following statement, to be compared to \cite[Proposition 6]{CMS} for a different set of definition.

\begin{proposition}
\label{proposition:bound_elementary}
Let $b\in(0,1)$, $N \in \N_{\geq 1}$, $0 < h \leq \tfrac{\gamma}{2\sqrt 2N^{\tau}}$ and $\sigma\in(0,1)$ such that $h-\sigma,b-\sigma>0$. For any integer $n$ such that $|n|\leq N$, the map $\lambda_n$ given by \eqref{equation:elementary_maps} is analytic on the set $K_h^b$, and there exists a constant $C_0 > 0$, not depending on any of the parameters, such that 
\[
\|\lambda_n\|_{\mathscr C^1_{\text{hol}}(K_{h-\sigma}^{b-\sigma},\CC)}
\leq
C_0\frac{b}{\gamma\sigma}|n|^{\tau}.
\]
\end{proposition}
The proof of Proposition \ref{proposition:bound_elementary} is postponed at the end of this section, after stating a few technical results.
\vspace{0.2cm}

The family composed of all the $\lambda_n$'s plays a central role in KAM schemes. Notice that the map $\lambda_n$ is analytic with poles at each rational number of the form $m/n$ with $m\in\ZZ$. Hence the union of all their poles is dense in $\R$. The classical idea of KAM schemes is to estimate the norms of all the $\lambda_n$'s on a common subset of $\R$ made of Diophantine numbers, which in our case is the set $D(\gamma,\tau)$ defined at \eqref{eq:Diophantine_Lazutkin}. This set was introduced in \cite{Lazutkin_KAM} and is suitable for KAM schemes associated to billiard and billiard-like maps. Note that $0$ belongs to this set as an accumulation point.

A central part of this work, to which this section is devoted, consists of estimating the norms of each $\lambda_n$ on a common subspace $\CC\setminus\QQ$ containing $D(\gamma, \tau)$, defined at \eqref{eq:complement_diamonds} and denoted by $K(\gamma,\tau)$. This set was introduced in \cite{CMS} but applied to a different set of Diophantine numbers which avoids a neighborhood of $0$.
\vspace{0.2cm}

\begin{lemma}
\label{lemma:bound_complexification}
There exists a constant $c>0$ such that if $z\in\CC$ and $\omega\in\RR$ satisfy
$|\im z|\geq |\re z-\omega|$
then we have the inequality
$$|e^{2i\pi z}-1|\geq c|\omega-m|$$
where $m\in\ZZ$ is such that $|\omega-m|=\inf_{m'\in\ZZ}|\omega-m'|$.
\end{lemma}

\begin{proof}
Let $c_1>0$ be the infimum of the map $|e^{2i\pi z}-1|/|z|$ for $|z|\leq 1/2$, and $c_2>0$ the infimum of the map $|e^{2i\pi z}-1|$ for $|\im z|\geq 1/2$. Let us now consider a general $z\in\CC$, and $m\in\ZZ$ as in the statement of Lemma \ref{lemma:bound_complexification}. In the case when $\im z|\leq 1/2$, we can estimate
$$|e^{2i\pi z}-1|=|e^{2i\pi (z-m)}-1|\geq c_1|z-m|.$$
Now, if one writes $z=x+iy$ with $x,y\in\RR$, by assumptions we can compute
$$|z-m|^2 = (x-m)^2+y^2\geq (x-m)^2+(x-\omega)^2$$
and the latter expression is minimal when $x=m+\omega$. Thus
$$|z-m|^2\geq \omega^2+m^2\geq |\omega-m|^2$$
and therefore $|e^{2i\pi z}-1|\geq c_1|\omega-m|$. In the case when $|\im z|\geq 1/2$, we can use the fact that $|\omega-m|\leq 1/2$ and write
$$|e^{2i\pi z}-1|\geq c_2 \geq 2c_2|\omega-m|$$
and the result follow with $c=\min (c_1,2c_2)$.
\end{proof}

Now given an integer $N\geq 1$, we define the set 
\[
D(\gamma,\tau,N) := \{\omega\in\RR\,|\, \forall m/n\in\QQ\quad 
|n|\leq N \Rightarrow| n\omega-m|\geq \gamma|m|n^{-\tau}\}
\]
and its complex extension
\[
K(\gamma,\tau,N) = 
\{ z\in\C
\,:\,
\exists\omega\in D(\gamma,\tau, N)\quad |\im z|\geq |\re z-\omega|
\}.
\]
Notice first the inclusions
\[
D(\gamma,\tau)\subseteq D(\gamma,\tau,N)
\quad\text{hence}\quad
K(\gamma,\tau)\subseteq K(\gamma,\tau,N).
\]
Moreover, the construction of $D(\gamma,\tau,N)$ only requires rational numbers with denominators smaller than $N$; in a bounded region of $\R$, there are only fintely many such rational numbers. For this reason it has non-empty interior. The following result states that it contains in fact families of open sets

\begin{lemma}
 \label{lemma:neighborhoods}
 If $\gamma'\in(0,\gamma)$ and $h\leq (\gamma-\gamma')N^{-\tau}/\sqrt 2$, then
 \[
 K_h^b
 \subseteq 
 K(\gamma',\tau,N).
 \]
\end{lemma}

\begin{proof}
    Let $z\in K_h^b$. By construction, there exists $z'\in K$ such that $|z'-z|<h$.
    There exists $\omega'\in D(\gamma,\tau)$ such that 
    \[
    |\re z'-\omega'|\leq |\im z'|.
    \]
    The existence of $z'$ implies that one of the orthogonal projections of $z$ onto the complex lines 
    \[
    L_+=\{\tilde z\,|\, \re\tilde z-\omega = \im\tilde z\}
    \qquad\text{or}\qquad
    L_-=\{\tilde z\,|\, \omega-\re\tilde z = \im\tilde z\}
    \]
    is at distance at most $h$ from $z$. Hence, by intersecting $L_-\cup L_+$ with the horizontal complex line passing through $z$, we find the existence of a point $z''$ on one of the two lines $L_-$ or $L_+$ such that 
    \[
    \im z'' = \im z
    \quad\text{and}\quad 
    |z-z''| \leq \sqrt 2 h.
    \]
    In particular $z''\in K$, and we set $\omega = \omega'+z-z''\in\R$.
    Note first that
    \[
    |\re z-\omega| = |\re z'-\omega'|\leq|\im z'|=|\im z|,
    \]
    hence it is enough to show that $\omega\in D(\gamma',\tau,N)$. To prove it, consider $m/n\in\Q$ with $1\leq n\leq N$. Then using the triangular inequality as well as the construction of $\omega'$ and of $z''$
    \[
    \left|\omega-\tfrac{m}{n}\right|
    \geq
    \left|\omega'-\tfrac{m}{n}\right|-|z''-z|
    \geq 
    \gamma|m|n^{-\tau}-\sqrt 2h.
    \]
    Using the assumption on $h$ we obtain
    \[
    \left|\omega-\tfrac{m}{n}\right|
    \geq 
    \gamma|m|n^{-\tau}-(\gamma-\gamma')N^{-\tau}\geq \gamma'|m|n^{-\tau}.
    \]
    Hence $\omega\in D(\gamma',\tau,N)$ and the result follows.
\end{proof}

We can now prove Proposition \ref{proposition:bound_elementary}.

\begin{proof}[Proof of Proposition \ref{proposition:bound_elementary}]
    By Lemma \ref{lem:c0_control}, it is enough to show that
    \[
    |\lambda_n|_{K_h^b}\leq C\frac{b}{\gamma}n^{\tau}.
    \]
    for some constat $C>0$ independant of any of the parameters.
    To prove the latter inequality, introduce the analytic function $\psi:\C\setminus\Z\to\C$ defined by 
    \[
        \psi(z) = \frac{z}{e^{2i\pi z}-1},
        \qquad z\in\C\setminus\Z
    \]
    and note that 
    \[
    \lambda_n(z) = \frac{1}{n}\psi(nz).
    \]
    Hence if we consider the set,
    \[
    S = \{z\in\C
    \,|\,
    \exists\omega\in[-\tfrac{1}{2},\tfrac{1}{2}]\}
    \quad 
    |\re z-\omega|\leq|\im z|\},
    \]
    there is a constant $C>0$ such that 
    \begin{equation}
    \label{equation:bound_psi}
    |\psi(z)|\leq C\max\{ 1, |z|\},
    \qquad z\in S.
    \end{equation}
    This is the result of the two following estimates, obtained when $z\to 0$ and when $|z|\to+\infty$ respectively:
    \[
    \psi(z) = \mathcal O(1),\quad \text{as }z\to0;
    \]
    \[
    |\psi(z)| = \mathcal O(|z|),\quad \text{as }|z|\to+\infty.
    \]
    Hence if $z\in K_h^b$, there are two possibilities. The first one is when $nz\in S$. In this situation, \ref{equation:bound_psi} implies that 
    \[
    |\psi(z)|\leq C|z|\leq Cb.
    \]
    In the other case when $nz\notin S$, by Lemma \ref{lemma:neighborhoods} applied with $\gamma'=\gamma/2$, there is $\omega\in D(\gamma/2,\tau,N)$ such that
    \[
    |\re z-\omega|\leq|\im z|. 
    \]
    This still holds when we replace $z$ and $\omega$ by $nz$ and $n\omega$ respectively. Now by Lemma \ref{lemma:bound_complexification} there is
    a constant $c>0$ independant of any of the parameters such that 
    \[
    |e^{2i\pi n z}-1|\geq c|n\omega-m|
    \]
    where $m$ is an integer such that $|n\omega-m| = \inf_{m'\in\Z}|n\omega-m'|$. Note that $m$ cannot be zero, or else $n\omega\in[-\tfrac{1}{2},\tfrac{1}{2}]$ which implies that $nz\in S$, but this is not the case.
    Hence since $\omega\in D(\gamma/2,\tau,N)$
    \[
    |e^{2i\pi n z}-1|\geq c\gamma|m|n^{-\tau}\geq c\gamma n^{-\tau}.
    \]
    We conclude in the case when $nz\notin S$ that
    \[
    |\lambda_n(z)|\leq \frac{n^{\tau}}{c}|z|\leq \frac{bn^{\tau}}{c\gamma}
    \]
    and the result follows.
\end{proof}

\newpage
\appendix

\section{Four Billiard Models}
\label{sec:four_billiards}

In this appendix, we recall four closely related planar billiard models frequently studied in the literature:
Birkhoff billiards, outer billiards, symplectic billiards, outer length billiards
(see, e.g., \cite{BBS}). 
Each model is defined relative to a strictly (or strongly) convex planar domain $\Omega$ with smooth boundary $\partial\Omega$. They are distinguished by whether their dynamics occur in the interior or exterior of $\Omega$, and whether the action of their periodic orbits relates to perimeter or area.

\subsection{Birkhoff Billiards (Inner Length Billiards)}

\begin{figure}[ht!]
    \centering
\definecolor{qqwuqq}{rgb}{0,0.39215686274509803,0}
\definecolor{ududff}{rgb}{0.30196078431372547,0.30196078431372547,1}
\definecolor{xdxdff}{rgb}{0.49019607843137253,0.49019607843137253,1}
\begin{tikzpicture}[line cap=round,line join=round,x=2cm,y=2cm]
\clip(-1.6,-1) rectangle (1.6,1);
\draw [shift={(1.0111502084967459,-0.5140416064766641)},line width=0.3pt,color=qqwuqq,fill=qqwuqq,fill opacity=0.10000000149011612] (0,0) -- (178.1457448821165:0.07730871859020828) arc (178.1457448821165:218.4525105608136:0.07730871859020828) -- cycle;
\draw [shift={(1.0111502084967459,-0.5140416064766641)},line width=0.3pt,color=qqwuqq,fill=qqwuqq,fill opacity=0.10000000149011612] (0,0) -- (38.452510560813586:0.07730871859020828) arc (38.452510560813586:78.77179632737457:0.07730871859020828) -- cycle;
\draw [shift={(-1.0881474802175675,-0.4460786068454301)},line width=0.3pt,color=qqwuqq,fill=qqwuqq,fill opacity=0.10000000149011612] (0,0) -- (-44.55971630969508:0.07730871859020828) arc (-44.55971630969508:-1.8542551178835047:0.07730871859020828) -- cycle;
\draw [rotate around={0:(0,0)},line width=1pt] (0,0) ellipse (2.58996693cm and 1.64557853cm);
\draw [line width=0.5pt] (-1.0881474802175675,-0.4460786068454301)-- (1.0111502084967459,-0.5140416064766641);
\draw [-latex,line width=0.5pt] (1.0111502084967459,-0.5140416064766641) -- (1.0941310259743455,-0.09603787080411325);
\begin{scriptsize}
\draw[color=black] (-0.5,0.6) node {$\Omega$};
\draw[color=black] (-1.2,-0.5) node {$s$};
\draw[color=black] (-0.9,-0.51) node {$\theta$};
\draw[color=black] (1.05,-0.58) node {$s'$};
\draw[color=black] (1.12,-0.3) node {$\theta'$};
\draw[color=black] (0,-0.37) node {$L(s,s')$};
\end{scriptsize}
\end{tikzpicture}
    \caption{one reflection of the Birkhoff billiard map inside the domain $\Omega$.}
    \label{fig:placeholder}
\end{figure}
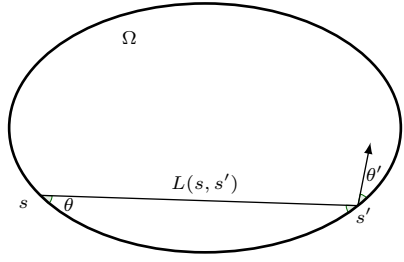
A point particle moves at constant speed inside the domain $\Omega$ along straight lines and undergoes elastic reflections at the boundary $\partial\Omega$ according to the standard optical law: the angle of incidence equals the angle of reflection.

\medskip
\noindent\textbf{Standard Coordinates and Generating Function.} 
Let $\partial\Omega$ be parametrized by arc length $s \in [0, |\partial\Omega|]$, and let $\theta \in (0, \pi)$ denote the angle between the incoming trajectory and the positive tangent line at the collision point. In the standard coordinates $(s, -\cos \theta)$, the phase space is the cylinder $\mathbb{R}/|\partial\Omega|\mathbb{Z} \times (-1, 1)$. The billiard map is a positive twist map with the generating function 
\[
L(s, s') = -|\gamma(s') - \gamma(s)|,
\]
where $\gamma(s)$ parametrizes the boundary $\partial\Omega$.

\medskip
\noindent\textbf{Lazutkin Coordinates \cite{Lazutkin_book}.} 
Near the boundary $\partial\Omega$ (corresponding to glancing collisions where $\theta \to 0$), Lazutkin introduced non-linear coordinate transformations $(x, y)$ that simplify the billiard map into an integrable map up to higher-order terms:
\[
x = \frac{\int_0^s \varrho(t)^{2/3} \, dt}{\int_0^{|\partial\Omega|} \varrho(t)^{2/3} \, dt}, \quad y = C \varrho(s)^{1/3} \sin\left(\frac{\theta}{2}\right),
\]
where $\varrho(s)$ denotes the curvature of $\partial\Omega$ at $\gamma(s)$ and $C > 0$ is a normalizing constant. In these Lazutkin coordinates, the billiard map takes the following form
\[
(x, y) \mapsto (x + y + \mathcal{O}(y^3), y + \mathcal{O}(y^4)),
\]
and is exact-symplectic for the one-form $\lambda=y^2dx$. Hence the Birkhoff billiard map is a billiard-like map of order $3$, see Definition \ref{def:billiard-like}.

\subsection{Outer Billiards (Outer Area Billiards)}
The dynamics takes place in the exterior $\mathbb{R}^2 \setminus \Omega$. Given a point $x_0$ outside $\Omega$, one draws the unique line passing through $x_0$ that is tangent to $\partial\Omega$ at a point $q$, such that the vector $\vec{x_0 q}$ agrees with the orientation of $\partial\Omega$. The image point $x_1$ lies on the same tangent line such that $q$ is the midpoint of the segment $[x_0, x_1]$.

\medskip
\noindent\textbf{Coordinates \cite{Douady}.} 
In envelope coordinates $(\vartheta, \lambda) \in \mathbb{R}/2\pi\mathbb{Z} \times (0, \infty)$, a point $M$ exterior to $\Omega$ is represented as $M = \gamma(\vartheta) + \lambda \frac{\dot{\gamma}(\vartheta)}{|\dot{\gamma}(\vartheta)|}$, where $\vartheta$ denotes the direction of the right normal to $\partial\Omega$. Alternatively, using symplectic polar coordinates $(\phi, r)$, outer billiards also admit Lazutkin-type boundary coordinates near $\partial\Omega$, bringing the map into a billiard-like map.

\subsection{Symplectic Billiards (Inner Area Billiards)}

A particle moves inside $\Omega$ following non-local geometric reflections, see Figure \ref{figure:symplectic_bounce}: for three successive impact points $q_1, q_2, q_3 \in \partial\Omega$, the chord joining $q_1$ and $q_3$ is parallel to the tangent line $T_{q_2}\partial\Omega$.

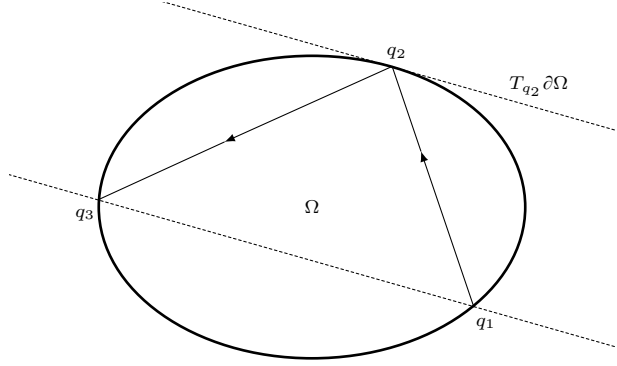
\begin{figure}[!ht]
\centering
\begin{tikzpicture}[line cap=round,line join=round,x=2.0cm,y=2.0cm]
\clip(-2,-1.2) rectangle (2,1.35);
\draw [line width= 1pt, rotate around={0:(0,0)}] (0,0) ellipse (2.82cm and 2cm);
\draw [dash pattern=on 1pt off 1pt,domain=-3.19:3.33] plot(\x,{(--32-8.43*\x)/29.69});
\draw [dash pattern=on 1pt off 1pt,domain=-3.19:3.33] plot(\x,{(-10.52-8.43*\x)/29.69});
\draw [-latex] (1.07,-0.66) -- (0.72,0.37);
\draw [-latex] (0.53,0.93) -- (-0.58,0.43);
\draw (-0.58,0.43)-- (-1.41,0.05);
\draw (0.72,0.37)-- (0.53,0.93);
\begin{scriptsize}
\draw[color=black] (1.15,-0.77) node {$q_1$};
\draw[color=black] (0.56,1.02) node {$q_2$};
\draw[color=black] (-1.5,-0.05) node {$q_3$};
\draw[color=black] (1.5,0.8) node {$T_{q_2}\partial\Omega$};
\draw[color=black] (0,0) node {$\Omega$};
\end{scriptsize}
\end{tikzpicture}
\caption{A bounce of the symplectic billiard dynamics at $q_2$ in a domain $\Omega$. The line $q_1q_3$ and the tangent line to $\partial\Omega$ at $q_2$ are parallel.}
\label{figure:symplectic_bounce}
\end{figure}

\medskip
\noindent\textbf{Coordinates \cite{AlbersTabachnikov,FierobeSorrentinoVig}.} 
We parametrize $\partial\Omega$ by affine curvature $t\mapsto\gamma(t)$. If $q_j = \gamma(t_j)$, the phase space consists of pairs $(t_0,\varepsilon_0)$ where $\varepsilon_0=t_1-t_0$. In these coordinates, the billiard map $T$ is an exact-symplectic twist map which can be expressed as
\[
(t, \varepsilon) \mapsto (t + \varepsilon, t + \mathcal{O}(\varepsilon^4)).
\]
Moreover, as it follows from \cite[Lemma 2.7]{AlbersTabachnikov}, $T$ is exact symplectic for a one-form of the type
\[
\lambda = s(t,\varepsilon)dt
\]
where $s$ is a smooth map admitting the asymptotic expansion $s(t,\varepsilon) = \mathcal O(\varepsilon^2)$ uniformly in $t$ as $\varepsilon\to 0$. This is not exactly the definition of billiard-like map (see Definition \ref{def:billiard-like}), as the coefficient of the $\varepsilon^2$-term in the expansion of $s$ may depend on $t$. Since the latter coefficient doesn't vanish, this can be solved by normalizing $s$ in the coordinate $t$, making $T$ a billiard-like map of order $3$.

\subsection{Outer Length Billiards ($4^{\text{th}}$ Billiards)}

The dynamics takes place in the exterior of $\Omega$, see Figure \ref{fig:outer_length_billiard_reflection}. We parametrize $\partial\Omega$ by arc-length $s\mapsto\gamma(s)$. Given a point $P$ outside $\Omega$, consider the two tangent lines from $P$ to $\partial\Omega$ touching at $\gamma(s_0)$ and $\gamma(s_1)$. A unique circle is drawn tangent to $\partial\Omega$ at $\gamma(s_1)$ and to the tangent line at $\gamma(s_0)$. The image point $P'$ is defined as the intersection of the tangent line at $\gamma(s_1)$ with the second line tangent to this circle and $\partial\Omega$.

\medskip
\noindent \textbf{Coordinates \cite{BBF}.} 
As for the symplectic billiard, the phase space is parametrized by pairs $(s_0, \varepsilon_0)$ where $\varepsilon_0 = s_1 - s_0$. In this coordinates, the billiard map is exact-symplectic and admits the following expansion as $\varepsilon \to 0$
\[
(s, \varepsilon) \mapsto (s + \varepsilon, s + \mathcal{O}(\varepsilon^4)).
\]
Moreover, as it follows from \cite[Section 3]{BBF}, $T$ is exact symplectic for a one-form of the type
\[
\lambda = h(s,\varepsilon)dt
\]
where $h$ is a smooth map admitting the asymptotic expansion $h(s,\varepsilon) = \mathcal O(\varepsilon^2)$ uniformly in $t$ as $\varepsilon\to 0$. This is not exactly the definition of billiard-like map (see Definition \ref{def:billiard-like}), as the coefficient of the $\varepsilon^2$-term in the expansion of $h$ may depend on $s$. Since the latter coefficient doesn't vanish, this can be solved by normalizing $h$ in changing the coordinate $s$, making $T$ a billiard-like map of order $3$.

\begin{figure}[ht!]
    \centering
\begin{tikzpicture}[line cap=round,line join=round,x=2cm,y=2cm]
\clip(-1.8,-0.9) rectangle (1.8,1.6);
\draw [line width=1pt, rotate around={0:(0,0)}] (0,0) ellipse (2.5cm and 1.5cm);
\draw [domain=-2.4644525360396767:2.925337292636209] plot(\x,{(--14.0625--10.87381925368578*\x)/4.808139752941082});
\draw [domain=-2.4644525360396767:2.925337292636209] plot(\x,{(--14.0625--4.077092715913812*\x)/17.475364801934});
\draw [domain=-2.4644525360396767:2.925337292636209] plot(\x,{(--14.0625-7.712987683825162*\x)/13.649615479846691});
\draw [dash pattern=on 1pt off 1pt] (-0.5304631168743307,0.9873081843929444) circle (0.5976cm);
\begin{scriptsize}
\draw[color=black] (1,0.6221857002178954) node {$\gamma(s_0)$};
\draw[color=black] (-0.42393158262452424,0.6) node {$\gamma(s_1)$};
\draw[color=black] (-1.38,0.26685360315769346) node {$\gamma(s_2)$};
\draw[color=black] (0.3113595885199015,0.938818261954709) node {$P$};
\draw[color=black] (-1.1,0.6292219793676024) node {$P'$};
\draw[color=black] (0,-0.2) node {$\partial\Omega$};
\end{scriptsize}
\end{tikzpicture}
    \caption{The outer-length billiard map around the domain $\Omega$. It associates the point $P$ to the point $P'$.}
    \label{fig:outer_length_billiard_reflection}
\end{figure}
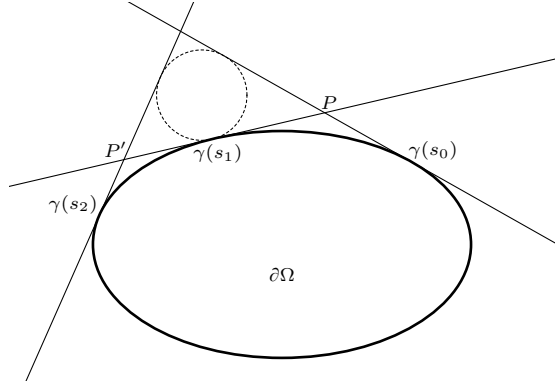

\section{Technical results}

\subsection{An infinite sum}

A proof of the following lemma can be found in the proof of \cite[Proposition 12]{CMS}.
\begin{lemma}
\label{lemma:classical_sum}
Given $\alpha > 0$ and $\beta > 1,$
$$\Sigma_{\alpha, \beta} := \sum_{n>0}n^\beta e^{-\alpha n} \leq \frac{C_{\alpha,\beta}}{\alpha ^{\beta+1}},$$
where $C_{\alpha,\beta} = \beta!+\alpha(e^{-1}\beta)^\beta>0$.
\end{lemma}

\subsection{Change of power}

\begin{proposition}
\label{proposition:norm_power_times_map}
Let $b>0$, $K\subset B(0,b)$ be a compact set, $B$ be a Banach space, an integer $s\geq 1$, and $\varphi:K\to B$ be a $\mathscr C^1$-holomorphic map. Then
$$\|y^s\varphi\|_{\mathscr C^1_{\text{hol}}(K,B)}\leq 2sb^{s-1}\|\varphi\|_{\mathscr C^1_{\text{hol}}(K,B)}.$$
\end{proposition}

\begin{proof}
 By multiplicativity, 
 \[
 \|y^s\varphi\|_{\mathscr C^1_{\text{hol}}(K,B)}\leq \|y^s\|_{\mathscr C^1_{\text{hol}}(K,\CC)}\|\varphi\|_{\mathscr C^1_{\text{hol}}(K,B)}.
 \]
 Hence it is enough to estimate $\|y^s\|_{\mathscr C^1_{\text{hol}}(K,\CC)}$. We first consider the following inequality obtained by restriction
 \[
 \|y^s\|_{\mathscr C^1_{\text{hol}}(K,\CC)} \leq \|y^s\|_{\mathscr C^1_{\text{hol}}(\overline B(0,b),\CC)}.
 \]
 The sup-norm of $f$ defined for any $y$ by $f(y) = y^s$ on $\overline B(0,b)$ is $n_0(f)=|f|_{\overline B(0,b)} = b^s$. The computation is the same for the derivative, $|f'|_{\overline B(0,b)}\leq sb^{s-1}$. Moreover, for any $y_0,y_1\in\overline{B}(0,b)$,
 \[
 |\delta_f(y_0,y_1)|\leq 2b^{s}
 \]
 which implies that $n_1(f)\leq 2sb^{s-1}$. Finally, if $y_0\neq y_1$, observe that 
 \[
 \Omega_{f,f'}(y_0,y_1) = \sum_{k=0}^{q-1}y_0^ky_1^{s-1-k}-sy_0^{s-1}
 \]
 and it follows from the triangular inequality that $|\Omega_{f,f'}(y_0,y_1)|\leq 2sb^{s-1}$. Hence the result follows.
\end{proof}

\subsection{Truncation estimates}

Let $K\subset \C$ be a compact set without isolated points, $r>0$, an integer $k>0$ and a $\mathscr C^1$-holomorphic map $f:K\to\mathscr C^{\omega}(\T_r,\C^k)$.
Given an integer $N>0$, define $\mathscr T_N f$ by the formula
\[
 \mathscr T_N f(x,y) = \sum_{|p|\leq N} \widehat f_p(y)e_p(x)
 \]
 where $e_p$ is the function defined by 
 \[
 e_p(x) = e^{2i\pi p x},
 \qquad x\in\C.
 \]

\begin{lemma}
\label{lemma:truncation_estimates}
Let $\sigma\in(0,1)$ such that $r-\sigma>0$. Then
\[
\|\mathscr T_N f\|_{\mathscr C^1_{\text{hol}}(K,\mathscr C^{\omega}(\T_{r-\sigma},\C^k))}
\leq
\frac{\pi + 1}{\pi\sigma}\|f\|_{\mathscr C^1_{\text{hol}}(K,\mathscr C^{\omega}(\T_r,\C^k))},
\]
\[ 
\|f - \mathscr T_N f\|_{\mathscr C^1_{\text{hol}}(K,\mathscr C^{\omega}(\T_{r-\sigma},\C^k))}
\leq
\frac{e^{-2\pi N\sigma}}{\pi\sigma} \|f\|_{\mathscr C^1_{\text{hol}}(K,\mathscr C^{\omega}(\T_r,\C^k))}. 
\]
In particular, if 
$N \geq \frac{\left| \log \varepsilon \right|}{2\pi\sigma}$ 
for some $\varepsilon\in(0,1)$, then
\[ 
\|f - \mathscr T_N f\|_{\mathscr C^1_{\text{hol}}(K,\mathscr C^{\omega}(\T_{r-\sigma},\C^k))}
\leq
\frac{\varepsilon }{\pi\sigma}\|f\|_{\mathscr C^1_{\text{hol}}(K,\mathscr C^{\omega}(\T_r,\C^k))}. 
\]
\end{lemma}

\begin{proof}
We have
\[
\|\mathscr T_N f\|_{\mathscr C^1_{\text{hol}}(K,\mathscr C^{\omega}(\T_{r-\sigma},\C^k))} 
\leq 
\sum_{|p|\leq N} \|\widehat f_p\|_{\mathscr C^1_{\text{hol}}(K,\C^k)}e^{2\pi|p|},
\]
\[
\|f - \mathscr T_N f\|_{\mathscr C^1_{\text{hol}}(K,\mathscr C^{\omega}(\T_{r-\sigma},\C^k))} 
\leq 
\sum_{|p|> N} \|\widehat f_p\|_{\mathscr C^1_{\text{hol}}(K,\CC^k)}e^{2\pi|p|}.
\]
By \cite[Corollary 10]{CMS},
\[
\|\widehat f_p\|_{\mathscr C^1_{\text{hol}}(K,\C^k)} 
\leq 
e^{-2\pi |p|r}\|f\|_{\mathscr C^1_{\text{hol}}(K,\mathscr C^{\omega}(\T_{r},\C^k))},
\]
for any $p \in \Z$. Hence
\begin{align*}
\|\mathscr T_N f\|_{\mathscr C^1_{\text{hol}}(K,\mathscr C^{\omega}(\T_{r-\sigma},\C^k))} 
&\leq 
\|f\|_{\mathscr C^1_{\text{hol}}(K,\mathscr C^{\omega}(\T_{r},\C^k))}
\sum_{|p|\leq N} e^{-2\pi|p|\sigma} \\
& \leq 
\|f\|_{\mathscr C^1_{\text{hol}}(K,\mathscr C^{\omega}(\T_{r},\C^k))} 
\left( 1+2\sum_{p>0}e^{-2\pi p\sigma} \right) \\ 
& \leq \|f\|_{\mathscr C^1_{\text{hol}}(K,\mathscr C^{\omega}(\T_{r},\C^k))} 
\left( 1 + \frac{1}{\pi\sigma}\right).
\end{align*}

Similarly,
\begin{align*}
\|f - \mathscr T_N f\|_{\mathscr C^1_{\text{hol}}(K,\mathscr C^{\omega}(\T_{r-\sigma},\C^k))} 
& \leq 
\|f\|_{\mathscr C^1_{\text{hol}}(K,\mathscr C^{\omega}(\T_{r},\C^k))}
\sum_{|p| > N} e^{-2\pi|p|\sigma} \\ 
& \leq 
\|f\|_{\mathscr C^1_{\text{hol}}(K,\mathscr C^{\omega}(\T_{r},\C^k))}
\frac{e^{-2\pi N\sigma}}{\pi\sigma}.
\end{align*}
The last assertion in the lemma follows directly if one replaces $N$ by $\tfrac{\left| \log \varepsilon \right|}{2\pi\sigma}$ in the RHS of the previous inequality.
\end{proof}

\subsection{Composition estimates}

\begin{proposition}
\label{proposition:composition_estimates}
Let $m\in\Z_{\geq 1}$, $f\in\mathscr C^1_{r,b,h,m}$ and $\varphi=(\varphi_1,\varphi_2):\T_{r-\sigma}\times K_{b-\overline\sigma}^{h-\overline\sigma}\to\C^2$ analytic such that 
\begin{equation}
\label{eq:size_compo}
\max\left\{\|\varphi_1\|_{\mathscr C_{\text{hol}}^1(K_{h-\overline\sigma}^{b-\overline\sigma})}, \|\varphi_2\|_{\mathscr C_{\text{hol}}^1(K_{h-\overline\sigma}^{b-\overline\sigma})} 
\right\}
\leq \kappa\min(\sigma,\overline\sigma)
\end{equation}
where $\kappa,\sigma,\overline\sigma\in(0,1)$ satisfy $r-\sigma,h-\overline\sigma,b-\overline\sigma>0$. 

Then the maps $g,g_x,g_y:\T_{r-\sigma}\times K_{b-\overline\sigma}^{h-\overline\sigma}\to \C^2$ defined by $g=f\circ(I+\varphi)$, $g_x=\partial_xf\circ(I+\varphi)$ and $g_y=\partial_yf\circ(I+\varphi)$ are respectively elements of $\mathscr C^1_{r-\sigma,b-\overline\sigma,h-\overline\sigma,m}$, $\mathscr C^1_{r-\sigma,b-\overline\sigma,h-\overline\sigma,m}$ and $\mathscr C^1_{r-\sigma,b-\overline\sigma,h-\overline\sigma,m-1}$ whose norms satisfy
\begin{equation}
 \label{equation:inequality_composition}
 \|g\|_{r-\sigma,b-\overline\sigma,h-\overline\sigma,m} \leq K_1\|f\|_{r,b,h}
\end{equation}
\[
\|g_x\|_{r-\sigma,b-\overline\sigma,h-\overline\sigma,m}\leq \frac{K_1}{\sigma}\|f\|_{r,b,h}
\]
and
\[
\|g_y\|_{r-\sigma,b-\overline\sigma,h-\overline\sigma,m-1}\leq \frac{K_1}{\overline\sigma}\|f\|_{r,b,h}
\]
where $K_1>0$ is a constant depending uniquely on $m$ and $\kappa$.
\end{proposition}

\begin{proof}
Assume first that $f:\T_{r}\times K_{b}^{h}\to \C$ is a complex-valued map and define $g=(\partial_x^p\partial_y^q f)\circ(I+\varphi)$ for integers $p,q\geq 0$. 
Set $r'=r-\sigma$, $h'=h-\overline\sigma$ and $b'=b-\overline\sigma$.
Note that $g$ is well-defined by Condition \eqref{eq:size_compo}.
Moreover the following Taylor expansion holds for $g$:
\[
g(x,y) = \sum_{u,v\geq 0}\frac{1}{u!v!}\partial_x^{p+u}\partial_y^{q+v} f(x,y)\varphi_1^u\varphi_2^v.
\]
Applying Cauchy estimates for the $\mathscr C^1$-holomorphic norm -- see CMS Corollary 10 which can be easily extended to the $y$ variable -- the following expansion holds:
\[
\left\|\partial_x^{p+u}\partial_y^{q+v} f\right\|_{\mathscr C_{\text{hol}}^1(K_{h'}^{b'},\mathscr C^{\omega}(\T_{r'},\C))} \leq \frac{(p+u)!(q+v)!}{\sigma^{p+u}\overline\sigma^{q+v}}\|f\|_{\mathscr C_{\text{hol}}^1(K_h^b,\mathscr C^{\omega}(\T_{r},\C))}
\]
hence by multipicativity of the $\mathscr C^1$-holomorphic norm -- see CMS Lemma 5,
\[
\left\|g\right\|_{\mathscr C_{\text{hol}}^1(K_{h'}^{b'},\mathscr C^{\omega}(\T_{r'},\C))}
\leq
\|f\|_{\mathscr C_{\text{hol}}^1(K_{h}^{b},\mathscr C^{\omega}(\T_{r},\C))}\sum_{u,v\geq 0}\frac{1}{u!v!}
\frac{(p+u)!(q+v)!}{\sigma^{p+u}\overline\sigma^{q+v}}
\|\varphi_1\|_{\mathscr C_{\text{hol}}^1(K_{h'}^{b'})}^u\|\varphi_2\|_{\mathscr C_{\text{hol}}^1(K_{h'}^{b'})}^v.
\]
Using assumption \eqref{eq:size_compo} we obtain
\begin{equation}
\label{equation:Cauchy_estimates}
\|g\|_{\mathscr C_{\text{hol}}^1(K_{h'}^{b'},\mathscr C^{\omega}(\T_{r'},\C))}\leq \frac{C_{p+1}C_{q+1}}{\sigma^p\overline\sigma^q}\|f\|_{\mathscr C_{\text{hol}}^1(K_{h}^{b},\mathscr C^{\omega}(\T_{r},\C))}
\end{equation}
where we denoted for any integer $j\geq 0$
\[
C_{j+1} = j!\sum_{u\geq 0}\binom{j+u}{u}\kappa^u = \frac{j!}{(1-\kappa)^{j+1}}.
\]
Now if 
$f(x,y) = (y^mf_1(x,y),y^{m+1}f_2(x,y))$, the composition $f(I+\varphi)$ is given by
\begin{equation}
\label{equation:differentiate_}
f(I+\varphi)(x,y) = \left(y^m\tilde f_1(x,y),y^{m+1}\tilde f_2(x,y)\right)
\end{equation}
where
\[
\tilde f_1(x,y) = (1+\varphi_2(x,y))^mf_1(I+\varphi)(x,y)
\]
and
\[
\tilde f_2(x,y) = (1+\varphi_2(x,y))^{m+1}f_2(I+\varphi)(x,y).
\]
Inequality \eqref{equation:Cauchy_estimates} applied to $\tilde f_1$ and $\tilde f_2$ ensure the result for $f(I+\varphi)$. The result holds in a similar way for $\partial_xf(I+\varphi)$. For $\partial_yf(I+\varphi)$, similar computations give
\begin{equation}
\label{equation:differentiate_y}
 \partial_yf(I+\varphi)(x,y) = \left(y^{m-1}\tilde f_1(x,y),y^{m}\tilde f_2(x,y)\right)
\end{equation}
where
\[
\tilde f_1(x,y) = Y^{m-1}(mf_1(I+\varphi)(x,y)+Y\partial_yf_1(I+\varphi)(x,y))
\]
\[
\tilde f_2(x,y) = Y^{m}((m+1)f_2(I+\varphi)(x,y)+Y\partial_yf_2(I+\varphi)(x,y)),
\]
and $Y = 1+\varphi_2(x,y)$. The result follows again from Inequality \eqref{equation:Cauchy_estimates}.
\end{proof}

\begin{proposition}
\label{proposition:difference_estimates}
Let $m,n\in\Z_{\geq0}$, $f\in\mathscr C^1_{r,b,h,m}$ and $\varphi\in\mathscr C^1_{r-\sigma,b-\overline\sigma,h-\overline\sigma,n}$ such that 
\begin{equation}
\label{eq:size_compo_2}
\|\varphi\|_{r-\sigma,b-\overline\sigma,h-\overline\sigma,n} \leq \kappa\min(\sigma,\overline\sigma)
\end{equation}
where $\kappa,\sigma,\overline\sigma\in(0,1)$ satisfy $r-\sigma,h-\overline\sigma,b-\overline\sigma>0$.  

Then the map $h:\T_{r'}\times K_{b'}^{h'}\to \C^2$ defined by
$h=f(I+\varphi)-f$ is an element of $\mathscr C^1_{r-\sigma,b-\overline\sigma,h-\overline\sigma,m+n}$ whose norm satisfies
$$\|h\|_{r-\sigma,b-\overline\sigma,h-\overline\sigma,m+n}\leq K_2\|f\|_{r,b,h,m}\|\varphi\|_{r-\sigma,b-\overline\sigma,h-\overline\sigma,n}\left(
\frac{1}{\sigma}
+\frac{1}{\overline\sigma}\right)$$
where $K_2>0$ is a constant depending only on $m$, $b$ and $\kappa$.
\end{proposition}

\begin{proof}
The map $h$ can be written in the following integral form
 $$h(x,y) = \int_0^1 df(I+t\varphi)\cdot\varphi\, dt.$$
 Now for a fixed $t\in[0,1]$, the differential $g(t,x,y)=df(I+t\varphi)\cdot\varphi$ inside the integral can written as
 \[
g(t,x,y) = 
y^n\varphi_1\partial_xf(I+t\varphi)+y^{n+1}\varphi_2\partial_yf(I+t\varphi) 
= (y^{m+n}g_1(t,x,y),y^{m+n+1}g_2(t,x,y))
\]
where
\[
g_1(t,x,y) = \varphi_1\partial_xf_1(I+t\varphi)+\varphi_2(mf_1(I+t\varphi)+y\partial_yf_1(I+t\varphi))
\]
and
\[
g_2(t,x,y) = \varphi_1\partial_xf_2(I+t\varphi)+\varphi_2((m+1)f_2(I+t\varphi)+y\partial_yf_2(I+t\varphi)).
\]
Therefore $g(t,\cdot)\in \mathscr C^1_{r-\sigma,b-\overline\sigma,h-\overline\sigma,m+n}$ and 
\[
\|g(t,\cdot)\|_{r-\sigma,b-\overline\sigma,h-\overline\sigma,m+n}\leq K_2\|f\|_{r,b,h,m}\|\varphi\|_{r-\sigma,b-\overline\sigma,h-\overline\sigma,n}\left(
\frac{1}{\sigma}
+\frac{1}{\overline\sigma}\right)
\]
where $K_2>0$ is a constant depending only on $m$ and $\kappa$. The result follows by integrating previous inequality in $t$.
\end{proof}

\subsection{Inverse estimates}

\begin{proposition}
\label{proposition:inverse_estimates}
 Let $\Phi$ be a diffeomorphism of the form $\Phi=I+\varphi$ where $\varphi\in\mathscr C^{1+}_{r,b,h,m}$ is a $\mathscr C^1$-holomorphic map. Let $\kappa,\sigma,\overline\sigma\in(0,1)$ satisfy $r-\sigma,h-\overline\sigma,b-\overline\sigma>0$ and assume that 
 \begin{equation}
 \label{equation:assumption_inverse}
 \|\varphi\|_{r,b,h,m}^+
 \leq
 \frac{\kappa}
 {K}\frac{\sigma\overline\sigma}{\sigma+\overline\sigma}
 \end{equation}
 where $K$ is the largest of the constants $K_1$ and $K_2$ given in Propositions \ref{proposition:composition_estimates} and \ref{proposition:difference_estimates}. Then we can write $\Phi^{-1} = I+\psi$ where 
 $\psi\in\mathscr C^{1+}_{r-\sigma,b-\overline\sigma,h-\overline\sigma,m}$
 satifies the estimates
 \begin{equation}
 \label{equation:norm_inverse}
 \|\psi\|_{r-\sigma,b-\overline\sigma,h-\overline\sigma,m}^+\leq K\|\varphi\|_{r,b,h,m}^+
 \end{equation}
 Moreover $\varphi+\psi\in\mathscr C^{1+}_{r-\sigma,b-\overline\sigma,h-\overline\sigma,2m}$ and it holds that
 \begin{equation}
 \label{equation:norm_inverse_plus_diffeo}
 \|\varphi+\psi\|_{r-\sigma,b-\overline\sigma,h-\overline\sigma,2m}^+ \leq K^2({\|\varphi\|^+})^2\left(\frac{1}{\sigma}+\frac{1}{\overline\sigma}\right).
 \end{equation}
\end{proposition}

\begin{proof}
 A similar result can be found in \cite{MartinRamirezTamarit}. The proof relies on the following observation: the map $\Psi=I+\psi$ is an inverse of $\Phi$ if and only if $\Phi\circ \Psi=I$ which is equivalent to say that $\psi$ is a fixed point of the map $F$ defined by
 $$F(h) = -\varphi\circ(I+h).$$
 Write $r'=r-\sigma$, $b'=b-\overline\sigma$ and $h'=h-\overline\sigma$.  Let $\mathcal B$ be the closed ball of radius 
 \[
\kappa\frac{\sigma\overline\sigma}{\sigma+\overline\sigma}
 \]
 in the space $\mathscr C^{1+}_{r',b',h',m}$. Note that 
  \[
\kappa\frac{\sigma\overline\sigma}{\sigma+\overline\sigma}
\leq \kappa\min(\sigma,\overline{\sigma}) ;
 \]
 hence
 given $h\in\mathcal B$, by Proposition \ref{proposition:composition_estimates}, $F(h)\in \mathscr C^{1}_{r',b',h',m}$ is a well-defined holomorphic map which satisfies
 \begin{equation}
 \label{eq:bound_operator_inverse}
 \|F(h)\|_{r',b',h',m}\leq K_1\|\varphi\|_{r,b,h,m}\leq \kappa\frac{\sigma\overline\sigma}{\sigma+\overline\sigma}
 \end{equation}
 where the last inequality follows from the assumption \eqref{equation:assumption_inverse} on the norm of $\varphi$.
 Now we can apply the same argument to $h\circ A$ and obtain the same conclusion in the space $\mathscr C^{1}_{r',b',h',m}$.
 Hence $F(h)\in\mathcal B$ and $F$ is a well-defined functionnal on the complete metric space $\mathcal B$.
 
 Let us now show that $F$ is a contracting map. Given $h_1,h_2\in\mathcal B$, we can write 
 $$F(h_2)-F(h_1) = -\int_0^1 df\circ (I+h_1+t(h_2-h_1))\cdot(h_1-h_1)dt.$$
 By Proposition \ref{proposition:composition_estimates}, given $t\in [0,1]$, the map 
 $g_t = d\varphi\circ (I+h_1+t(h_2-h_1))\cdot(h_1-h_1)\in \mathscr C^1_{r',b',h',2m}$
 satisfies
 \[
 \|g_t\|_{r',b',h',2m} 
 \leq 
 K_2\|\varphi\|_{r,b,h,m}
 \left(\frac{1}{\sigma}+\frac{1}{\overline\sigma}\right)
 \|h_2-h_1\|_{r',b',h',m}.
 \]
 Again from \eqref{equation:assumption_inverse}, one has
 \[
 K_2\|\varphi\|_{r,b,h,m}
 \left(\frac{1}{\sigma}+\frac{1}{\overline\sigma}\right)
 \leq K_2\frac{\kappa}{K}\frac{\sigma\overline\sigma}{\sigma+\overline\sigma}\left(\frac{1}{\sigma}+\frac{1}{\overline\sigma}\right)
 \leq \kappa,
 \]
 which shows that $F$ is contracting.
 By Banach's fixed point theorem, we deduce the existence of $\psi$. Estimates on its norm follow from the property $\psi=F(\psi)$ and Proposition \ref{proposition:composition_estimates}.
 
 Finally, notice that $\varphi+\psi=F(\psi)-F(0)$, hence by Proposition \ref{proposition:difference_estimates}, $\varphi+\psi\in\mathscr C^1_{r',b',h',2m}$ and the following estimates hold
 \[
 \|\varphi+\psi\|_{r',b',h',2m} \leq K_2\|\varphi\|_{r,b,h,m}\|\psi\|_{r',b',h',m}
 \left(\frac{1}{\sigma}+\frac{1}{\overline\sigma}\right)
 \]
 which together with Inequality \eqref{equation:norm_inverse} finish the proof.
\end{proof}

\subsection{Higher derivatives}

Let us point out that, similarly to the analytic functions setting, estimating the $\mathscr C^1$-holomorphic norm of a function on a $h$-neighbourhood of its domain allows one to estimate the $\mathscr C^1$-holomorphic norm of its derivatives. More precisely, we have the following lemma:

\begin{lemma}
\label{lem:cauchy}
Let $f \in \mathscr C_{\text{hol}}^1(\overline{K_h},B)$ for some closed set $K \subseteq \C$ without isolated points and some $0 < h < 1$. Then,
\[ \|f^{(r)}\|_{\mathscr C^1_{\text{hol}}(K,B)} \leq \frac{C_r}{h^r} \|f\|_{\mathscr C^1_{\text{hol}}(\overline{K_h},B)},\]
for any $r \geq 1$, where $C_r > 1$ is a constant depending only on $r$. Moreover,
\[
\|f\|_{\mathscr C^1_{\text{hol}}(K,B)}\leq \frac{5}{h}\sup_{\overline{K_h}}|f|.
\]
\end{lemma}
\begin{proof}
Fix $r \geq 1$. Notice that $f^{(r)}, f^{(r + 1)}: K_h \to B$ are analytic. By Cauchy's differentiation formula, there exists $C > 1$, depending only on $r$, such that 
\begin{equation}
\label{eq:generalized_cauchy}
\begin{aligned}
|f^{(r)}|_{K} \leq |f^{(r)}|_{K_\delta} \leq \frac{C}{(h - \delta)^r} | f|_{K_h}, \\
|f^{(r + 1)}|_{K} \leq |f^{(r + 1)}|_{K_\delta} \leq \frac{C}{(h - \delta)^r} | f'|_{K_h},
\end{aligned}
\end{equation}
for any $0 \leq \delta < h$.

Notice that for any $z, z' \in K$, we have the following dichotomy. Either $|z - z'| < \tfrac{h}{2}$ and
\begin{align*} |\Omega_{f^{(r)}, f^{(r + 1)}}(z, z')| & \leq \left| \frac{f^{(r)}(z) - f^{(r)}(z')}{z - z'} \right| + |f^{(r + 1)}(z)|\\& \leq |f^{(r + 1)}(z'')| + |f^{(r + 1)}(z)|, 
\end{align*}
for some $z'' \in [z, z'] \subseteq K_{h/2}$, where $[z, z']$ denotes the segment joining $z$ and $z'$. Or, $|z - z'| \geq \tfrac{h}{2}$ and
\begin{align*} |\Omega_{f^{(r)}, f^{(r + 1)}}(z, z')| & \leq \left| \frac{f^{(r)}(z) - f^{(r)}(z')}{z - z'} \right| + |f^{(r + 1)}(z)| \\ & \leq 2\frac{|f^{(r)}(z)| + |f^{(r)}(z')|}{h} + |f^{(r + 1)}(z)|. 
\end{align*}

Hence, by the previous equations and \eqref{eq:generalized_cauchy} with $\delta = \tfrac{h}{2},$ there exists $C' > 1$, depending only on $r$, such that
\[ | \Omega_{f^{(r)}, f^{(r + 1)}}|_{K \times K} \leq \frac{C'}{h^r} | f'|_{K_h}.\]

Therefore,
\begin{align*}
\|f^{(r)}\|_{\mathscr C^1_{\text{hol}}(K)} &= |f^{(r)}|_K + \max \left\{ |f^{(r + 1)}|_K, |\delta_{f^{(r)}}|_{K \times K} \right\} + | \Omega_{f^{(r)}, f^{(r + 1)}}|_{K \times K} \\
& \leq |f^{(r)}|_K + \max \left\{ \tfrac{C}{h^r}|f'|_{K_h}, 2|f^{(r)}|_{K} \right\} + \tfrac{C'}{h^r} | f'|_{K_h} \\
& \leq \tfrac{3C}{h^r}|f|_{K_h} +\tfrac{C + C'}{h^r} | f^{(r)}|_{K_h} \\
& \leq \tfrac{C_r}{h} \|f\|_{\mathscr C^1_{\text{hol}}(\overline{K_h})},
\end{align*}
where $C_r := 4C + C'$.
\end{proof}

\begin{remark}
If the geometry of the set $K$ is not too complicated, for example, if any two points in $K$ can be connected by a piecewise smooth path of length bounded by $M$, it is possible to obtain bounds of the form $\|f^{(r)}\|_{\mathscr C_{\text{hol}}^1(K,B)} \leq \frac{C(r, M)}{h^{r + 1}} |f|_{K_h}$, for any $f \in \mathscr C_{\text{hol}}^1(\overline{K_h},B)$ and any $r \geq 0$, where $C(r, M)$ is a constant depending only on $r$ and $M$. 

Similar bounds (with constants depending also on $K$) can be obtained if $K$ is the union of finitely many connected components having the property mentioned above.
\end{remark}

In fact, we also have the following result

\begin{lemma}
\label{lem:c0_control}
Let $f \in \mathscr C_{\text{hol}}^1(\overline{K_\sigma},B)$ for some closed set $K \subseteq \C$ without isolated points and some $\sigma\in(0,1)$. Then, there exists a constant $C>0$ independant of any of the parameters, such that
\[
\|f\|_{\mathscr C^1_{\text{hol}}(K,B)}\leq \frac{C}{\sigma}\sup_{\overline{K_\sigma}}|f|.
\]
\end{lemma}

\begin{proof}
    By Cauchy estimates, there exists a constant $C'>0$ such that 
    \begin{equation}
    \label{eq:Cauchy_estimates}
    |f'|_{K}\leq \frac{C'}{\sigma}|f|_{\overline{K_\sigma}}.
    \end{equation}
    To estimate the $\mathscr C^1$-holomorphic norm of $f$ on $K$, we need to compare $n_1(f)$ and $n_2(f)$ (on $K$) to $|f|_{\overline{K_\sigma}}$.
    Applying \eqref{eq:Cauchy_estimates} and the fact that $|\delta_f|\leq 2|f|$, we obtain first
    \[
    n_1(f) \leq \frac{C''}{\sigma}|f|_{\overline{K_\sigma}}
    \]
    for some constant $C''>0$ independant of the parameters.
    Now given $z,z'\in K$, either $|z-z'|\leq \sigma/2$, and in this situation
    \[
    |\Omega_{f,f'}(z,z')|\leq\left|\frac{f(z')-f(z)}{z'-z}\right|+|f'(z)| \leq |f'(z'')|+|f(z)|
    \]
    for some $z''\in[z,z']\subset \overline{K_{\sigma/2}}$, which by \eqref{eq:Cauchy_estimates} implies
    \[
    |\Omega_{f,f'}(z,z')|\leq 2\frac{C'}{\sigma}|f|_{\overline{K_\sigma}}.
    \]
    If $|z'-z|>\sigma/2$, then 
    \[
    |\Omega_{f,f'}(z,z')|\leq \frac{2}{\sigma}(|f(z')|+|f(z)|)+|f'(z)|,
    \]
    and the result follows in the same way.
\end{proof}

\section{Average estimates of billiard-like maps}

\begin{lemma}
\label{lem:average_scheme}\footnote{Not super important, but if we have a $C^0$ version of Proposition \ref{prop:cohomological_equation} then the bounds get a bit better and we do not need to use holomorphic norms in the proof. I say this not really for the bounds but because in the proof we end up using the holomorphic norm to get estimates on the $C^0$ norm only because we want to apply Proposition \ref{prop:cohomological_equation} which is stated with the holomorphic norms... so the proof looks a bit weird.}
  Let $r,b, h, \sigma, \overline \sigma\in(0,1)$ and $N \in \N$ satisfying
  \[0 < 3\overline \sigma <  b, h < r < 1, \qquad 0 < \overline \sigma < \sigma < \tfrac{r}{3}, \qquad h \leq \frac{1}{2\sqrt{2}}\gamma N^{-\tau}.\]
  Let $T: \Domain{r}{b}{h} \to \DomainP{r + \overline\sigma}{b + \overline\sigma}{h + \overline\sigma}$ 
  be an analytic billiard-like map of order $m \geq 2$ of the form $T = A + R$, with $R: \Domain{r}{b}{h} \to B(0, \overline\sigma) \times B(0, \overline\sigma)$ given by
    \[ R(x, y) =  (R_1(x, y), R_2(x , y)) = (y^2 \overline R_1(x, y), y^3 \overline R_2(x, y)),\]
    and with the associated 1-form
    \[\lambda(x, y) = (y^2 + \mu(x, y))dx + \nu(x, y)dy = (y^2 + y^2\overline{\mu}(x, y))dx + y\overline \nu(x, y)dy,\] 
    defined on $\DomainP{r + \overline\sigma}{b + \overline\sigma}{h + \overline\sigma}$, satisfying
            \begin{equation}
    \label{eq:smallness_form}
      |\overline \mu|_{\DomainP{r + \overline\sigma}{b + \overline\sigma}{h + \overline\sigma}}, |\overline \nu|_{\DomainP{r + \overline\sigma}{b + \overline\sigma}{h + \overline\sigma}} < \frac{1}{4},
    \end{equation}
    and $T^*\lambda - \lambda = dS,$ for some real-analytic map $S: \Domain{r}{b}{h} \to \C$.

    Then there exists $C_3 > 0$, depending only on $\gamma$ and $\tau$, such that
    \[ \| (0, [R_2])\|_{r, b - 3\overline\sigma, h - 3\overline\sigma, 2} =  \| [\overline{R}_2] \|_{r, b - 3\overline\sigma, h - 3\overline\sigma} \leq C_3\frac{e^{-2\pi N \sigma}}{\pi \sigma \overline \sigma^3}\| R \|_{r,b,h,2} + C_3\frac{\| R \|_{r,b,h,2}^2}{\sigma^{\tau + 2} \overline \sigma^5}.\]

\end{lemma}
\begin{proof}

    A direct calculation shows that
    \begin{align*}
        \partial_x S dx + \partial_y dS & =  T^*\lambda - \lambda \\
         & = (y^2 + 2yR_2 + R^2_2 + \mu\circ T)(dx + dy + \partial_xR_1dx + \partial_yR_1dy)  \\ 
         & \, \quad + \nu\circ T (dy + \partial_xR_2dx + \partial_y R_1 dy)  - (y^2 + \mu)dx - \nu dy.
    \end{align*}
Hence, adding and substracting $( (\mu \circ A) \partial_x R_1 + (\nu \circ A) \partial_x R_2)dx$ in the equation above,we can express $\partial_x S$ as
\begin{equation}
\label{eq:x_component_expression_1}
\begin{aligned}
\partial_x S &  =  2yR_2  + (\mu \circ A) \partial_x R_1 +  (\nu \circ A) \partial_x R_2 + \mu \circ T - \mu \circ A\\
& \quad + y^2 \partial_xR_1 + \mu \circ A - \mu\\
& \quad + (2yR_2 \partial_x R_1 + R^2_2(1 + \partial_x R_1)) + (\mu \circ T - \mu \circ A) \partial_x R_1 +  (\nu \circ T - \nu \circ A) \partial_x R_2.
\end{aligned}
\end{equation}
Notice that the second line in the expression above has zero average on $\T$, for any $y \in K_h^b$ fixed, and that the third line is `quadratic' on $R$. 

Let $h(t, x, y)= \mu(x + y + tR_1(x, y), y + tR_2(x, y))$. Then, by Taylor's theorem, 
\[ \mu \circ T = \mu \circ A + (\partial_x\mu \circ A)R_1 + (\partial_y\mu \circ A)R_2 + \int_0^1 \partial_t^2h(t, \cdot,\cdot) (1 -t)dt. \]
Hence
\begin{align*}
     (\mu \circ A) & \partial_x R_1 +  (\nu \circ A) \partial_x R_2 + \mu \circ T - \mu \circ A \\ & = \partial_x((\mu \circ A)R_1 + (\nu \circ A)R_2) - (\partial_x\mu \circ A)R_1 - (\partial_x\nu \circ A)R_2 +  \mu \circ T - \mu \circ A \\
     & = \partial_x((\mu \circ A)R_1 + (\nu \circ A)R_2) + (\partial_y\mu \circ A - \partial_x\nu \circ A)R_2 + \int_0^1 \partial_t^2h(t, \cdot,\cdot) (1 -t)dt.
\end{align*}

Plugging the equation above back into \eqref{eq:x_component_expression_1}, we can express $R_2$ as
\begin{equation}
\label{eq:x_component_expression_2}
2y R_2 =  Z + (\partial_y\mu \circ A - \partial_x\nu \circ A)R_2 + Q,
\end{equation}
where $Z$ is a function of zero average on $\T$, for any $y \in K_h^b$ fixed, and $Q$ depends `quadratically' on $R$. More precisely, $Z$ and $Q$ can be expressed as
and
\begin{align*} Z & = y^2 \partial_xR_1 + \mu \circ A - \mu + \partial_x((\mu \circ A)R_1 + (\nu \circ A)R_2),
\end{align*}
\begin{equation}
\label{eq:Q_equation}
\begin{aligned}
Q & =(2yR_2 \partial_x R_1 + R^2_2(1 + \partial_x R_1)) + (\mu \circ T - \mu \circ A) \partial_x R_1 +  (\nu \circ T - \nu \circ A) \partial_x R_2 \\ & \quad + \int_0^1 \partial_t^2h(t, \cdot,\cdot) (1 -t)dt.
\end{aligned}
\end{equation}
Notice that $[\mu \circ A - \mu]$ is indeed equal to $0$ on $K_h^b$ since, for any $y \in K_h^b \cap \R$, we have
\[ [\mu \circ A](y) = \int_\T \mu(x + y, y)dx = \int_\T \mu(x, y)dx = [\mu](y),\]
and $K_h^b$ is connected. Denoting 
 \[f := \frac{1}{y} ( \partial_y \mu -  \partial_x \nu) = 2 \overline \mu + y\partial_y \overline \mu -  \partial_x \overline \nu,\] we can rewrite \eqref{eq:x_component_expression_2}  as
 
 \begin{equation*}
y\left(2 - [f]  \right)R_2 = Z + y(f \circ A - [f \circ A])R_2 + Q .
\end{equation*}
Hence, taking averages and dividing by $y^4$ in the equation above, we obtain
\begin{equation}
\label{eq:equality_average_R2}
    \left(2 - [f] \right)\left[\overline{R}_2\right] = \left[ (f \circ A - [f \circ A])\overline{R}_2  \right] + \left[ \frac{Q}{y^4} \right].
\end{equation}
 Notice that, 
 \begin{equation}
 \label{eq:decays}
     \nu = \bigo{y}, \qquad R_1, \mu = \bigo{y^2}, \qquad R_2 = \bigo{y^3},
 \end{equation}
 and thus the term $\tfrac{Q}{y^4}$ in the equation above is a well-defined analytic function on $\T_r \times K_h^b$. 
 
 By \eqref{eq:smallness_form}, we have $|[f]|_{K_{h + \overline \sigma}^{b + \overline\sigma}} < 1$. Therefore, Equation \eqref{eq:equality_average_R2} yields
 \begin{equation}
\label{eq:R2_average}
\left|\left[\overline{R}_2\right] \right|_{U} \leq  \left| f - [f]\right|_{A(\T \times U)} \left| \overline{R}_2 \right|_{\T \times U} +  \left| \left[ \frac{Q}{y^4} \right]\right|_{U}, \qquad \text{ for any } U \subseteq K_h^b.
\end{equation}

\begin{claim}
\label{cl:bounds_Q}
    There exists $C_0 > 0$, independent of all parameters, such that 
 \begin{equation*}
 \label{eq:bound_quadratic_terms_initial}
     \left| \left[ \frac{Q}{y^4} \right]\right|_{K_{h - \overline\sigma}^{b - \overline\sigma}} \leq  \frac{C_0}{\overline\sigma^2}\big\|R\big\|_{r, b, h,2}^2.
 \end{equation*} 
\end{claim}
\begin{proof}[Proof of Claim \ref{cl:bounds_Q}]
    This follows directly by Equation \eqref{eq:Q_equation} together with \eqref{eq:smallness_form} and Cauchy's estimates.
\end{proof}

Now let us consider $f - [f]$. Since $T$ preserves the 2-form $d\lambda = (\partial_x \nu - \partial_y \mu - 2y)dx \wedge dy = -y(f + 2)dx\wedge dy$, a direct calculation yields
            \begin{align*}
        d\lambda & = T^*d\lambda  \\
        & = (\partial_x\nu \circ T -  \partial_y\mu \circ T -2y -2 yR_2)((1  +  \partial_xR_1)dx + (1 + \partial_yR_1)dy)\wedge(\partial_xR_2dx + (1 + \partial_yR_2)dy)\\
        & = (\partial_x\nu \circ T -  \partial_y\mu \circ T -2y -2 yR_2)((1  +  \partial_xR_1)(1 + \partial_yR_2) - (1 + \partial_yR_1)\partial_xR_2)dx\wedge dy,
    \end{align*}
and thus 
\begin{equation*}
    \label{eq:equality_difference}
    -yf  = 2y + (\partial_x\nu \circ T -  \partial_y\mu \circ T -2y -2 yR_2)((1  +  \partial_xR_1)(1 + \partial_yR_2) - (1 + \partial_yR_1)\partial_xR_2).
\end{equation*}
Adding $yf \circ A = \partial_y \mu \circ A - \partial_x \nu \circ A$ to the equation above, we obtain
\begin{align*}
       y(f \circ A - f) &=  (\partial_x\nu \circ T - \partial_x\nu \circ A) -  (\partial_y\mu \circ T - \partial_y\mu \circ A)  -2 yR_2 \\
      & \quad + (\partial_x\nu \circ T -  \partial_y\mu \circ T -2y -2 yR_2)(  \partial_xR_1 +\partial_yR_2 + \partial_xR_1 \partial_yR_2- (1 + \partial_yR_1)\partial_xR_2)
    \end{align*}

 Notice that, by \eqref{eq:decays} and since $\partial_x\nu \circ T - \partial_x\nu \circ A = \int_0^1 \langle \nabla (\partial_x \nu) \circ (A + tR), R \rangle dt= \mathcal{O}(y^2)$ (resp. $\partial_y\mu \circ T - \partial_y\mu \circ A = \mathcal{O}(y^2)$), the RHS of the equation above divided by $y^2$, which we will denote by $\xi(x, y)$, is a well-defined analytic function on $\Domain{r}{b}{h}$. Thus we can rewrite the equation above as
\begin{equation}
\label{eq:difference_equation_mu_nu}
    f \circ A - f = y \xi.
\end{equation}
By \eqref{eq:smallness_form} and applying Cauchy's estimates to estimate the norm of $\xi$, it follows that 
 \begin{equation}
 \label{eq:bound_xi}
   \|\xi\|_{r - \overline\sigma, b - \overline\sigma, h - \overline \sigma} \leq  \frac{C_1}{\overline \sigma^3}\| R \|_{r,b,h,2}
 \end{equation} 
 for some constant $C_1 > 0$,  independent of all parameters. 
 
 \begin{claim}
 \label{cl:f_minus_average_scheme}
There exists $C_2 > 0$,  depending only on $\gamma$ and $\tau$, such that
\[\left\|f - [f] \right\|_{r - 2\sigma - \overline\sigma, b - 2\overline\sigma, h - 2\overline\sigma}^+ \leq C_2\frac{e^{-2\pi N \sigma}}{\pi \sigma \overline \sigma^2} + C_2\frac{\| R \|_{r,b,h,2}}{\sigma^{\tau + 2} \overline \sigma^3},\]
 \end{claim}
\begin{proof}[Proof of Claim \ref{cl:f_minus_average_scheme}] 
Applying the truncation operator $\mathcal{T}_N$ to $f \circ A - f = y \xi$, it follows that
\[ (\mathcal{T}_Nf) \circ A - \mathcal{T}_Nf = y \mathcal{T}_N\xi. \]
Notice that, by \eqref{eq:smallness_form} and Cauchy's estimates, we have
\[ \| f \|_{r, b, h}^+ \leq \frac{C_3}{\overline \sigma^2}, \]
for some constant $C_3$, independent of all parameters. Moreover, by Lemma \ref{lemma:truncation_estimates} together with Equations \eqref{eq:smallness_form} and \eqref{eq:bound_xi}, 
\[ \|\mathcal{T}_N f - f \|_{r - \sigma, b, h}^+ \leq C_3\frac{e^{-2\pi N \sigma}}{\pi \sigma \overline \sigma^2}, 
\qquad \|\mathcal{T}_N \xi \|_{r - \sigma - \overline\sigma, b - \overline\sigma, h - \overline\sigma} \leq \frac{2C_1}{\sigma \overline \sigma^3}\| R \|_{r,b,h,2}. \]

By Proposition \ref{prop:cohomological_equation}, there exists a constant $C_4 > 0$, depending only on $\gamma$ and $\tau$, such that  
\[ \|\mathcal T_Nf - \mathcal T_N[f]\|_{r - 2\sigma - \overline\sigma, b- 2\overline\sigma, h - 2\overline\sigma}^+ \leq C_4\frac{\| \mathcal T_N\xi\|_{r - \sigma - \overline\sigma, b-\sigma}}{\sigma^{\tau + 1}\overline \sigma} \leq 2C_1C_4\frac{\| R \|_{r,b,h,2}}{\sigma^{\tau + 2} \overline \sigma^4}. \]

Therefore,
\[\left\|f - [f] \right\|_{r - 2\sigma - \overline\sigma, b- 2\overline\sigma, h - 2\overline\sigma}^+ \leq 2C_3\frac{e^{-2\pi N \sigma}}{\pi \sigma \overline \sigma^2} + 2C_1C_4\frac{\| R \|_{r,b,h,2}}{\sigma^{\tau + 2} \overline \sigma^4}. \]
\end{proof}

We can now conclude the proof of the lemma. By Equation \eqref{eq:R2_average} applied to $U = K^{b - 2\overline \sigma}_{h - 2\overline \sigma}$, the two Claims above, and noticing that $A(\T \times K^{b - 2\overline \sigma}_{h - 2\overline \sigma}) \subseteq \DomainP{r - 2\sigma - \overline\sigma}{b - 2\overline \sigma}{h - 2\overline \sigma}$, it follows that
\[ \left|\left[\overline{R}_2\right] \right|_{K_{b - 2\overline \sigma}^{h - 2\overline\sigma}} \leq 2C_3\frac{e^{-2\pi N \sigma}}{\pi \sigma \overline \sigma^2}\| R \|_{r,b,h,2} + 2C_1C_4\frac{\| R \|_{r,b,h,2}^2}{\sigma^{\tau + 2} \overline \sigma^4}  +   \frac{C_0}{\overline\sigma^2}\big\|R\big\|_{r, b, h,2}^2.\]
The lemma now follows by Lemma \ref{lem:c0_control}.
\end{proof}

\end{document}